\pdfoutput=1

\documentclass{amsart}

\usepackage[margin=1 in]{geometry}

\usepackage{verbatim}
\usepackage{xcolor}
\usepackage{tikz}
\usepackage{tikz-cd}
\usepackage{arydshln}
\usepackage{caption}
\usepackage{subcaption}
\usepackage{graphicx}
\usepackage[hyphens]{url}
\usepackage{amsmath}
\usepackage{bbm}
\usepackage{hyperref}
\usepackage{mathrsfs}

\newtheorem{theorem}{Theorem}[section]
\newtheorem{proposition}[theorem]{Proposition}
\newtheorem{lemma}[theorem]{Lemma}
\newtheorem{corollary}[theorem]{Corollary}

\theoremstyle{definition}

\theoremstyle{remark}
\newtheorem{remark}[theorem]{Remark}

\numberwithin{equation}{section}

\newcommand{\mcg}{\mathrm{Mod}_g}
\newcommand{\mc}{\mathbf{g}}
\renewcommand{\tt}{\mathcal{T}_g}
\newcommand{\mm}{\mathcal{M}_g}
\newcommand{\qt}{\mathcal{Q}\mathcal{T}_g}

\newcommand{\qut}{\mathcal{Q}^1\mathcal{T}_g}
\newcommand{\qutp}{\mathcal{Q}^1\mathcal{T}_g(\mathbf{1})}
\newcommand{\qum}{\mathcal{Q}^1\mathcal{M}_g}

\newcommand{\mf}{\mathcal{MF}_g}
\newcommand{\pmf}{\mathcal{PMF}_g}

\newcommand{\PC}{\mathcal{PC}}
\newcommand{\RR}{\mathbf{R}}

\newcommand{\VV}{\overline{V}}
\newcommand{\m}{\mathbf{m}}

\newcommand{\esupp}{\mathrm{ess \thinspace supp}}
\newcommand{\odd}{\mathrm{odd}}

\newcommand{\up}{\underline{\smash{p}}}

\newcounter{count}

\newcounter{counterk} 
\newcommand{\newconk}[1]{\refstepcounter{counterk}\label{#1}} 
\newcommand{\useconk}[1]{\kappa_{\ref{#1}}}

\newcounter{counterc} 
\newcommand{\newconc}[1]{\refstepcounter{counterc}\label{#1}} 
\newcommand{\useconc}[1]{c_{\ref{#1}}}

\newcounter{countercc} 
\newcommand{\newconcc}[1]{\refstepcounter{countercc}\label{#1}} 
\newcommand{\useconcc}[1]{C_{\ref{#1}}}

\newcounter{countere} 
\newcommand{\newcone}[1]{\refstepcounter{countere}\label{#1}} 
\newcommand{\usecone}[1]{\epsilon_{\ref{#1}}}

\newcounter{counterd} 
\newcommand{\newcond}[1]{\refstepcounter{counterd}\label{#1}} 
\newcommand{\usecond}[1]{\delta_{\ref{#1}}}

\newcounter{countern} 
\newcommand{\newconn}[1]{\refstepcounter{countern}\label{#1}} 
\newcommand{\useconn}[1]{N_{\ref{#1}}}

\newcounter{counterr}
\newcommand{\newconr}[1]{\refstepcounter{counterr}\label{#1}} 
\newcommand{\useconr}[1]{r_{\ref{#1}}}

\begin{document}

\title[Effective mapping class group dynamics I]{Effective mapping class group dynamics I:\\ Counting lattice points in Teichmüller space}

\author{Francisco Arana--Herrera}
\address{Department of Mathematics, Stanford University, 450 Jane Stanford Way, Stanford, CA 94305, USA.}
\email{farana@stanford.edu}

\begin{abstract}
	We prove a quantitative estimate with a power saving error term for the number of points in a mapping class group orbit of Teichmüller space that lie within a Teichmüller metric ball of given center and large radius. Estimates of the same kind are also proved for sector and bisector counts. These estimates effectivize asymptotic counting results of Athreya, Bufetov, Eskin, and Mirzakhani.
\end{abstract}

\maketitle

%    Text of article.

\thispagestyle{empty}

\tableofcontents

\section{Introduction}

In his thesis \cite{Mar04}, Margulis proved a precise asymptotic formula for the number of points in a $\pi_1(M)$-orbit of the universal cover of a compact, negatively curved Riemannian manifold $M$ that lie within a ball of given center and large radius. The techniques introduced by Margulis in this work have proved to be quite robust. For instance, Eskin, McMullen, Gorodnik, Oh, and Shah \cite{EM93, GO07, GOS10} applied these techniques to prove analogous counting results for general locally symmetric spaces.

In \cite{ABEM12}, Athreya, Bufetov, Eskin, and Mirzakhani adapted Margulis's techniques to prove an asymptotic formula for the number of points in a mapping class group orbit of Teichmüller space that lie within a Teichmüller metric ball of given center and large radius. Furthermore, analogous asymptotic formulas for sector and bisector counts are also proved in the same work. In this paper we effectivize these results by proving quantitative estimates with a power saving error term for the same quantities.

The quantitative bisector counting results proved in this paper can be used to tackle a wide variety of related effective counting problems. For instance, in \cite{Ara21b} we combine these results with the main theorem of \cite{Ara21a} to prove a quantitative estimate with a power saving error term for the number of filling closed geodesics of a given topological type and length $\leq L$ on an arbitrary closed, negatively curved surface. This estimate complements a recent theorem of Eskin, Mirzakhani, and Mohammadi for countings of simple closed geodesics \cite{EMM19}, effectivizes asymptotic counting results of Mirzakhani, Erlandsson, and Souto \cite{Mir16,ES16}, and solves an open problem advertised by Wright \cite[Problem 18.2]{Wri19} for a generic class of closed curves. In the same work we also prove a quantitative estimate with a power saving error term for the number of points in a mapping class group orbit of Teichmüller space that lie within a Thurston metric ball of given center and large radius, effectivizing an asymptotic counting result of Rafi and Souto \cite{RS19}.

The proofs of the main theorems of this paper exploit the strong connection shared between the mapping class group and the Teichmüller geodesic flow. Indeed, the main source of effective estimates employed in the proofs is the exponential mixing property of the Teichmüller geodesic flow on the principal stratum of quadratic differentials \cite{AR12}. The proofs also rely on a solid quantitative understanding of the dynamics of the Teichmüller geodesic flow on the thin part of the principal stratum, which we develop using new analytic estimates of Kahn and Wright \cite{KW20}. In order to prove the effective sector and bisector counting results, a good understanding of the relation between the piecewise linear structures of the principal stratum and the space of singular measured foliations is required. To this end, we show that these structures are related in a Lipschitz way, and that regular test functions on the space of singular measured foliations induce regular test functions, in the sense of Ratner \cite{Ra87}, on the principal stratum. The proofs in this paper illustrate the rich interplay between the different analytic, geometric, and dynamical perspectives one can consider when studying Teichmüller space.

\subsection*{Statements of the main theorems.} For the rest of this paper we fix an integer $g \geq 2$ and a connected, oriented, closed surface $S_g$ of genus $g$. Denote by $\tt$ the Teichmüller space of marked complex structures on $S_g$, by $\mcg$ the mapping class group of $S_g$, and by $\mm := \tt/\mcg$ the moduli space of complex structures on $S_g$. Let $h := 6g-6$. Denote by $B_R(X) \subseteq \mathcal{T}_g$ be the ball of radius $R > 0$ centered at $X \in \mathcal{T}_g$ with respect to the Teichmüller metric. The following theorem, which corresponds to an effective version of \cite[Theorem 1.2]{ABEM12}, is one of the main results of this paper.

\begin{theorem}
	\label{theo:count_intro}
	Let $\mathcal{K} \subseteq \tt$ be a compact subset. Then, for every $X,Y \in \mathcal{K}$ and every $R>0$,
	\[
	\#\left(\mcg\cdot Y \cap B_R(X) \right) =  C \cdot e^{hR} + O_\mathcal{K}\left(e^{(h-\kappa)R}\right),
	\]
	where $C,\kappa > 0$ are constants depending only on $g$ and where the constant implicit in $O_\mathcal{K}$ depends only on $\mathcal{K}$.
\end{theorem}

A more explicit description of the constant $C$ in Theorem \ref{theo:count_intro} will be provided in Theorem \ref{theo:count}. In this paper we also prove effective versions of the asymptotic sector and bisector counting results \cite[Theorems 2.9, 2.10]{ABEM12}. See Theorems \ref{theo:sector_count_box} and \ref{theo:sector_count_smooth}, and Theorems \ref{theo:bisector_count_box} and \ref{theo:bisector_count_smooth}, respectively, for precise statements. 

\subsection*{Main ideas of the proof of Theorem \ref{theo:count_intro}.} To prove Theorem \ref{theo:count_intro} we follow the general approach of \cite{ABEM12}, incorporating a wide variety of new quantitative estimates throughout. Averaging and unfolding arguments reduce the proof of Theorem \ref{theo:count_intro} to an effective mean equidistribution theorem for Teichmüller metric balls, which we now describe.

Denote by $\qut$ the Teichmüller space of marked, unit area holomorphic quadratic differentials on $S_g$. Consider the natural projection $\pi \colon \qut \to \tt$. Let $\mu$ be the Masur-Veech measure on $\qut$ and $\mathbf{m} := \pi_* \mu$ be its pushfoward to $\tt$. Denote by $\widehat{\mathbf{m}}$ the pushforward of $\mathbf{m}$ to $\mm$. Every Teichmüller metric ball $B_R(X) \subseteq \tt$ carries a natural measure $\mathbf{m}_X^R$ defined as the restriction of $\mathbf{m}$ to it. Denote by $\widehat{\mathbf{m}}_X^R$ the pushforward of $\mathbf{m}_X^R$ to $\mathcal{M}_g$. This measure, which only depends on $X \in \mm$ and not on the marking, keeps track of how the image of $B_R(X)$ wraps around $\mm$. The main tool used in the proof of Theorem \ref{theo:count_intro}, also of independent interest, is the following effective mean equidistribution theorem for Teichmüller metric balls, which corresponds to an effective version of \cite[Theorem 2.8]{ABEM12}.

\begin{theorem}
	\label{theo:ball_equidistribution_intro}
	Let $\mathcal{K} \subseteq \mathcal{M}_g$ be a compact subset and $\phi_1,\phi_2 \in  L^\infty(\mm,\widehat{\m})$ be essentially bounded functions with  $\esupp(\phi_1), \esupp(\phi_2) \subseteq \mathcal{K}$. Then, for every $R > 0$,
	\begin{gather*}
	\int_{\mm} \phi_1(X) \left( \int_{\mm} \phi_2(Y) \ d\widehat{\mathbf{m}}_X^R(Y) \right) d\widehat{\mathbf{m}}(X) \\
	=  C\cdot \left(\int_{\mm} \phi_1(X) \ d\widehat{\mathbf{m}}(X)\right) \cdot \left(\int_{\mm} \phi_2(Y) \ d\widehat{\mathbf{m}}(Y)\right) \cdot e^{hR} \\
	+ O_\mathcal{K}\left(\|\phi_1\|_\infty \cdot \|\phi_2\|_\infty \cdot e^{(h-\kappa)R}\right),
	\end{gather*}
	where $C,\kappa > 0$ are constants depending only on $g$ and where the constant implicit in $O_\mathcal{K}$ depends only on $\mathcal{K}$.
\end{theorem}

A more explicit description of the constant $C$ in Theorem \ref{theo:ball_equidistribution_intro} will be provided in Theorem \ref{theo:ball_equidistribution}.  In this paper we also prove effective mean equidistribution theorems for sectors of $\tt$ and $\qut$; see Theorems \ref{theo:sector_equidistribution} and \ref{theo:sector_equidistribution_box}, and Theorems \ref{theo:bisector_equidistribution} and \ref{theo:bisector_equidistribution_box}, respectively, for precise statements. These theorems are one of the main tools used in the proofs of the effective sector and bisector counting results, Theorems \ref{theo:sector_count_box}, \ref{theo:sector_count_smooth},  \ref{theo:bisector_count_box}, and \ref{theo:bisector_count_smooth}, but a lot more machinery is required in this case.

Denote by $\qum := \qut/\mcg$ the moduli space of unit area holomorphic quadratic differentials on $S_g$ and by $\widehat{\mu}$ the Masur-Veech measure on $\qum$. The main driving force behing the proof of Theorem \ref{theo:ball_equidistribution_intro} is the exponential mixing property of the Teichmüller geodesic flow on $\qum$ with respect $\widehat{\mu}$. This property was proved by Avila and Resende \cite{AR12}, building on previous work of Avila, Gouëzel, and Yoccoz \cite{AGY06}. Another important tool used in the proof of Theorem \ref{theo:ball_equidistribution_intro} is a description of the measure $\mathbf{m}$ in \textit{polar coordinates}. To explain the nature of this description we first introduce some notation.  

Denote by $S(X) := \pi^{-1}(X) \subseteq \qut$ the fiber of the projection $\pi \colon \qut \to \tt$ above $X \in \mathcal{T}_g$. Let $\{s_X\}_{X \in \mathcal{T}_g}$ be the measures obtained by disintegrating the Masur-Veech measure $\mu$ on $\qut$ along these fibers. Denote by $\qutp$ the principal stratum of $\qut$. The measures $s_X$ give zero mass to the multiple zero locus $\qut \backslash \qutp$. Denote by $a_t \colon \qut \to \qut$, $t \in \mathbf{R}$, the Teichmüller geodesic flow.

Fix $X \in \tt$. Consider the \textit{polar coordinates} map $\Phi_X \colon S(X) \times \mathbf{R}_{>0} \to \mathcal{T}_g$ given by $\Phi_X(q,t) := \pi(a_tq)$. This map is a homeomorphism onto $\tt \backslash \{X\}$ and a diffeomorphism onto its image when restricted to $(S(X) \cap \qutp) \times \mathbf{R}_{>0}$. In particular, for every $q \in S(X) \cap \qutp$ and every $t > 0$ we can write
\[
|\Phi_X^*(\mathbf{m})(q,t)| = \Delta(q,t) \cdot |s_X(q) \wedge dt|,
\]
where $\Delta \colon (S(X) \cap \qutp) \times \mathbf{R}_{>0} \to \mathbf{R}_{>0}$ is a smooth, positive function. Theorem \ref{theo:polar_coordinates_estimate}, which corresponds to a strengthening of \cite[Proposition 2.5]{ABEM12}, provides a quantitative estimate for the function $\Delta$. The leading term of this estimate is described using the Hubbard-Masur functions $\lambda^-,\lambda^+ \colon \qutp \to \mathbf{R}_{>0}$ introduced in \cite[\S 2.3]{ABEM12}. The main tools used in the proof of Theorem \ref{theo:polar_coordinates_estimate} are:
\begin{itemize}
	\item The bounds of Forni on the spectral gap of the Kontsevich-Zorich cocycle on $\qutp$ \cite{F02}. 
	\item The analysis of the projection $\pi \colon \qut \to \tt$ in forthcoming work of Kahn and Wright \cite{KW20}.
\end{itemize}

To use Theorem \ref{theo:polar_coordinates_estimate} in combination with the exponential mixing rate of the Teichmüller geodesic flow on $\qum$, one needs to have a reasonable understanding of the regularity of the Hubbard-Masur functions $\lambda^-,\lambda^+ \colon \qutp \to \mathbf{R}_{>0}$. Using work of Dumas \cite{Du15}, we show these functions are $\mathrm{SO}(2)$-invariant; see Proposition \ref{prop:hm_inv_sum}. Using forthcoming work of Kahn and Wright \cite{KW20}, we prove bounds for these functions in terms of the length of shortest saddle connection of $q \in \qutp$; see Proposition \ref{prop:HM_bound_sum}.

The quantitative estimate for the function $\Delta$ in Theorem \ref{theo:polar_coordinates_estimate}, although valid everywhere on the principal stratum, gets worse as one considers quadratic differentials approaching the multiple zero locus. Completely different arguments are introduced to control the corresponding contributions. More concretely, we prove effective versions of \cite[Theorems 2.6 and 2.7]{ABEM12} as well as other closely related results; see Proposition \ref{prop:small_HM_integral_sum} and Theorems \ref{theo:thin_trajectories_sum}, \ref{theo:thin_sector_sum}, and \ref{theo:large_deviations_sum}. The main sources of effective estimates in this case are:
\begin{itemize}
	\item The quantitative estimates of Eskin, Mirzakhani, and Rafi for the number of Teichmüller geodesic segments that spend a definite fraction of their time in the thin part of a stratum \cite{EMR19}.
	\item The metric hyperbolicity properties of the Teichmüller geodesic flow on $\qut$ proved in \cite{ABEM12}.
	\item The large deviation estimates of Athreya for the Teichmüller geodesic flow on $\qum$ \cite{A06}.
\end{itemize} 

\subsection*{Finite index subgroups of the mapping class group.} The results in this paper, in particular Theorems \ref{theo:count_intro}, \ref{theo:count}, \ref{theo:sector_count_box}, \ref{theo:sector_count_smooth}, \ref{theo:bisector_count_box}, \ref{theo:bisector_count_smooth},
and \ref{theo:bisector_equidistribution_box}, hold when $\mcg$ is replaced with a finite index subgroup $\Gamma \subseteq \mcg$. The crucial observation which explains this remark is the fact that the Teichmüller geodesic flow on $\qut / \Gamma$ is exponentially mixing with respect to the local pushforward of the Masur-Veech measure $\mu$ on $\qut$.

\subsection*{Organization of the paper.} In \S 2 we develop the preliminaries necessary to understand the proofs of the main theorems. In \S 3 we present the proof of Theorem \ref{theo:ball_equidistribution_intro} in complete detail. In \S 4 we use Theorem \ref{theo:ball_equidistribution_intro} to prove Theorem \ref{theo:count_intro}. In \S5 we study the regularity of the Hubbard-Masur functions. In \S6 we prove the crucial estimate for the function $\Delta$. In \S7 we prove the estimates needed to control the contributions near the multiple zero locus. In \S8 we strengthen the arguments used in the proof of Theorem \ref{theo:ball_equidistribution_intro} to prove effective mean equidistribution theorems for sectors of Teichmüller space. In \S9 we use these results to extend Theorem \ref{theo:count_intro} to sector counts. In \S10 we further strengthen these results to bisector counts.

\subsection*{Acknowledgments.} The author is very grateful to Alex Wright and Steve Kerckhoff for their invaluable advice, patience, and encouragement. The author would also like to thank Alex Eskin, Amir Mohammadi, Ian Frankel, Jayadev Athreya, and Carlos Matheus for very helpful and enlightening conversations. This work got started while the author was participating in the \textit{Dynamics: Topology and Numbers} trimester program at the Hausdorff Research Institute for Mathematics (HIM). The author is very grateful for the HIM's hospitality and for the hard work of the organizers of the trimester program.

\section{Preliminaries}

\subsection*{Outline of this section.} In this section we cover the background material needed to understand the proofs of the main theorems of this paper. Of particular importance for later will be the definitions (\ref{eq:hm_1}), (\ref{eq:hm_2}), (\ref{eq:hm_3}), and (\ref{eq:hm_4}) of the Hubbard-Masur functions.

\subsection*{Abelian and quadratic differentials.} Let $X$ be a Riemann surface and $K$ be its canonical bundle. Holomorphic sections of $K$ correspond to holomorphic $1$-forms on $X$. An Abelian differential $\omega$ on $X$ is a holomorphic section of $K$. In local coordinates, $\omega = f(z) dz$ for some holomorphic function $f(z)$. A quadratic differential $q$ on $X$ is a holomorphic section of the symmetric square $K \vee K$. In local coordinates, $q = f(z) dz^2$ for some holomorphic function $f(z)$. If $X$ has genus $g$, the number of zeroes of $q$ counted with multiplicity is $4g-4$. We will sometimes denote quadratic differentials by $(X,q)$, keeping track of the Riemann surface they are defined on. The area of a quadratic differential is $\text{Area}(X,q) := \int_X |q|$. We denote by $Q(X)$ the complex vector space of quadratic differentials on $X$ and by $S(X) \subseteq Q(X)$ its unit area locus. 

\subsection*{Teichm\"uller and moduli spaces of quadratic differentials.} Denote by $\mathcal{Q}\mathcal{T}_{g}$ the Teichm\"uller space of marked, non-zero quadratic differentials on $S_g$ and by $\qut \subseteq \mathcal{Q}\mathcal{T}_g$ its unit area locus. The mapping class group $\mcg$ acts properly discontinously on $\qt$ by changing the markings. This action preserves the unit area locus $\qut \subseteq \qt$. The quotient $\mathcal{Q}^1\mathcal{M}_g := \qut /\text{Mod}_g$ is the moduli space of unit area quadratic differentials on $S_g$. The following diagram summarizes the notation to be used in the rest of this paper for the natural quotient and projection maps:

\begin{center}
\begin{tikzcd}
	\qut \arrow[r, "p"] \arrow[d, "\pi" ] & \qum \arrow[d, "\underline{\pi}"] \\ 
	\tt \arrow[r, "\underline{\smash{p}}"] & \mm 
\end{tikzcd}.
\end{center}

\subsection*{The Masur-Veech volume form.} The projection $Q\mathcal{T}_{g} \to \mathcal{T}_{g}$ makes $Q\mathcal{T}_{g}$ into a bundle over $\mathcal{T}_{g}$. This bundle can be naturally identified with the complement of the zero section of the cotangent bundle of $\mathcal{T}_g$. In particular, $Q\mathcal{T}_{g}$ supports a canonical volume form $\mu'$, called the Masur-Veech volume form. The corresponding smooth measure on $\qt$ is called the Masur-Veech measure. Contracting $\mu'$ by any vector field $V$ satisfying $d\sqrt{\text{Area}}(V) \equiv 1$ yields a volume form $\mu$ on $\qut$, also called the Masur-Veech volume form. The corresponding smooth measure on $\qut$ is also called the Masur-Veech measure. This measure is invariant under the action of $\mcg$ on $\qut$. Its local pushforward to $\mathcal{Q}^1\mathcal{M}_g$, which we denote by $\widehat{\mu}$, is also called the Masur-Veech measure. These measures were originally introduced by Masur \cite{Ma82} and Veech \cite{Ve82} to study the prevalence of unique ergodicity among interval exchange transformations. As part of this work, Masur and Veech independently showed that the measure $\widehat{\mu}$ on $\qum$ is finite.

\subsection*{Singular measured foliations.} Denote by $\mf$ the space of singular measured foliations on $S_g$ up to isotopy and Whitehead moves. The set of weighted simple closed curves on $S_g$ embeds densely into $\mf$. The geometric intersection number of weighted simple closed curves extends continuously to a pairing $i \colon \mathcal{MF}_g \times \mathcal{MF}_g \to \mathbf{R}_{\geq 0}$, also called geometric intersection number. Train track coodinates induce a natural piecewise integral linear structure on $\mf$. In particular, $\mf$ carries a natural Lebesgue class measure called the Thurston measure. We denote by $\nu$ the normalization of the Thurston measure induced by the symplectic structure described in \cite[\S 3.2]{PH92}. This measure gives zero mass to the subset of singular measured foliations having a singularity with more than three prongs \cite[Lemma 2.4]{Mir08a}. The space $\mathcal{PMF}_g$ of projective singular measures foliations on $S_g$ is the quotient of $\mathcal{MF}_g$ by the natural scaling action of $\mathbf{R}_{>0}$ on transverse measures. Denote by $[\eta] \in \pmf$ the equivalence class of $\eta \in \mf$. We refer the reader to \cite[\S 11.2]{FM11} and \cite[\S 5]{FLP12} for further details on the theory of singular measured foliations.

\subsection*{The Hubbard-Masur map.} Every $q \in \qt$ gives rise to singular measured foliations $\Re(q),\Im(q) \in \mf$. If $q = dz^2$ in local coordinates $z = x + iy$, the $1$-forms $dx$ and $dy$ induce the singular measured foliations $\Re(q)$ and $\Im(q)$, respectively. Denote by $\Re,\Im \colon \qt \to \mf$ the corresponding maps. Let $\Delta \subseteq \mathcal{MF}_g \times \mathcal{MF}_g$ the closed set of non transverse pairs of singular measured foliations on $S_g$. More precisely,
\[
\Delta := \{(\eta_1,\eta_2) \in \mf \times \mf \ | \ \exists \eta_3 \in \mf\colon i(\eta_1,\eta_3) = i(\eta_2,\eta_3) = 0 \}.
\]
The map $\qt \to \mf \times \mf - \Delta$ given by $q \mapsto (\Re(q),\Im(q))$ is a homeomorphism sending area of quadratic differentials to geometric intersection number of singular measured foliations. Following the convention in \cite{LM08}, we refer to this map as the Hubbard-Masur map. The pullback of the product of Thurston measures $\nu \times \nu$ on $\mathcal{MF}_g \times \mathcal{MF}_g - \Delta$ is equal to the Masur-Veech measure $\mu'$ on $\qt$ up to a multiplicative constant depending only on $g$ \cite[Lemma 4.3]{Mir08a}. 

\subsection*{Half-translation structures.} A half-translation structure on a surface $S$ corresponds to an atlas of charts to $\mathbf{C}$ on the complement of a finite set of points $\Sigma \subseteq S$ whose transition functions are of the form $z \mapsto \pm z + c$ with $c \in \mathbf{C}$. Every quadratic differential $q$ gives rise to a half-translation structure on the Riemann surface it is defined on by considering local coordinates on the complement of the zeroes of $q$ for which $q = dz^2$. Viceversa, every half-translation structure induces a quadratic differential on its underlying surface by pulling back the differential $dz^2$ on the corresponding charts.

Pulling back the standard Euclidean metric on $\mathbf{C}$ using the charts of a half-translation structure induces a singular Euclidean metric on the underlying surface. In particular, every quadratic differential $q$ gives rise to a singular Euclidean metric on the Riemann surface it is defined on. This metric has a cone point of angle $(k+2) \pi$ at every zero of $q$ of order $k$. A saddle connection of $q$ is a Euclidean geodesic segment joining two zeroes of $q$ and having no zeroes in its interior. Denote by $\ell_{\gamma}(q) > 0$ the Euclidean length of a saddle connection $\gamma$ of $q$ and by $\ell_{\min}(q) > 0$ the Euclidean length of the shortest saddle connection of $q$. Denote by $\mathrm{diam}(q) > 0$ the Euclidean diameter of $q$. The total Euclidean area of $q$ is precisely $\mathrm{Area}(q) > 0$.

\subsection*{The $\mathbf{SL(2,\mathbf{R})}$ action and the Teichmüller geodesic flow.} The group $\mathrm{SL}(2,\mathbf{R})$ acts naturally on half-translation structures by postcomposing the corresponding charts with the linear action on $\mathbf{C} =\mathbf{R}^2$ of the elements of this group. In particular, $\mathrm{SL}(2,\mathbf{R})$ acts naturally on $\qt$, preserving $\qut$. 

For every $t \in \mathbf{R}$ and every $\theta \in [0,2\pi]$ denote
\[
a_t := \left( 
\begin{array}{c c}
e^t & 0 \\
0 & e^{-t}
\end{array} 
 \right), \quad 
r_\theta := \left(
\begin{array}{c c}
\cos \theta & -\sin \theta \\
\sin \theta & \cos \theta
\end{array}
\right).
\]
The flow induced by the action of the subgroup $\{a_t\}_{t \in \mathbf{R}} \subseteq \mathrm{SL}(2,\mathbf{R})$ on $\qt$ is the Teichmüller geodesic flow. For every $\theta \in [0,2\pi]$ and every $q \in \qt$, $r_\theta q = e^{2\theta i}q$. In particular, the action of $\mathrm{SO}(2)$ preserves the fibers $Q(X)\backslash\{0\}$ of $\qt$ and $S(X)$ of $\qut$.

The $\mathrm{SL}(2,\mathbf{R})$ and $\text{Mod}_g$ actions on $\qut$ commute. In particular, there is a well defined $\mathrm{SL}(2,\mathbf{R})$ action and a well defined Teichmüller geodesic flow on $\qum$. 

\subsection*{The dynamical foliations of $\boldsymbol{\mathbf{\qut}}$.} The unit area locus $\qut \subseteq \qt$ can be foliated in several dynamically meaningful ways. For every $q_0 \in \qut$, its strongly stable, central, and strongly unstable leaves $\alpha^{ss}(q_0), \alpha^c(q_0), \allowbreak\alpha^{uu}(q_0) \subseteq \qut$ are given by
\begin{align*}
\alpha^{ss}(q_0) &:= \{ q \in \qut \ | \ \Re(q) = \Re(q_0) \}, \\
\alpha^c(q_0) &:= \{a_t q_0 \ | \ t \in \mathbf{R}\},\\
\alpha^{uu}(q_0) &:= \{ q \in \qut \ | \ \Im(q) = \Im(q_0) \}. 
\end{align*}
These leaves give rise to topological foliations $\mathcal{F}^{ss}$, $\mathcal{F}^c$, and $\mathcal{F}^{uu}$ called the strongly stable, central, and strongly unstable foliations of $\qut$. The stable and unstable leaves $\alpha^{s}(q_0), \allowbreak \alpha^{u}(q_0) \subseteq \qut$ of a quadratic differential $q_0 \in \qut$ are given by
\begin{align*}
\alpha^s(q_0) &:= \{ q \in \qut \ | \ [\Re(q)] = [\Re(q_0)] \} = \bigcup_{t \in \mathbf{R}} a_t \alpha^{ss}(q_0), \\
\alpha^u(q_0) &:= \{ q \in \qut \ | \ [\Im(q)] = [\Im(q_0)] \} = \bigcup_{t \in \mathbf{R}} a_t \alpha^{uu}(q_0). 
\end{align*}
These leaves give rise to topological foliations $\mathcal{F}^{s}$ and $\mathcal{F}^{u}$ called the stable and unstable foliations of $\qut$. The foliations $\mathcal{F}^{ss}$, $\mathcal{F}^{uu}$, $\mathcal{F}^{s}$, and $\mathcal{F}^{u}$ are, respectively, the strongly stable, strongly unstable, stable, and unstable foliations of the Teichmüller geodesic flow in the sense of Veech \cite{Ve86} and Forni \cite{F02}.

\subsection*{The leafwise measures of $\boldsymbol{\qut}$.} Given $\eta \in \mathcal{MF}_g$, denote by $\mathcal{MF}_g(\eta) \subseteq \mathcal{MF}_g$ the open, dense, full measure subset of singular measured foliations on $S_g$ transverse to $\eta$. Let $q_0 \in \qut$. The map $\alpha^{s}(q_0) \to \mathcal{MF}_g(\Re(q_0))$ given by $q \mapsto \Im(q)$ is a homeomorphism. Denote by $\mu_{\alpha^{s}(q_0)}$ the pullback to $\alpha^{s}(q_0)$ of the Thurston measure $\nu$ on $\mathcal{MF}_g(\Re(q_0))$. Analogously, the map $\alpha^{u}(q_0) \to \mathcal{MF}_g(\Im(q_0))$ given by $q \mapsto \Re(q)$ is a homeomorphism. Denote by $\mu_{\alpha^{u}(q_0)}$ the pullback to $\alpha^{u}(q_0)$ of the Thurston measure $\nu$ on $\mathcal{MF}_g(\Im(q_0))$.

For every $\eta \in \mathcal{MF}_g$ denote 
\[
\mathcal{MF}_g^1(\eta) := \{\eta' \in \mathcal{MF}_g(\eta) \ | \ i(\eta,\eta') = 1\}.
\]
This subset carries a natural measure $\overline{\nu}_\eta$ obtained by conning-off the Thurston measure $\nu$ on $\mathcal{MF}_g(\eta)$. More precisely, for every measurable subset $A \subseteq \mathcal{MF}_g^1(\eta)$, 
$
\overline{\nu}_\eta(A) := \nu([0,1] \cdot A).
$

Let $q_0 \in \qut$. The map $\alpha^{ss}(q_0) \to \mathcal{MF}_g^1(\Re(q_0))$ given by $q \mapsto \Im(q)$ is a homeomorphism. Denote by $\mu_{\alpha^{ss}(q_0)}$ the pullback to $\alpha^{ss}(q_0)$ of the measure $\overline{\nu}_{\Re(q_0)}$ on $\mathcal{MF}_g^1(\Re(q_0))$. Analogously, the map $\alpha^{uu}(q_0) \to \mathcal{MF}_g(\Im(q_0))$ given by $q \mapsto \Re(q)$ is a homeomorphism. Denote by $\mu_{\alpha^{uu}(q_0)}$ the pullback to $\alpha^{uu}(q_0)$ of the measure $\overline{\nu}_{\Im(q_0)}$ on $\mathcal{MF}_g^1(\Im(q_0))$.

\subsection*{Strata of quadratic differentials.} The space $\qt$ can be stratified according to the orders of the zeroes. More concretely, given an integer partition $\mathbf{k} := (k_1,\dots,k_n)$ of $4g-4$ and a boolean parameter $\epsilon \in \{0,1\}$, $\qt(\mathbf{k},\epsilon) \subseteq \qt$ denotes the strata of marked quadratic differentials on $S_g$ whose zeroes have orders given by $\mathbf{k}$ and which are the square of an Abelian differential if and only if $\epsilon = 1$. The natural $\mathrm{SL}(2,\RR)$ action on $\qt$ preserves this stratification. The unit area locus $\qut \subseteq \qt$ can be stratified in the same way. The corresponding strata will be denoted by $\qut(\mathbf{k},\epsilon) \subseteq \qut$. The $\mcg$ action on $\qut$ preserves this stratification and so $\qum$ naturally inherits it. The corresponding strata will be denoted by $\qum(\mathbf{k},\epsilon) \subseteq \qum$. 

The principal strata of $\qt$, $\qut$, and $\qum$, corresponding to $\mathbf{k} = (1,\dots, 1)$ and $\epsilon = 0$, will be denoted by $\qt(\mathbf{1})$, $\qut(\mathbf{1})$, and $\qum(\mathbf{1})$, respectively. These are the unique top-dimensional strata of $\qt$, $\qut$, and $\qum$. Their complements in $\qt$, $\qut$, and $\qum$, the multiple zero loci, are zero measure subsets of the respective Lebesgue measure classes \cite[Theorem 2.2]{ABEM12}.

\subsection*{Masur's compactness criterion.}  For every $\delta> 0$ consider the subset of $\qum$ given by
\begin{equation*}
\label{eq:compact}
K_\delta := \{q \in \qum \ | \ \ell_{\min}(q) \geq \epsilon \}.
\end{equation*}
By Masur's compactness criterion \cite[Proposition 3.6]{MT02}, this subset is compact. Moreover, the intersection of this subset with any stratum of $\qum$ is compact. Denote by $K_\delta(\mathbf{1}) \subseteq \qum(\mathbf{1})$ its intersection with the principal stratum.

\subsection*{Period coordinates of strata.} Every quadratic differential $q$ gives rise to a canonical double cover $\widetilde{X}$ of its underlying Riemann surface $X$, branched over the odd order zeroes of $q$, onto which $q$ pulls back to the square of an Abelian differential $\omega$. This double cover supports a canonical involution $\sigma \colon \widetilde{X} \to \widetilde{X}$ satisfying $\sigma^* \omega = - \omega$ and $X = \widetilde{X}/\langle\sigma\rangle$. Denote by $\Sigma \subseteq X$ the singularities of $q$ and by $\widetilde{\Sigma}$ the lifts of these singularities to $\widetilde{X}$. Let $H_1^\odd(\widetilde{X},\widetilde{\Sigma};\mathbf{Z})$, $H^1_\odd(\widetilde{X},\widetilde{\Sigma};\mathbf{R})$, and $H^1_\odd(\widetilde{X},\widetilde{\Sigma};\mathbf{C})$ be the $-1$ eigenspaces of the linear involutions induced by $\sigma$ on the respective homology and cohomology groups. Consider the maps $\Re, \Im \colon H^1_\odd(\widetilde{X},\widetilde{\Sigma};\mathbf{C}) \to H^1_\odd(\widetilde{X},\widetilde{\Sigma};\mathbf{R})$. Denote by $\cup$ the cup intersection form on $H_\odd^1(\widetilde{X};\mathbf{C})$. We normalize $\cup$ so that the area of $q$ is given by $\mathrm{Area}(q) = \Re([\omega]) \cup \Im([\omega])$. 

Fix a stratum $\qt(\mathbf{k},\epsilon) \subseteq \qt$. The Gauss-Manin connection locally identifies the cohomology groups $H^1_\odd(\widetilde{X},\widetilde{\Sigma};\mathbf{C})$ coming from canonical double covers of quadratic differentials in $\qt(\mathbf{k},\epsilon)$. Using these local identifications one can define maps on open subsets of the stratum which take $q \in \qt(\mathbf{k},\epsilon)$ to $[\omega] \in H^1_\odd(\widetilde{X},\widetilde{\Sigma};\mathbf{C})$. Every class in $H^1_\odd(\widetilde{X},\widetilde{\Sigma};\mathbf{C})$ is determined by its action on $H_1^\odd(\widetilde{X},\widetilde{\Sigma};\mathbf{Z})$. Thus, choosing a basis of the $\mathbf{Z}$-module $H_1^\odd(\widetilde{X},\widetilde{\Sigma};\mathbf{Z})$ yields  maps from open subsets of the stratum to $\mathbf{C}^d$, where $d:= 2g-2 + |\mathbf{k}| + \epsilon$ is the dimension of $\qt(\mathbf{k},\epsilon)$. These maps provide local coordinates whose transition functions are integral linear transformations. These coordinates are known as period coordinates. They identify the tangent space of $q \in \qt(\mathbf{k},\epsilon)$ with $H^1_\odd(\widetilde{X},\widetilde{\Sigma};\mathbf{C})$. For more details see \cite[\S 2]{Lan04}.

In the case of the principal stratum $\qt(\mathbf{1})$, the forgetful maps $H^1_\odd(\widetilde{X},\widetilde{\Sigma};\mathbf{C}) \to H^1_\odd(\widetilde{X};\mathbf{C})$ are isomorphisms. In particular, one can identify the tangent space of $q \in \qt(\mathbf{1})$ with $H^1_\odd(\widetilde{X};\mathbf{C})$. Under this identification, for every $X \in \tt$ and every $q \in Q(X) \cap \qt(\mathbf{1})$, $T_qQ(X) = H^{1,0}_\odd(\widetilde{X};\mathbf{C})$, the subspace of odd holomorphic $1$-forms on $\widetilde{X}$ \cite[Lemma 1]{DH75}. 

The restriction of the $\mathrm{SL}(2,\mathbf{R})$ action to any stratum of $\qt$ coincides with the product linear action of this group on $\mathbf{C}^d = (\mathbf{R}^2)^d$ after local identification using period coordinates. 

\subsection*{The Masur-Veech measure revisited.} Pulling back the Lebesgue measure on $\mathbf{C}^{6g-6}$ through period coordinates yields a smooth measure on the principal stratum $\qt(\mathbf{1}) \subseteq \qt$. By work of Masur \cite{M95} and Bertola, Korotkin, and Norton \cite{BKN17}, this measure is equal to the Masur-Veech measure $\mu'$ on $\qt$ up to a multiplicative constant depending only on $g$. In particular, the  $\mathrm{SL}(2,\mathbf{R})$ action preserves the Masur-Veech measure $\mu$ on $\qut$.

\subsection*{The local product structure of $\boldsymbol{\qutp}$.} The intersections of the foliations $\mathcal{F}^{ss}$, $\mathcal{F}^c$, $\mathcal{F}^{uu}$, $\mathcal{F}^{s}$, and $\mathcal{F}^{u}$ of $\qut$ with the principal stratum $\qutp$ can be described explicitely in period coordinates. Given a quadratic differential $q \in \qutp$, identify $T_q\qutp$ with $H_\odd^1(\widetilde{X};\mathbf{C})$ using the canonical double cover $\widetilde{X}$ of $q$. Consider the linear subspaces $E^{ss}(q), E^c(q),E^{uu}(q) \subseteq H_\odd^1(\widetilde{X};\mathbf{C})$ given by
\begin{align*}
E^{ss}(q) &:= \{v \in H_\odd^1(\widetilde{X};i\mathbf{R}) \ | \ v \cup \Re([\omega]) = 0 \},\\
E^c(q) &:= \mathbf{R} \cdot [\overline{\omega}],\\
E^{uu}(q) &:= \{v \in H_\odd^1(\widetilde{X};\mathbf{R}) \ | \ v \cup \Im([\omega]) = 0 \}.
\end{align*}
The tangent space of $\qutp$ at $q$ can be decomposed as
%\begin{equation}
%\label{eq:local_prod_struct}
\[
T_q \qutp = E^{uu}(q) \oplus E^c(q) \oplus E^{ss}(q).
\]
%\end{equation}
We refer to the subspaces $E^{ss}(q)$, $E^c(q)$, $E^{uu}(q)$, $E^{s}(q) := E^{ss}(q) \oplus E^c(q)$, and $E^{u}(q) := E^c(q) \oplus E^{uu}(q)$, as the strongly stable, central, strongly unstable, stable, and unstable subspaces of $q$. These subspaces correspond to the tangent spaces at $q$ of the intersections of $\qutp$ with the foliations $\mathcal{F}^{ss}$, $\mathcal{F}^c$, $\mathcal{F}^{uu}$, $\mathcal{F}^{s}$, and $\mathcal{F}^{u}$ of $\qut$. We refer to the corresponding subbundles $E^{ss}$, $E^c$, $E^{uu}$, $E^{s}$, and $E^{u}$ as the strongly stable, central, strongly unstable, stable, and unstable subbundles of $\qutp$. 

\subsection*{The leafwise measures of $\boldsymbol{\qut}$ revisited.} For every $q \in \qut(\mathbf{1})$, the strongly stable subspace $E^{ss}(q) \subseteq H^1_\odd(\widetilde{X};i\mathbf{R})$ supports a canonical volume form defined as follows. Consider the volume form on $H^1_\odd(\widetilde{X};i\mathbf{R})$ induced by the restriction of the cup intersection form. Contract this volume form by any vector $V \in H^1_\odd(\widetilde{X};i\mathbf{R})$ such that $d\text{Area}_q(V) = 1$. For example, take $V := i \Im([\omega])$. Denote by  $\mu_{E^{ss}(q)}$ the restriction of this contraction to $E^{ss}(q)$. These volume forms integrate to smooth measures on the intersections $\alpha^{ss}(q) \cap \qut(\mathbf{1})$. One can check these measures are equal to the leafwise measures $\mu_{\alpha^{ss}(q)}$ up to a multiplicative constant depending only on $g$. An analogous description can be provided for the leafwise measures $\mu_{\alpha^{uu}(q)}$. 

Given $q \in \qut(\mathbf{1})$, denote by $t$ the $\mathbf{R}$-coordinate of the central subspace $E^c(q) := \mathbf{R} \cdot [\overline{\omega}]$. This subspace supports a canonical volume form $\mu_{\alpha^c(q)} := dt$. The canonical volume forms $\mu_{E^{s}(q)} := \mu_{E^{ss}(q)} \wedge dt$ on the stable subspaces $E^{s}(q)$ integrate to smooth measures on the intersections $\alpha^{s}(q) \cap \qut(\mathbf{1})$. One can check these measures are equal to the leafwise measures $\mu_{\alpha^{s}(q)}$ up to a multiplicative constant depending only on $g$. An analogous description can be provided for the leafwise measures $\mu_{\alpha^{u}(q)}$.

Given $q \in \qut(\mathbf{1})$, denote by $\mu_q$ the Masur-Veech volume form on the tangent space $T_q\qutp$. We normalize the Masur-Veech volume form so that it admits the following decomposition \cite[\S 4]{AG13}:
\[
\mu_q = \mu_{E^{ss}(q)} \wedge \mu_{E^c(q)} \wedge \mu_{E^{uu}(q)}.
\]

\subsection*{The Hubbard-Masur functions.} In \cite{HM79}, Hubbard and Masur proved the following theorem: For every $\eta \in \mf$, the projection $\pi \colon \qt \to \tt$ is a homeomorphism onto $\tt$ when restricted to $\Re^{-1}(\eta)$ and a diffeomorphism onto its image when restricted to $\Re^{-1}(\eta) \cap \qt(\mathbf{1})$. In particular, for every $q \in \qut$, the projection $\pi \colon \qut \to \tt$ is a homeomorphism onto $\tt$ when restricted to $\alpha^{s}(q)$ and a diffeomorphism onto its image when restricted to $\alpha^{s}(q) \cap \qut(\mathbf{1})$. Analogous results hold for unstable leaves of $\qut$. Recall that $\mathbf{m} := \pi_* \mu$ denotes the pushforward to $\tt$ of the Masur-Veech measure $\mu$ on $\qut$. This measure is $\mcg$-invariant and smooth. Given $X \in \tt$, denote by $\mathbf{m}_X$ the corresponding volume form on $T_X\tt$. The Hubbard-Masur functions, introduced in \cite[\S 2.3]{ABEM12}, are the unique smooth, positive functions $\lambda^-,\lambda^+ \colon \qut(\mathbf{1}) \to \mathbf{R}_{>0}$ such that, for every $q \in \qutp$,
\begin{align}
|\mathbf{m}_{\pi(q)}| &= \lambda^{-}(q) \cdot |\pi_*(\mu_{E^{s}(q)})|, \label{eq:hm_1}\\
|\mathbf{m}_{\pi(q)}| &= \lambda^{+}(q) \cdot |\pi_*(\mu_{E^{u}(q)})|. \label{eq:hm_2}
\end{align}
Directly from these definitions, one can check that the $\lambda^-$ and $\lambda^+$ are $\mcg$-invariant.

\subsection*{The fiberwise measures of $\boldsymbol{\qut}$.} General measure theory considerations allow one to disintegrate the Masur-Veech measure $\mu$  on $\qut$ along the fibers of the projection $\pi \colon \mathcal{Q}^1\mathcal{T}_g \to \mathcal{T}_g$. More precisely, there exists a unique family of fiberwise probability measures $\{s_X\}_{X \in \mathcal{T}_g}$, with $s_X$ supported on $S(X) := \pi^{-1}(X)$, such that the following disintegration formula holds:
\begin{equation}
\label{eq:fiberwise_measures}
d\mu(X,q) = ds_X(q) \thinspace d\mathbf{m}(X).
\end{equation}
The fiberwise measures $\{s_X\}_{X \in \mathcal{T}_g}$ are smooth and one can make sense of (\ref{eq:fiberwise_measures}) at the level of volume forms. See \S 5 for a construction of fiberwise volume forms inducing these measures. By \cite[Theorem 2.2]{ABEM12}, the fiberwise measures $\{s_X\}_{X \in \mathcal{T}_g}$ give zero mass to the multiple zero locus.

Denote by $\widehat{\mathbf{m}}$ the local pushforward to $\mm$ of the measure $\mathbf{m}$ on $\tt$. If $g = 2$, the quotient map $p \colon \qut \to \qum$ is an isomorphism on fibers. If $g > 2$, the quotient map $p \colon \qut \to \qum$ is an isomorphism on fibers lying above Riemann surfaces without automorphisms. As this is a full measure subset of the Lebesgue measure class, the following disintegration formula holds:
\begin{equation}
\label{eq:disintegrate}
d\widehat{\mu}(X,q) = ds_X(q) \thinspace d\widehat{\mathbf{m}}(X).
\end{equation}

\subsection*{The Hubbard-Masur functions revisited.} The Hubbard-Masur functions $\lambda^-,\lambda^+ \colon \qut(\mathbf{1}) \to \mathbf{R}_{>0}$ were defined in (\ref{eq:hm_1}) and (\ref{eq:hm_2}) using leafwise measures. These functions can also be defined using the fiberwise measures $\{s_X\}_{X \in \mathcal{T}_g}$, as we now explain. Fix $X \in \tt$. By work of Hubbard and Masur \cite{HM79}, the maps $\Re,\Im \colon \qt\to \mf$ are homeomorphisms onto $\mf$ when restricted to $Q(X)$. Moreover, these maps are piecewise $\mathcal{C}^1$ isomorphisms onto their image when restricted to $Q(X) \cap \qt(\mathbf{1})$; see \cite[Lemma 4.3]{Mir08a} or \cite[Theorem 5.8]{Du15} for a proof. These restrictions map area of quadratic differentials to extremal length with respect to $X$ of singular measured foliations \cite{Ker80}. Denote by $\text{Ext}_\eta(X) $ the extremal length of $\eta \in \mf$ with respect to $X \in \tt$. Consider the subset $E_X \subseteq \mathcal{MF}_g$ given by
\[
E_X := \{\eta \in \mathcal{MF}_g \ | \ \text{Ext}_\eta(X) =  1 \}.
\]
Denote by $\overline{\nu}_X$ the conned-off version of the Thurston measure on $E_X$. The pullback measures $(\Re|_{S(X)})^*\overline{\nu}_X$ and $(\Im|_{S(X)})^*\overline{\nu}_X$ are smooth on $S(X) \cap \mathcal{Q}^1\mathcal{T}_g(\mathbf{1})$ and give zero mass to the multiple-zero locus \cite[Lemma 4.3]{Mir08a}. By \cite[Proposition 2.3]{ABEM12}, the Hubbard-Masur functions $\lambda^-,\lambda^+ \colon \qutp \to \mathbf{R}_{>0}$ are the unique smooth, positive functions such that for every $q \in \qutp$,
\begin{align}
|((\Re|_{S(X)})^*\overline{\nu}_X)(q)| &= \lambda^-(q) \cdot  |s_X(q)|, \label{eq:hm_3}\\
|((\Im|_{S(X)})^*\overline{\nu}_X)(q)| &= \lambda^+(q) \cdot |s_X(q)|. \label{eq:hm_4}
\end{align}

\subsection*{The Hubbard-Masur constant.} Consider the function $\Lambda \colon \mathcal{T}_g \to \mathbf{R}_{>0}$ which maps every $X \in \tt$ to
\[
\Lambda(X) := \nu(\{\eta \in \mathcal{MF}_g \ | \ \mathrm{Ext}_\eta(X) \leq 1 \}).
\]
Directly from (\ref{eq:hm_3}) and (\ref{eq:hm_4}) we see that, for every $X \in \tt$, 
\begin{align*}
\Lambda(X) &= \int_{S(X)} d((\Re|_{S(X)})^*\overline{\nu}_X)(q) = \int_{S(X)} \lambda^-(q) \ ds_X(q)\\
&= \int_{S(X)} d((\Im|_{S(X)})^*\overline{\nu}_X)(q) = \int_{S(X)} \lambda^+(q) \ ds_X(q).
\end{align*}
Combining results of Dumas \cite{Du15} with Gardiner's variational formula \cite{Ga84}, Mirzakhani showed the function $\Lambda \colon \tt \to \mathbf{R}_{>0}$ is constant \cite[Theorem 5.10]{Du15}. We denote its value by $\Lambda_g > 0$ and refer to it as the Hubbard-Masur constant. 

\subsection*{The Teichmüller metric.} Denote by $d_\mathcal{T}$ the Teichmüller metric on $\mathcal{T}_g$. This metric is Finsler. Denote by $\|\cdot \|_{\mathcal{T}}$ the corresponding family of fiberwise norms on $T\mathcal{T}_g$. The unit speed geodesics of this metric are precisely the paths $t \mapsto \pi(a_t q)$ with $q \in \qut$ arbitrary. The Teichmüller metric is complete. Indeed, by Teichmüller's theorem \cite[Theorem 11.19]{FM11}, any two points $X,Y \in \tt$ can be joined by a unique Teichmüller geodesic segment. Denote by $B_R(X) \subseteq \tt$ the ball of radius $R>0$ centered at $X \in \tt$ with respect to the Teichmüller metric. Kerckhoff formula \cite[Theorem 4]{Ker80} provides a convenient tool for estimating Teichmüller distances: For every $X,Y \in \mathcal{T}_g$,
\begin{equation}
\label{eq:ker}
d_\mathcal{T}(X,Y) = \sup_{\eta \in \mf} \frac{\sqrt{\mathrm{Ext}_{\eta}(Y)}}{\sqrt{\mathrm{Ext}_{\eta}(X)}}.
\end{equation}

\subsection*{The Hodge inner product.} Let $X$ be a closed Riemann surface.  Consider the Hodge decomposition of its complex cohomology group
$
H^1(X;\mathbf{C}) = H^{1,0}(X) \oplus H^{0,1}(X)
$
into holomorphic and anti-holomorphic $1$-forms. On $H^1(X;\mathbf{C})$ consider the Hodge intersection pairing
$
(\alpha, \beta)_X:= \frac{i}{2} \ \int_X \alpha \wedge \overline{\beta}.
$
This pairing is Hermitian, positive definite on $H^{1,0}(X)$, and negative definite on $H^{0,1}(X)$. 

The Hodge inner product $\langle \cdot , \cdot \rangle_X$ on $H^1(X;\mathbf{C})$ is the unique Hermitian inner product given by the Hodge intersection pairing on $H^{1,0}(X)$, the negative of the Hodge intersection pairing on $H^{0,1}(X)$, and which makes $H^{1,0}(X)$ and $H^{0,1}(X)$ orthogonal. A direct computation shows that, for every $\omega \in H^{1,0}(X)$,
$\langle \omega, \omega \rangle_X = \text{Area}(\omega^2)$. The Hodge inner product  restricts to a real inner product on $H^1(X;\mathbf{R})$.

\subsection*{The Hodge star operator.} Let $X$ be a closed Riemann surface. The Hodge star operator
$
\star \colon H^1(X;\mathbf{C}) \to H^1(X;\mathbf{C}) 
$
is the unique complex linear operator acting as $(-i)$ on $H^{1,0}(X)$ and as $i$ on $H^{0,1}(X)$. Notice that $\star^2 = -1$ and that $\star$ commutes with complex conjugation of $1$-forms. The Hodge star operator preserves the subspaces $H^1(X;\mathbf{R})$ and $H^1(X;i\mathbf{R})$. For every $\omega \in H^{1,0}(X)$, $\star \Re(\omega) = \Im(\omega)$ and $\star \Im(\omega) = - \Re(\omega)$.

Recall that we normalize the cup intersection pairing on $H^1(X;\mathbf{C})$ so that
$
\alpha \cup \beta = \frac{1}{2} \ \int_X \alpha \wedge \beta.
$
The composition of the Hodge star operator with complex conjugation of $1$-forms conjugates the Hodge inner product, defined using the complex structure of $X$, to the cup intersection pairing, defined using only the topology of $X$. More precisely, 
$
\langle \alpha, \beta \rangle_X = \alpha \cup \overline{\star \beta}
$
on $H^1(X;\mathbf{C})$.
In particular, 
$
\langle \alpha, \beta \rangle_X = \alpha \cup \star \beta
$
on $H^1(X;\mathbf{R})$.

From this property we see that the notions of simplectic orthogonality with respect to the cup intersection pairing and orthogonality with respect to the Hodge inner product on $H^1(X;\mathbf{R})$ coincide. 
Using the same property, one can check that the Hodge star operator on $H^1(X;\mathbf{R})$ is antisymmetric and orthogonal with respect to the Hodge inner product. In particular, the volume forms induced by the cup intersection pairing and the Hodge inner product on $H^1(X;\mathbf{R})$ coincide.

In this paper we will specifically consider the restriction of the Hodge inner product to odd subspaces $H^1_\text{odd}(\widetilde{X};\mathbf{R}) \subseteq H^1(\widetilde{X};\mathbf{R})$ arising from canonical double covers $\widetilde{X}$ of quadratic differentials in $\qutp$. The properties introduced above remain true in this setting. Given a subspace $V \subseteq H^1_\text{odd}(\widetilde{X};\mathbf{R})$ we will denote by $V^\perp \subseteq H^1_\text{odd}(\widetilde{X};\mathbf{R})$ its orthogonal complement.

\subsection*{Constants.} Throughout this paper we will need to keep track of several constants. Constants will always be denoted with a subindex to help the reader keep track of when and where they are introduced. We will always make the dependencies of the constants explicit. 

Let $A,B \in \mathbf{R}$ be real quantities and $*$ be a set of parameters. We write $A \preceq_* B$ if there exists a constant $c= c(*) > 0$ depending only on $*$ such that $A \leq c \cdot B$. We write $A \asymp_* B$ if $A \preceq_* B $ and $B \preceq_* A$. We write $A = O_*(B)$ if there exists a constant $c = c(*) > 0$ depending only on $*$ such that $|A| \leq c \cdot B$.

\section{Effective mean equidistribution of Teichmüller metric balls}

\subsection*{Outline of this section.} The goal of this section is to prove Theorem \ref{theo:ball_equidistribution_intro}. We actually prove a slightly more precise version, which we introduce below as Theorem \ref{theo:ball_equidistribution}. The exponential mixing property of the Teichmüller geodesic flow is the main tool used in the proof. The proof also uses results from \S 5, \S 6, and \S 7. We summarize these results in this section an defer their proofs to \S 5, \S 6, and \S 7.

\subsection*{Statement of the main theorem.} Recall that $\mathbf{m} := \pi_* \mu$ denotes the pushforward to $\tt$ of the Masur-Veech measure $\mu$ on $\qut$ under the projection $\pi \colon \qut \to \tt$. Recall that $\widehat{\mathbf{m}}$ denotes the local pushforward of $\mathbf{m}$ to $\mm$. Recall that $B_R(X) \subseteq \mathcal{T}_g$ denotes the  ball of radius $R>0$ centered at $X \in \tt$ with respect to the Teichmüller metric. Recall that  $\mathbf{m}_X^R$ denotes the restriction of $\mathbf{m}$ to $B_R(X)$ and that $\widehat{\mathbf{m}}_X^R$ denotes the pushforward of $\mathbf{m}_X^R$ to $\mathcal{M}_g$. Recall that $\widehat{\mathbf{m}}_X^R$ depends only on $X \in \mm$ and not on the marking. Recall that $h := 6g-6$ and that $\Lambda_g > 0$ denotes the Hubbard-Masur constant introduced in \S 2. In this section we prove the following version of Theorem \ref{theo:ball_equidistribution_intro}. 

\newconk{ball_equid}
\begin{theorem}
	\label{theo:ball_equidistribution}
	Let $\mathcal{K} \subseteq \mathcal{M}_g$ be a compact subset and $\phi_1,\phi_2 \in  L^\infty(\mm,\widehat{\m})$ be essentially bounded functions with  $\esupp(\phi_1), \esupp(\phi_2) \subseteq \mathcal{K}$. Then, for every $R > 0$,
	\begin{gather*}
	\int_{\mm} \phi_1(X) \left( \int_{\mm} \phi_2(Y) \ d\widehat{\mathbf{m}}_X^R(Y) \right) d\widehat{\mathbf{m}}(X) \\
	=  \frac{\Lambda_g^2}{h \cdot \widehat{\mathbf{m}}(\mathcal{M}_g)} \cdot \left(\int_{\mm} \phi_1(X) \ d\widehat{\mathbf{m}}(X)\right) \cdot \left(\int_{\mm} \phi_2(Y) \ d\widehat{\mathbf{m}}(Y) \right) \cdot e^{hR} \\
	+ O_\mathcal{K}\left(\|\phi_1\|_\infty \cdot \|\phi_2\|_\infty \cdot e^{(h-\useconk{ball_equid})R}\right), 
	\end{gather*}
	where $\useconk{ball_equid} =\useconk{ball_equid}(g) > 0$ is a constant depending only on $g$.
\end{theorem}

\subsection*{The exponential mixing property of the Teichmüller geodesic flow.} The backbone of the proof of Theorem \ref{theo:ball_equidistribution} is the exponential mixing property of the Teichmüller geodesic flow on $\qum$. Denote by $\mathbf{S}^1 \subseteq \mathbf{C}$ the unit circle. Recall that $\widehat{\mu}$ denotes the Masur-Veech measure on $\qum$. The Ratner class of observables $\mathcal{R}(\qum,\widehat{\mu}) \subseteq L^2(\qum,\widehat{\mu})$ is the set of functions $\varphi \in L^2(\qum,\widehat{\mu})$ such that the map $e^{i\theta} \in \mathbf{S}^1 \mapsto \varphi \circ r_\theta \in L^2(\qum,\widehat{\mu})$ is Lipschitz. Denote by $\|\varphi\|_{\mathrm{Lip}(\widehat{\mu})}$ the minimal Lipschitz constant of such map. On $\mathcal{R}(\qum,\widehat{\mu})$ consider the norm
\[
\|\varphi\|_{\mathcal{R}(\widehat{\mu})} := \|\varphi\|_{L^2(\widehat{\mu})} + \| \varphi \|_{\mathrm{Lip}(\widehat{\mu})}.
\]

The following theorem was proved by Avila and Resende \cite[Theorem 1.1]{AR12}, building on previous work of Avila, Gouëzel, and Yocozz \cite{AGY06}. More general results were later proved by Avila and Gouëzel \cite{AG13}. Ratner's connection between the spectral gap of a unitary $\mathrm{SL}(2,\mathbf{R})$ representation and the exponential mixing rate of its diagonal flows \cite{Ra87} is one of the key ideas behind these results. See \cite{Ma13} for a detailed treatment of Ratner's work. See \cite[Proof of Corollary 1]{Do98} for an explanation on how to reduce from $\mathrm{SO}(2)$-smooth observables, as in Ratner's work, to $\mathrm{SO}(2)$-Lipschitz observables, as in the following statement.

\newconk{exp_mix}
\begin{theorem}[\cite{AR12}]
	\label{theo:exp_mixing}
	Let $\varphi_1,\varphi_2 \in\mathcal{R}(\qum,\widehat{\mu})$ be a pair of Ratner observables. Then, for every $t > 0$,
	\begin{align*}
	\int_{\qum} \varphi_2(a_t q) \thinspace \varphi_1(q) \thinspace d\widehat{\mu}(q) &= \frac{1}{\widehat{\mu}(\qum)} \cdot \left(\int_{\qum} \varphi_1(q) \thinspace d\widehat{\mu}(q)\right) \cdot \left(\int_{\qum} \varphi_2(q) \thinspace d\widehat{\mu}(q)\right) \\
	&\phantom{= } + O_g\left(\|\varphi_1\|_{\mathcal{R}(\widehat{\mu})} \cdot \| \varphi_2\|_{\mathcal{R}(\widehat{\mu})} \cdot e^{-\useconk{exp_mix} t}\right),
	\end{align*}
	where $\useconk{exp_mix} = \useconk{exp_mix}(g) > 0$ is a constant depending only on $g$.
\end{theorem}

\subsection*{Regularity of the Hubbard-Masur functions.} The Hubbard-Masur functions $\lambda^-,\lambda^+ \colon \qutp \to \mathbf{R}_{>0}$ introduced in \S 2 play a crucial role in the proof of Theorem \ref{theo:ball_equidistribution}. In particular, it will be important to understand the regularity of these functions. The results we now summarize will be proved in \S5 

Using work of Dumas \cite{Du15} we prove the following invariance property. This property will be crucial to apply Theorem \ref{theo:exp_mixing} in the proof of Theorem \ref{theo:ball_equidistribution}.

\begin{proposition}
	\label{prop:hm_inv_sum}
	The functions $\lambda^-,\lambda^+ \colon \mathcal{Q}^1\mathcal{T}_g(\mathbf{1}) \to \mathbf{R}_{>0}$ are $\mathrm{SO}(2)$-invariant.
\end{proposition}

The following closely related property, although not strictly needed in the proof of Theorem \ref{theo:ball_equidistribution}, conveniently simplifies the notation for the rest of this paper.

\begin{proposition}
	\label{prop:one_hm_function_sum}
	The functions $\lambda^-,\lambda^+ \colon \mathcal{Q}^1\mathcal{T}_g(\mathbf{1}) \to \mathbf{R}_{>0}$ coincide everywhere.
\end{proposition}

Following Proposition \ref{prop:one_hm_function_sum}, we denote the functions $\lambda^-,\lambda^+$ by $\lambda \colon \qut(\mathbf{1}) \to \mathbf{R}_{>0}$ and refer to $\lambda$ as the Hubbard-Masur function. Using work of Kahn and Wright \cite{KW20} we prove estimates for $\lambda(q)$ in terms of $\ell_{\min}(q)$, the length of the shortest saddle connection of $q \in \qutp$. 

\begin{proposition}
	\label{prop:HM_bound_sum}
	Let $\mathcal{K} \subseteq \mathcal{T}_g$ be a compact subset. Then, for every $q \in \mathcal{Q}^1\mathcal{T}_g(\mathbf{1}) \cap \pi^{-1}(\mathcal{K})$,
	\[
	\lambda(q) \preceq_{\mathcal{K}} \ell_{\min}(q)^{-(h-1)}.
	\]
\end{proposition}

Recall that, for every $\delta > 0$, $K_\delta(\mathbf{1}) \subseteq \qum(\mathbf{1})$ denotes the intersection of the principal stratum with
\[
K_\delta := \{q \in \qum \ | \ \ell_{\min}(q) \geq \delta \}.
\]
Recall that $p \colon \qut \to \qum$ denotes the quotient map. A computation in train track coordinates will yield the following important estimate.

\begin{proposition}
	\label{prop:small_HM_integral_sum}
	Let $\mathcal{K} \subseteq \mathcal{T}_g$ be a compact subset. Then, for every $X \in \mathcal{K}$ and every $\delta > 0$,
	\[
	\int_{S(X) \backslash p^{-1}(K_\delta(\mathbf{1}))} \lambda(q) \thinspace ds_X(q) \preceq_{\mathcal{K}} \delta.
	\]
\end{proposition}

\subsection*{The Masur-Veech measure in polar coordinates.} Fix $X \in \tt$. Recall that $\Phi_X \colon S(X) \times \mathbf{R}_{>0} \to \mathcal{T}_g$ denotes the map $\Phi_X(q,t) := \pi(a_tq)$ and that, for every $q \in S(X) \cap \qutp$ and every $t > 0$, 
\begin{equation}
\label{eq:polar}
|\Phi_X^*(\mathbf{m})(q,t)| = \Delta(q,t) \cdot |s_X(q) \wedge dt|,
\end{equation}
where $\Delta \colon (S(X) \cap \qutp) \times \mathbf{R}_{>0} \to \mathbf{R}_{>0}$ is a smooth, positive function. Given $\epsilon > 0$ and $t>0$, denote
\[
K_\epsilon(t) := \{q \in \qum  \colon |\{s \in [0,t] \ | \ a_s q \in K_\epsilon \}| \geq (1/2) \thinspace t \}.
\]
A crucial tool used in the proof of Theorem \ref{theo:ball_equidistribution} is the following estimate for the function $\Delta$. This estimate, which corresponds to a strenghtening of \cite[Proposition 2.5]{ABEM12}, will be proved in \S 6. The proof uses the bounds of Forni on the spectral gap of the Kontsevich-Zorich cocycle on $\qutp$ \cite{F02} together with the analysis of the projection $\pi \colon \qut \to \tt$ in forthcoming work of Kahn and Wright \cite{KW20}

\newconc{polar}
\begin{theorem}
	\label{theo:polar_coordinates_estimate}
	Let $\mathcal{K} \subseteq \mathcal{T}_g$ be a compact subset. Then, for every $q \in \mathcal{Q}^1 \mathcal{T}_g(\mathbf{1}) \cap \pi^{-1}(\mathcal{K})$, every $t > 0$ such that $a_t q \in \mcg \cdot \pi^{-1}(\mathcal{K})$, and every $\epsilon > 0$ such that $q \in p^{-1}(K_\epsilon(t))$, 
	\[
	\Delta(q,t) = \lambda(a_tq)  \cdot \lambda(q) \cdot e^{ht} + O_\mathcal{K}\left( \ell_{\min}(a_t q)^{-(h-1)} \cdot \ell_{\min}(q)^{-(h-1)} \cdot  e^{(h-\useconc{polar}) t}\right),
	\]
	where $\useconc{polar} = \useconc{polar}(g,\epsilon) > 0$ is a constant depending only on $g$ and $\epsilon$.
\end{theorem}

\subsection*{Estimates near the multiple zero locus.} As the error term in Theorem \ref{theo:polar_coordinates_estimate} gets worse as $q \in \qutp$ approaches the multiple zero locus, different arguments need to be considered to bound such contributions. We now summarize the relevant estimates. These estimates will be proved in \S 7.

Given $X \in \mathcal{T}_g$, $R>0$, $\mathcal{K} \subseteq \mathcal{T}_g$ compact, and $\epsilon > 0$, denote by $B_R(X,\mathcal{K},K_\epsilon) \subseteq \mathcal{T}_g$ the subset of points $Y \in B_R(X) \cap \mathrm{Mod}_g \cdot \mathcal{K}$ such that the unique Teichmüller geodesic segment from $X$ to $Y$ spends less that half of the time in $p^{-1}(K_\epsilon)$. The following estimate, which corresponds to an effective version of \cite[Theorem 2.7]{ABEM12}, is proved using results of Eskin, Mirzakhani, and Rafi \cite{EMR19}.

\newcone{thin_traj_sum} \newconk{thin_traj_sum_k}
\begin{theorem}
	\label{theo:thin_trajectories_sum}
	There exists a constant $\usecone{thin_traj_sum} = \usecone{thin_traj_sum}(g)> 0$ depending only on $g$ such that for every compact subset $\mathcal{K} \subseteq \mathcal{T}_g$, every $X \in \mathcal{K}$, and every $0 < \epsilon < \usecone{thin_traj_sum}$,
	\[
	\mathbf{m}\left(B_R(X,\mathcal{K},K_\epsilon)\right) \preceq_{\mathcal{K}} e^{(h-\useconk{thin_traj_sum_k})R},
	\]
	where $\useconk{thin_traj_sum_k} = \useconk{thin_traj_sum_k}(g) >0$ is a constant depending only on $g$.
\end{theorem}

Consider the function $q_s \colon \mathcal{T}_g \times \mathcal{T}_g \to \mathcal{Q}^1\mathcal{T}_g$ which to every pair $X,Y \in \mathcal{T}_g$ assigns the quadratic differential $q_s(X,Y) \in S(X)$ corresponding to the tangent direction at $X$ of the unique Teichmüller geodesic segment from $X$ to $Y$. For every $X \in \mathcal{T}_g$ and every $V \subseteq S(X)$ consider the sector
\[
\text{Sect}_V(X) := \{Y \in \mathcal{T}_g \ | \ q_s(X,Y) \in V \}.
\] 
The following estimate, which corresponds to an effective version of \cite[Theorem 2.6]{ABEM12}, bounds the measure of thin sectors containing the multiple zero locus. The proof uses Theorem \ref{theo:thin_trajectories_sum} together with the exponential contraction rate proved in \cite{ABEM12} for the modified Hodge metric along the strongly stable leaves of $\qut$ .

\newconk{thin_sect_sum}
\begin{theorem}
	\label{theo:thin_sector_sum}
	Let $\mathcal{K} \subseteq \mathcal{T}_g$ compact, $X \in \mathcal{K}$, $\delta > 0$, and $V := p^{-1}(K_\delta(\mathbf{1})) \cap S(X)$. Then, for every $R > 0$,
	\[
	\mathbf{m}\left(B_R(X) \cap \mathrm{Sect}_{V}(X) \cap \mathrm{Mod}_g \cdot \mathcal{K}\right) 
	\preceq_{\mathcal{K}} \delta \cdot e^{hR} + e^{(h - \useconk{thin_sect_sum}) R},
	\]
	where $\useconk{thin_sect_sum} = \useconk{thin_sect_sum}(g) >0$ is a constant depending only on $g$.
\end{theorem}

The following large deviations estimate is a direct consequence of work of Athreya \cite{A06}.

\newcone{large_dev} \newconk{large_dev_k}
\begin{theorem}
	\label{theo:large_deviations_sum}
	There exists a constant $\usecone{large_dev} = \usecone{large_dev}(g) > 0$ such that for every $0 < \epsilon < \usecone{large_dev}$ and every $t > 0$,
	\[
	\widehat{\mu}\left(\qum \backslash K_\epsilon(t)\right) \preceq_g e^{-\useconk{large_dev_k} t},
	\]
	where $\useconk{large_dev_k} = \useconk{large_dev_k}(g) > 0$ is a constant depending only on $g$.
\end{theorem}

\subsection*{Proof of Theorem \ref{theo:ball_equidistribution}} We are now ready to prove Theorem \ref{theo:ball_equidistribution}. 

\begin{proof}[Proof of Theorem \ref{theo:ball_equidistribution}.]
	Fix a compact subset  $\mathcal{K} \subseteq \mathcal{M}_g$ and a pair of essentially bounded functions $\phi_1,\phi_2 \in  L^\infty(\mm,\widehat{\m})$ with  $\esupp(\phi_1), \esupp(\phi_2) \subseteq \mathcal{K}$. For simplicity we will also denote by $\phi_1,\phi_2 \in L^\infty(\tt,\m)$ the lifts of these functions to $\tt$. Let $R > 0$ and $0 < \delta = \delta(R) < 1$, to be fixed later. Recall that $\widehat{\mathbf{m}}_X^R$ is the pushforward to $\mm$ of the measure $\mathbf{m}_X^R$ on $\tt$. Using (\ref{eq:polar}) we can write
	\begin{gather}
	\int_{\mathcal{M}_g} \phi_1(X) \left( \int_{\mathcal{M}_g} \phi_2(Y) \thinspace d\widehat{\mathbf{m}}_X^R(Y) \right) d\widehat{\mathbf{m}}(X) \label{eq:first}\\
	= \int_{\mathcal{M}_g} \phi_1(X) \left( \int_0^R \int_{S(X)} \phi_2(\pi(a_tq)) \thinspace \Delta(q,t) \thinspace ds_X(q) \thinspace dt \right) d\widehat{\mathbf{m}}(X). \nonumber
	\end{gather}
	To apply Theorem \ref{theo:polar_coordinates_estimate} we first bound the contributions near the multiple zero locus. By Theorem \ref{theo:thin_sector_sum},
	\begin{gather}
	\int_{\mathcal{M}_g} \phi_1(X) \left( \int_0^R \int_{S(X)} \phi_2(\pi(a_tq)) \thinspace \Delta(q,t) \thinspace ds_X(q) \thinspace dt \right) d\widehat{\mathbf{m}}(X) \label{eq:second}\\
	= \int_{\mathcal{M}_g} \phi_1(X) \left( \int_0^R \int_{S(X)} \phi_2(\pi(a_tq)) \thinspace \mathbbm{1}_{K_\delta(\mathbf{1})}(q) \thinspace \Delta(q,t) \thinspace ds_X(q) \thinspace dt \right) d\widehat{\mathbf{m}}(X) \nonumber  \\
	+ O_\mathcal{K}\left(\|\phi_1\|_\infty \cdot \|\phi_2\|_\infty \cdot \left(\delta \cdot e^{hR} + e^{(h - \useconk{thin_sect_sum}) R}\right)\right). \nonumber
	\end{gather}
	A symmetric argument using Fubini's theorem and Theorem \ref{theo:thin_sector_sum} shows that
	\begin{gather}
	\int_{\mathcal{M}_g} \phi_1(X) \left( \int_0^R \int_{S(X)} \phi_2(\pi(a_tq)) \thinspace \mathbbm{1}_{K_\delta(\mathbf{1})}(q) \thinspace \Delta(q,t) \thinspace ds_X(q) \thinspace dt \right) d\widehat{\mathbf{m}}(X) \label{eq:third} \\
	= \int_{\mathcal{M}_g} \phi_1(X) \left( \int_0^R \int_{S(X)} \phi_2(\pi(a_tq)) \thinspace \mathbbm{1}_{K_\delta(\mathbf{1})}(a_t q) \thinspace \mathbbm{1}_{K_\delta(\mathbf{1})}(q) \thinspace \Delta(q,t) \thinspace ds_X(q) \thinspace dt \right) d\widehat{\mathbf{m}}(X) \nonumber \\
	+ O_\mathcal{K}\left(\|\phi_1\|_\infty \cdot \|\phi_2\|_\infty \cdot \left(\delta \cdot e^{hR} + e^{(h - \useconk{thin_sect_sum}) R}\right)\right). \nonumber
	\end{gather}
	Let $\usecone{thin_traj_sum} = \usecone{thin_traj_sum}(g) > 0$ be as in Theorem \ref{theo:thin_trajectories_sum} and $\usecone{large_dev} = \usecone{large_dev}(g) > 0$ be as in Theorem \ref{theo:large_deviations_sum}. Fix an arbitrary $0 < \epsilon = \epsilon(g) < \min\{\usecone{thin_traj_sum},\usecone{large_dev}\}$. By Theorem \ref{theo:thin_trajectories_sum}, 
	\begin{gather}
	\int_{\mathcal{M}_g} \phi_1(X) \left( \int_0^R \int_{S(X)} \phi_2(\pi(a_tq)) \thinspace \mathbbm{1}_{K_\delta(\mathbf{1})}(a_t q) \thinspace \mathbbm{1}_{K_\delta(\mathbf{1})}(q) \thinspace \Delta(q,t) \thinspace ds_X(q) \thinspace dt \right) d\widehat{\mathbf{m}}(X) \label{eq:fourth}\\
	= \int_{\mathcal{M}_g} \phi_1(X) \left( \int_0^R \int_{S(X)} \phi_2(\pi(a_tq)) \thinspace \mathbbm{1}_{K_\delta(\mathbf{1})}(a_t q) \thinspace \mathbbm{1}_{K_\delta(\mathbf{1})}(q) \thinspace \mathbbm{1}_{K_\epsilon(t)}(q) \thinspace  \Delta(q,t) \thinspace ds_X(q) \thinspace dt \right) d\widehat{\mathbf{m}}(X) \nonumber\\
	+ O_{\mathcal{K}}\left(\|\phi_1\|_\infty \cdot \|\phi_2\|_\infty \cdot e^{(h-\useconk{thin_traj_sum_k})R}\right). \nonumber
	\end{gather}
	Let $\kappa = \kappa(g) := \min\{\useconk{thin_sect_sum}, \useconk{thin_traj_sum_k}\} > 0$. Putting (\ref{eq:first}), (\ref{eq:second}), (\ref{eq:third}), and (\ref{eq:fourth}) together we deduce
	\begin{gather}
	\int_{\mathcal{M}_g} \phi_1(X) \left( \int_{\mathcal{M}_g} \phi_2(Y) \thinspace d\widehat{\mathbf{m}}_X^R(Y) \right) d\widehat{\mathbf{m}}(X) \label{eq:stop1}\\
	= \int_{\mathcal{M}_g} \phi_1(X) \left( \int_0^R e^{ht} \int_{S(X)} \phi_2(\pi(a_tq)) \thinspace \mathbbm{1}_{K_\delta(\mathbf{1})}(a_t q) \thinspace \lambda(a_tq) \thinspace \mathbbm{1}_{K_\delta(\mathbf{1})}(q) \thinspace \mathbbm{1}_{K_\epsilon(t)}(q) \thinspace \lambda(q) \thinspace ds_X(q) \thinspace dt \right) d\widehat{\mathbf{m}}(X) \nonumber \\
	+ O_\mathcal{K}\left(\|\phi_1\|_\infty \cdot \|\phi_2\|_\infty \cdot \left(\delta \cdot e^{hR} + e^{(h - \kappa) R}\right)\right). \nonumber
	\end{gather}
	We are now in a good position to apply Theorem \ref{theo:polar_coordinates_estimate}. By Theorem \ref{theo:polar_coordinates_estimate},
	\begin{gather}
	\int_{\mathcal{M}_g} \phi_1(X) \left( \int_0^R \int_{S(X)} \phi_2(\pi(a_tq)) \thinspace \mathbbm{1}_{K_\delta(\mathbf{1})}(a_t q) \thinspace \mathbbm{1}_{K_\delta(\mathbf{1})}(q) \thinspace \mathbbm{1}_{K_\epsilon(t)} \thinspace \Delta(q,t) \thinspace ds_X(q) \thinspace dt \right) d\widehat{\mathbf{m}}(X) \label{eq:zerob}\\
	= \int_{\mathcal{M}_g} \phi_1(X) \left( \int_0^R e^{ht} \int_{S(X)} \phi_2(\pi(a_tq)) \thinspace \mathbbm{1}_{K_\delta(\mathbf{1})}(a_t q) \thinspace \lambda(a_tq) \thinspace \mathbbm{1}_{K_\delta(\mathbf{1})}(q) \thinspace \mathbbm{1}_{K_\epsilon(t)}(q) \thinspace \lambda(q) \thinspace ds_X(q) \thinspace dt \right) d\widehat{\mathbf{m}}(X) \nonumber \\
	+ O_\mathcal{K}\left( \|\phi_1\|_\infty \cdot \|\phi_2\|_{\infty} \cdot \delta^{-2(h-1)} \cdot e^{(h-\useconc{polar}) R}\right). \nonumber
	\end{gather}
	Using Fubini's theorem and (\ref{eq:disintegrate}) we can write
	\begin{gather}
		\int_{\mathcal{M}_g} \phi_1(X) \left( \int_0^R e^{ht} \int_{S(X)} \phi_2(\pi(a_tq)) \thinspace \mathbbm{1}_{K_\delta(\mathbf{1})}(a_t q) \thinspace \lambda(a_tq) \thinspace \mathbbm{1}_{K_\delta(\mathbf{1})}(q) \thinspace \mathbbm{1}_{K_\epsilon(t)}(q) \thinspace \lambda(q) \thinspace ds_X(q) \thinspace dt \right) d\widehat{\mathbf{m}}(X) \label{eq:firstb}\\
		=\int_0^R e^{ht} \left(\int_{\mathcal{M}_g} \int_{S(X)} \phi_2(\pi(a_tq)) \thinspace \mathbbm{1}_{K_\delta(\mathbf{1})}(a_t q) \thinspace \lambda(a_tq) \thinspace \phi_1(X) \thinspace \mathbbm{1}_{K_\delta(\mathbf{1})}(q) \thinspace \mathbbm{1}_{K_\epsilon(t)}(q) \thinspace \lambda(q) \thinspace ds_X(q) \thinspace d\widehat{\mathbf{m}}(X)	\right) dt \nonumber\\
		=\int_0^R e^{ht} \left(\int_{\mathcal{Q}^1\mathcal{M}_g} \phi_2(\pi(a_tq)) \thinspace \mathbbm{1}_{K_\delta(\mathbf{1})}(a_t q) \thinspace \lambda(a_tq) \thinspace \phi_1(\pi(q)) \thinspace \mathbbm{1}_{K_\delta(\mathbf{1})}(q) \thinspace \mathbbm{1}_{K_\epsilon(t)}(q) \thinspace \lambda(q) \thinspace d\widehat{\mu}(q)	\right) dt. \nonumber
	\end{gather}
	To apply Theorem \ref{theo:exp_mixing} we first use Theorem \ref{theo:large_deviations_sum} to write
	\begin{gather}
	\int_0^R e^{ht} \left(\int_{\mathcal{Q}^1\mathcal{M}_g} \phi_2(\pi(a_tq)) \thinspace \mathbbm{1}_{K_\delta(\mathbf{1})}(a_t q) \thinspace \lambda(a_tq) \thinspace \phi_1(\pi(q)) \thinspace \mathbbm{1}_{K_\delta(\mathbf{1})}(q) \thinspace \mathbbm{1}_{K_\epsilon(t)}(q) \thinspace \lambda(q) \thinspace d\widehat{\mu}(q)	\right) dt \label{eq:secondb}\\
	= \int_0^R e^{ht} \left(\int_{\mathcal{Q}^1\mathcal{M}_g} \phi_2(\pi(a_tq)) \thinspace \mathbbm{1}_{K_\delta(\mathbf{1})}(a_t q) \thinspace \lambda(a_tq) \thinspace \phi_1(\pi(q)) \thinspace \mathbbm{1}_{K_\delta(\mathbf{1})}(q) \thinspace \lambda(q) \thinspace d\widehat{\mu}(q)	\right) dt \nonumber\\
	+ O_g\left(\|\phi_1\|_\infty \cdot \| \phi_2\|_\infty \cdot e^{(h-\useconk{large_dev_k})R}\right). \nonumber
	\end{gather}
	By Theorem \ref{theo:exp_mixing}, 
	\begin{gather}
	\int_{\mathcal{Q}^1\mathcal{M}_g} \phi_2(\pi(a_tq)) \thinspace \mathbbm{1}_{K_\delta(\mathbf{1})}(a_t q) \thinspace \lambda(a_tq) \thinspace \phi_1(\pi(q)) \thinspace \mathbbm{1}_{K_\delta(\mathbf{1})}(q) \thinspace d\widehat{\mu}(q) \label{eq:thirdb}\\
	= \frac{1}{\widehat{\mathbf{m}}(\mathcal{M}_g)} \cdot \left( \int_{\mathcal{Q}^1\mathcal{M}_g} \phi_2(\pi(q)) \thinspace \mathbbm{1}_{K_\delta(\mathbf{1})}(q) \thinspace \lambda(q) \thinspace d\widehat{\mu}(q) \right) \cdot \left( \int_{\mathcal{Q}^1\mathcal{M}_g} \phi_1(\pi(q)) \thinspace \mathbbm{1}_{K_\delta(\mathbf{1})}(q) \thinspace \lambda(q) \thinspace d\widehat{\mu}(q)\right) \nonumber\\
	+  O_g\left(\| (\phi_1 \circ \pi) \cdot \mathbbm{1}_{K_\delta(\mathbf{1})} \cdot \lambda\|_{\mathcal{R}} \cdot \|(\phi_2 \circ \pi) \cdot \mathbbm{1}_{K_\delta(\mathbf{1})} \cdot \lambda\|_{\mathcal{R}} \cdot e^{-\useconk{exp_mix}t}\right). \nonumber
	\end{gather}
	Using Propositions \ref{prop:hm_inv_sum} and \ref{prop:HM_bound_sum} we deduce
	\begin{align}
		\| (\phi_1 \circ \pi) \cdot \mathbbm{1}_{K_\delta(\mathbf{1})} \cdot \lambda\|_{\mathcal{R}} &\preceq_{g} \delta^{-(h-1)} \cdot \|\phi_1 \|_{L^\infty}, \label{eq:fourthb}\\
		\| (\phi_2 \circ \pi) \cdot \mathbbm{1}_{K_\delta(\mathbf{1})} \cdot \lambda\|_{\mathcal{R}} &\preceq_{g} \delta^{-(h-1)} \cdot \|\phi_2 \|_{L^\infty}. \label{eq:fifthb}
	\end{align}
	Let $\kappa' =  \kappa'(g) := \min\{\kappa,\useconk{large_dev_k},\useconk{exp_mix}\} > 0$. Using (\ref{eq:stop1}), (\ref{eq:zerob}), (\ref{eq:firstb}), (\ref{eq:secondb}), (\ref{eq:thirdb}), (\ref{eq:fourthb}), and (\ref{eq:fifthb}) we deduce
	\begin{gather}
	\int_{\mathcal{M}_g} \phi_1(X) \left( \int_{\mathcal{M}_g} \phi_2(Y) \thinspace d\widehat{\mathbf{m}}_X^R(Y) \right) d\widehat{\mathbf{m}}(X)  \label{eq:stop_2}\\
	= \frac{1}{h \cdot \widehat{\mathbf{m}}(\mathcal{M}_g)} \cdot \left( \int_{\mathcal{Q}^1\mathcal{M}_g} \phi_2(\pi(q)) \thinspace \mathbbm{1}_{K_\delta(\mathbf{1})}(q) \thinspace \lambda(q) \thinspace d\widehat{\mu}(q) \right) \cdot \left( \int_{\mathcal{Q}^1\mathcal{M}_g} \phi_1(\pi(q)) \thinspace \mathbbm{1}_{K_\delta(\mathbf{1})}(q) \thinspace \lambda(q) \thinspace d\widehat{\mu}(q)\right) \cdot e^{hR} \nonumber\\
	+  O_\mathcal{K}\left(\|\phi_1\|_\infty \cdot \|\phi_2\|_\infty \cdot \left( \delta \cdot e^{hR} + \delta^{-2(h-1)}  \cdot e^{(h - \kappa') R}\right)\right). \nonumber
	\end{gather}
	We now incorporate the contributions near the multiple zero locus to the leading term in (\ref{eq:stop_2}). By Proposition \ref{prop:small_HM_integral_sum}, for $i \in \{1,2\}$,
	\begin{align}
	\int_{\mathcal{Q}^1\mathcal{M}_g} \phi_i(\pi(q)) \thinspace \mathbbm{1}_{K_\delta(\mathbf{1})}(q) \thinspace\lambda(q) \thinspace d\widehat{\mu}(q) 
	=  \int_{\mathcal{Q}^1\mathcal{M}_g} \phi_i(\pi(q)) \thinspace \lambda(q) \thinspace d\widehat{\mu}(q) 
	+O_g \left(\|\phi_i\|_\infty \cdot \delta\right). \label{eq:firstc}
	\end{align}
	Using (\ref{eq:disintegrate}) and the definition of the Hubbard-Masur constant $\Lambda_g > 0$ we can write
	\begin{gather}
		\int_{\mathcal{Q}^1\mathcal{M}_g} \phi_i(\pi(q)) \thinspace \lambda(q) \thinspace d\widehat{\mu}(q) 
		= \Lambda_g \cdot \left(\int_{\mathcal{M}_g} \phi_1(X) \thinspace d\widehat{\mathbf{m}}(X)\right). \label{eq:secondc}
	\end{gather}
	Putting (\ref{eq:stop_2}), (\ref{eq:firstc}), and (\ref{eq:secondc}) together we deduce
	\begin{gather}
	\int_{\mathcal{M}_g} \phi_1(X) \left( \int_{\mathcal{M}_g} \phi_2(Y) \ d\widehat{\mathbf{m}}_X^R(Y) \right) d\widehat{\mathbf{m}}(X) \label{eq:stop_3}\\
	= \frac{\Lambda_g^2}{h \cdot \widehat{\mathbf{m}}(\mathcal{M}_g)} \cdot \left( \int_{\mathcal{M}_g} \phi_1(X) \thinspace d\widehat{\mathbf{m}}(X) \right) \cdot \left( \int_{\mathcal{M}_g} \phi_2(Y) \thinspace d\widehat{\mathbf{m}}(Y) \right) \cdot e^{hR} \nonumber\\
	+O_\mathcal{K}\left(\|\phi_1\|_\infty \cdot \|\phi_2\|_\infty \cdot \left(\delta \cdot e^{hR} + \delta^{-2(h-1)}  \cdot e^{(h - \kappa') R}\right)\right). \nonumber
	\end{gather}
	Let $\delta = \delta(R) := e^{-\eta R}$ with $\eta = \eta(g) >0$ small enough so that $2(h-1)\eta < \kappa'$. It follows from (\ref{eq:stop_3}) that, for $\useconk{ball_equid} = \useconk{ball_equid}(g) := \min\{\eta, \kappa' - 2(h-1)\eta\} > 0$,
	\begin{gather*}
	\int_{\mm} \phi_1(X) \left( \int_{\mm} \phi_2(Y) \ d\widehat{\mathbf{m}}_X^R(Y) \right) d\widehat{\mathbf{m}}(X) \\
	=  \frac{\Lambda_g^2}{h \cdot \widehat{\mathbf{m}}(\mathcal{M}_g)} \cdot \left(\int_{\mm} \phi_1(X) \ d\widehat{\mathbf{m}}(X)\right) \cdot \left(\int_{\mm} \phi_2(Y) \ d\widehat{\mathbf{m}}(Y) \right) \cdot e^{hR} \\
	+ O_\mathcal{K}\left(\|\phi_1\|_\infty \cdot \|\phi_2\|_\infty \cdot e^{(h-\useconk{ball_equid})R}\right). \qedhere
	\end{gather*}
\end{proof}

Let $\mathcal{D}_g \subseteq \mathcal{T}_g$ be a measurable fundamental domain for the $\mcg$ action on $\tt$. Denote by $s \colon \mathcal{D}_g \to \mathcal{M}_g$ the restriction to $\mathcal{D}_g$ of the quotient map $\underline{\smash{p}} \colon \tt \to \mm$. The following result is an immediate consequence of Theorem \ref{theo:ball_equidistribution} and the proper discontinuity of the $\mcg$ action on $\tt$.

\newconk{ball_equid_extra}
\begin{theorem}
	\label{theo:ball_equidistribution_extra}
	Let $\mathcal{K} \subseteq \mathcal{T}_g$ be a compact subset and $\phi_1,\phi_2 \in  L^\infty(\tt,\m)$ be essentially bounded functions with  $\esupp(\phi_1), \esupp(\phi_2) \subseteq \mathcal{K}$. Then, for every $R > 0$,
	\begin{gather*}
	\sum_{\mc \in \mcg} \int_{\tt} \phi_1(X) \left( \int_{\mm} \phi_2(\mc.s^{-1}(Y)) \ d\widehat{\mathbf{m}}_X^R(Y) \right) d\mathbf{m}(X) \\
	=  \frac{\Lambda_g^2}{h \cdot \widehat{\mathbf{m}}(\mathcal{M}_g)} \cdot \left(\int_{\tt} \phi_1(X) \ d\mathbf{m}(X)\right) \cdot \left(\int_{\tt} \phi_2(Y) \ d\mathbf{m}(Y) \right) \cdot e^{hR} \\
	+ O_\mathcal{K}\left(\|\phi_1\|_\infty \cdot \|\phi_2\|_\infty \cdot e^{(h-\useconk{ball_equid_extra})R}\right), 
	\end{gather*}
	where $\useconk{ball_equid_extra} =\useconk{ball_equid_extra}(g) > 0$ is a constant depending only on $g$.
\end{theorem}

\section{Effective lattice point count in Teichmüller metric balls}

\subsection*{Outline of this section.} The goal of this section is to prove Theorem \ref{theo:count_intro}. We actually prove a slightly more precise version, which we introduce below as Theorem \ref{theo:count}. Averaging and unfolding arguments will reduce the proof of Theorem \ref{theo:count} to an application of Theorem \ref{theo:ball_equidistribution}, more specifically, of Theorem \ref{theo:ball_equidistribution_extra}. 

\subsection*{Statement of the main theorem.} Recall that $B_R(X) \subseteq \tt$ denotes the ball of radius $R > 0$ centered at $X \in \tt$ with respect to the Teichmüller metric. For every $X,Y \in \tt$ and every $R > 0$ denote
\[	
F_R(X,Y) := \#\{\mc \in \mcg \ | \ \mc. Y \in B_R(X)\}.
\]
Recall that $\mathbf{m} := \pi_* \mu$ denotes the pushforward to $\tt$ of the Masur-Veech measure $\mu$ on $\qut$ under the projection $\pi \colon \qut \to \tt$ and that $\widehat{\m}$ denote the local pushforward of $\mathbf{m}$ to $\mm$.  Recall that $h := 6g-6$ and that $\Lambda_g > 0$ denotes the Hubbard-Masur constant introduced in \S 2. In this section we prove the following version of Theorem \ref{theo:count_intro}.

\newconk{count}
\begin{theorem}
	\label{theo:count}
	Let $\mathcal{K} \subseteq \tt$ be a compact subset. For every $X,Y \in \mathcal{K}$ and every $R>0$,
	\[
	F_R(X,Y)=  \frac{\Lambda_g^2}{h \cdot \widehat{\mathbf{m}}(\mathcal{M}_g)}\cdot e^{hR} + O_\mathcal{K}\left(e^{(h-\useconk{count})R}\right),
	\]
	where $\useconk{count} = \useconk{count}(g) > 0 $ is a constant depending only on $g$.
\end{theorem}

\subsection*{Proof of Theorem \ref{theo:count}.} The proof of Theorem \ref{theo:count} will use Theorem \ref{theo:ball_equidistribution_extra} and the following estimate for the measure of small Teichmüller metric balls. This estimate can be proved using a compactness argument, the fact that the Teichmüller metric on $\tt$ is Finsler, and the smoothness of the measure $\mathbf{m}$.

\newcond{d:vol}
\begin{lemma}
	\label{lem:teich_ball_bound}
	Let $\mathcal{K} \subseteq \mathcal{T}_g$ be a compact subset. There exists a constant $\usecond{d:vol} = \usecond{d:vol} (\mathcal{K}) > 0$ such that for every $X \in \mathcal{K}$ and every $0 < \delta < \usecond{d:vol}$,
	\[
	\mathbf{m}(B_\delta(X)) \asymp_{\mathcal{K}} \delta^{h}.
	\]
\end{lemma}

We now prove Theorem \ref{theo:count}.

\begin{proof}[Proof of Theorem \ref{theo:count}.]
	Fix $\mathcal{K} \subseteq \tt$ compact. Let $\usecond{d:vol} = \usecond{d:vol}(\mathcal{K}) > 0$ be as in Lemma \ref{lem:teich_ball_bound} and $\delta_0 = \delta_0(\mathcal{K}) := \min\{\usecond{d:vol},1\} > 0$. Fix $X,Y \in \mathcal{K}$. Let $R > 2$ and $0 < \delta = \delta(\mathcal{K},R) < \delta_0$, to be fixed later. By the triangle inequality and the $\mcg$ invariance of the Teichmüller metric $d_\mathcal{T}$, if $X',Y' \in \mathcal{T}_g$ are such that $d_\mathcal{T}(X,X') < \delta$ and $d_\mathcal{T}(Y,Y') < \delta$, then 
	\begin{equation}
	\label{eq:count_comparison}
	F_{R-2\delta}(X',Y') \leq F_R(X,Y) \leq F_{R+2\delta}(X',Y').
	\end{equation}
	Multiplying (\ref{eq:count_comparison}) by $\mathbbm{1}_{B_\delta(X)}(X') \cdot \mathbbm{1}_{B_\delta(Y)}(Y')$ we obtain the following inequalities, valid for every $X',Y' \in \mathcal{T}_g$,
	\begin{gather}
	\mathbbm{1}_{B_\delta(X)}(X') \cdot \mathbbm{1}_{B_\delta(Y)}(Y') \cdot F_{R-2\delta}(X',Y') \leq \mathbbm{1}_{B_\delta(X)}(X') \cdot \mathbbm{1}_{B_\delta(Y)}(Y') \cdot F_R(X,Y) \label{eq:coun_ineq_1},\\
	\mathbbm{1}_{B_\delta(X)}(X') \cdot \mathbbm{1}_{B_\delta(Y)}(Y') \cdot F_R(X,Y) \leq \mathbbm{1}_{B_\delta(X)}(X') \cdot \mathbbm{1}_{B_\delta(Y)}(Y') \cdot F_{R+2\delta}(X',Y'). \label{eq:coun_ineq_2}
	\end{gather}
	We use (\ref{eq:coun_ineq_2}) together with Theorem \ref{theo:ball_equidistribution} and Lemma \ref{lem:teich_ball_bound} to show that 
	\begin{gather}
	F_R(X,Y)
	- \frac{\Lambda_g^2}{h \cdot \widehat{\mathbf{m}}(\mathcal{M}_g)} \cdot e^{hR}
	\preceq_{\mathcal{K}} \delta \cdot e^{hR} + \delta^{-2h} \cdot e^{(h-\useconk{ball_equid_extra})R}. \label{eq:gg1}
	\end{gather}
	Integrating (\ref{eq:coun_ineq_2}) with respect to $d\mathbf{m}(Y') \thinspace d\mathbf{m}(X')$ we deduce
	\begin{gather}
	\mathbf{m}(B_\delta(X)) \cdot \mathbf{m}(B_\delta(Y)) \cdot F_R(X,Y) \label{eq:count_ineq_1} \\
	\leq \int_{\mathcal{T}_g} \mathbbm{1}_{B_\delta(X)}(X') \left( \int_{\mathcal{T}_g} \mathbbm{1}_{B_\delta(Y)}(Y') \thinspace F_{R+2\delta}(X',Y') \thinspace d\mathbf{m}(Y')\right) d\mathbf{m}(X'). \nonumber
	\end{gather}
	Fix a measurable fundamental domain $\mathcal{D}_g \subseteq \mathcal{T}_g$ for the action of $\text{Mod}_g$ on $\mathcal{T}_g$. Denote by $s\colon \mathcal{D}_g \to \mathcal{M}_g$ the restriction to $\mathcal{D}_g$ of the quotient map $\underline{\smash{p}} \colon \mathcal{T}_g \to \mathcal{M}_g$. Using this map we can write
	\begin{gather}
	\int_{\mathcal{T}_g} \mathbbm{1}_{B_\delta(X)}(X') \left( \int_{\mathcal{T}_g} \mathbbm{1}_{B_\delta(Y)}(Y') \thinspace F_{R+2\delta}(X',Y') \thinspace d\mathbf{m}(Y')\right) d\mathbf{m}(X') \label{eq:pro_1}	\\
	= \sum_{\mc\in \mcg} \int_{\mathcal{T}_g} \mathbbm{1}_{B_\delta(X)}(X') \left( \int_{\mathcal{M}_g} \mathbbm{1}_{B_\delta(Y)}(\mc'.s^{-1}(Y')) \thinspace  F_{R+2\delta}(X',\mc.s^{-1}(Y')) \thinspace d\widehat{\mathbf{m}}(Y')\right) d\mathbf{m}(X').	\nonumber
	\end{gather}
	Fix $\mc \in \mcg$. An unfolding argument shows that, for every $X' \in \mathcal{M}_g$,
	\begin{gather}
	\int_{\mm} \mathbbm{1}_{B_\delta(Y)}(\mc.s^{-1}(Y')) \thinspace F_{R+2\delta}(X',\mc.s^{-1}(Y')) \thinspace d\widehat{\mathbf{m}}(Y') \label{eq:pro_2}\\
	= \int_{\tt} \sum_{\mathbf{h} \in \mcg} \mathbbm{1}_{\mathcal{D}_g}(\mathbf{h}.Y') \thinspace \mathbbm{1}_{B_\delta(Y)}(\mc.\mathbf{h}.Y') \thinspace \mathbbm{1}_{B_R(X')}(Y')\thinspace d\mathbf{m}(Y') \nonumber\\
	= \int_{\tt} \sum_{\mathbf{h} \in \mcg} \mathbbm{1}_{\mathcal{D}_g}(\mathbf{h}.Y') \thinspace \mathbbm{1}_{B_\delta(Y)}(\mc.\mathbf{h}.Y') \thinspace d\mathbf{m}_{X}^{R+2\delta}(Y') \nonumber.
	\end{gather}
	As $\widehat{\mathbf{m}}_{X}^{R+2\delta}$ is the pushforward to $\mathcal{M}_g$ of the measure $\mathbf{m}_{X}^{R+2\delta}$ on $\tt$,
	\begin{equation}
	\int_{\tt} \sum_{\mathbf{h} \in \mcg} \mathbbm{1}_{\mathcal{D}_g}(\mathbf{h}.Y') \thinspace \mathbbm{1}_{B_\delta(Y)}(\mc.\mathbf{h}.Y') \thinspace d\mathbf{m}_{X}^{R+2\delta}(Y') = \int_{\mathcal{M}_g} \mathbbm{1}_{B_\delta(Y)}(\mathbf{g}. s^{-1}(Y')) \thinspace d \widehat{\mathbf{m}}_{X}^{R+2\delta}(Y'). \label{eq:pro_3}
	\end{equation}
	Putting together (\ref{eq:pro_1}), (\ref{eq:pro_2}), and (\ref{eq:pro_3}) we deduce
	\begin{gather}
	\int_{\tt} \mathbbm{1}_{B_\delta(X)}(X') \left(\int_{\tt} \mathbbm{1}_{B_\delta(Y)}(Y') \thinspace F_{R+2\delta}(X',Y') \thinspace d\mathbf{m}(Y') \right) d\mathbf{m}(X') \label{eq:pro_4}  \\
	= \sum_{\mc \in \mcg} \int_{\tt} \mathbbm{1}_{B_\delta(X)}(X') \left(\int_{\mathcal{M}_g} \mathbbm{1}_{B_\delta(Y)}(\mathbf{g}. s^{-1}(Y')) \thinspace d \widehat{\mathbf{m}}_{X}^{R+2\delta}\right) d\mathbf{m}(X') \nonumber.
	\end{gather}
	By Theorem \ref{theo:ball_equidistribution_extra},
	\begin{gather}
	\sum_{\mc \in \mcg} \int_{\tt} \mathbbm{1}_{B_\delta(X)}(X') \left(\int_{\mathcal{M}_g} \mathbbm{1}_{B_\delta(Y)}(\mathbf{g}. s^{-1}(Y')) \thinspace d \widehat{\mathbf{m}}_{X}^{R+2\delta}\right) d\mathbf{m}(X') \label{eq:prog_5}\\
	=  \frac{\Lambda_g^2}{h \cdot \widehat{\mathbf{m}}(\mathcal{M}_g)} \cdot \mathbf{m}(B_\delta(X)) \cdot \mathbf{m}(B_\delta(Y)) \cdot e^{h(R+2\delta)} 
	+ O_\mathcal{K}\left(e^{(h-\useconk{ball_equid_extra})(R+2\delta)}\right). \nonumber
	\end{gather}	
	As $0 < \delta < 1$, the mean value theorem ensures
	\begin{equation}
	\label{eq:exp_approx}
	e^{h\delta} = 1 + O_g(\delta).
	\end{equation}
	Putting together (\ref{eq:pro_4}), (\ref{eq:prog_5}), and (\ref{eq:exp_approx}) we deduce
	\begin{gather}
	\int_{\tt} \mathbbm{1}_{B_\delta(X)}(X') \left(\int_{\tt} \mathbbm{1}_{B_\delta(Y)}(Y') \thinspace F_{R+2\delta}(X',Y') \thinspace d\mathbf{m}(Y') \right) d\mathbf{m}(X') \label{eq:prog55}\\
	= \frac{\Lambda_g^2}{h \cdot \widehat{\mathbf{m}}(\mathcal{M}_g)} \cdot \mathbf{m}(B_\delta(X)) \cdot \mathbf{m}(B_\delta(Y)) \cdot e^{hR} \nonumber\\
	+ O_\mathcal{K}\left(\mathbf{m}(B_\delta(X)) \cdot \mathbf{m}(B_\delta(Y)) \cdot \delta \cdot e^{hR} + e^{(h-\useconk{ball_equid_extra})(R+2\delta)}\right). \nonumber
	\end{gather}
	Combining (\ref{eq:count_ineq_1}) with (\ref{eq:prog55}) and dividing by $\mathbf{m}(B_\delta(X)) \cdot \mathbf{m}(B_\delta(Y))$ we deduce
	\begin{gather}
	F_R(X,Y) - \frac{\Lambda_g^2}{h \cdot \widehat{\mathbf{m}}(\mathcal{M}_g)} \cdot e^{hR}  \preceq_{\mathcal{K}} \delta \cdot e^{hR} + \mathbf{m}(B_\delta(X))^{-1} \cdot \mathbf{m}(B_\delta(Y))^{-1} \cdot e^{(h-\useconk{ball_equid_extra})R}. \label{eq:prog_6}
	\end{gather}
	As $0 < \delta < \usecond{d:vol}$, (\ref{eq:prog_6}) and Lemma \ref{lem:teich_ball_bound} imply (\ref{eq:gg1}) holds, that is,
	\begin{gather}
	F_R(X,Y) - \frac{\Lambda_g^2}{h \cdot \widehat{\mathbf{m}}(\mathcal{M}_g)} \cdot e^{hR}  \preceq_{\mathcal{K}} \delta \cdot e^{hR} + \delta^{-2h} \cdot e^{(h-\useconk{ball_equid_extra})R}. \label{eq:prog_7}
	\end{gather}
	Analogous arguments using (\ref{eq:coun_ineq_1}) instead of (\ref{eq:coun_ineq_2}) give the lower bound
	\begin{gather}
	F_R(X,Y) - \frac{\Lambda_g^2}{h \cdot \widehat{\mathbf{m}}(\mathcal{M}_g)} \cdot e^{hR}  \succeq_{\mathcal{K}} \delta \cdot e^{hR} + \delta^{-2h} \cdot e^{(h-\useconk{ball_equid_extra})R} \label{eq:prog_8}.
	\end{gather}
	Combining (\ref{eq:prog_7}) with (\ref{eq:prog_8}) we deduce
	\begin{gather}
	F_R(X,Y) = \frac{\Lambda_g^2}{h \cdot \widehat{\mathbf{m}}(\mathcal{M}_g)} \cdot e^{hR} + O_\mathcal{K}\left( \delta \cdot e^{hR} + \delta^{-2h} \cdot e^{(h-\useconk{ball_equid_extra})R} \right). \label{eq:prog_9}
	\end{gather}
	Let $\delta = \delta(\mathcal{K},R) := \delta_0 e^{-\eta R}$ with $\eta = \eta(g) >0$ small enough so that $2h\eta < \useconk{ball_equid_extra}$. Denote $\useconk{count} = \useconk{count}(g) := \min\{\eta, \useconk{ball_equid_extra} - 2h\eta\} > 0$. It follows from (\ref{eq:prog_9}) that, for every $R > 2$,
	\[
	F_R(X,Y)
	= e^{hR} \cdot \frac{\Lambda_g^2}{h \cdot \widehat{\mathbf{m}}(\mathcal{M}_g)} + O_{\mathcal{K}}\left(e^{(h-\useconk{count})R}\right).
	\]
	The same equality holds for every $R > 0$ by increasing the implicit constant in the error term.
\end{proof}

\section{The Hubbard-Masur functions}

\subsection*{Outline of this section.} In this section we study the regularity of the Hubbard-Masur functions $\lambda^-,\lambda^+ \colon\allowbreak \qut(\mathbf{1}) \to \mathbf{R}_{>0}$ introduced in \S 2. We begin by showing these functions are $\mathrm{SO}(2)$-invariant; see Proposition \ref{prop:hm_inv_sum}. Using similar ideas we  show the functions $\lambda^-$ and $\lambda^+$ coincide everywhere; see Proposition \ref{prop:one_hm_function_sum}. We then give a bound for these functions in terms of the length of the shortest saddle connection; see Proposition \ref{prop:HM_bound_sum}. Finally, we bound the integral of these functions near the multiple zero locus; see Proposition \ref{prop:small_HM_integral_sum}. 

\subsection*{The smooth fiberwise measures of $\boldsymbol{\qut}$.} We begin by introducing a construction that will be used in the proof of Proposition \ref{prop:hm_inv_sum}. Recall that $\mathbf{m} := \pi_* \mu$ denotes the pushforward to $\tt$ of the Masur-Veech measure $\mu$ on $\qut$ under the projection $\pi \colon \qut \to \tt$. Recall that the Masur-Veech measure $\mu$ on $\qut$ can be disintegrated along the fibers of the projection $\pi \colon \mathcal{Q}^1\mathcal{T}_g \to \mathcal{T}_g$ in the following sense: There exists a unique family of fiberwise probability measures $\{s_X\}_{X \in \mathcal{T}_g}$, with $s_X$ supported on $S(X) := \pi^{-1}(X)$, such that the following disintegration formula holds:
\begin{equation}
\label{eq:fiberwise_measures_smooth}
d\mu(X,q) = ds_X(q) \thinspace d\mathbf{m}(X).
\end{equation}

The following procedure constructs the fiberwise measures $\{s_X\}_{X \in \mathcal{T}_g}$ in a smooth way. Fix $X \in \mathcal{T}_g$. Consider an arbitrary basis $v_1,\dots,v_h$ of $T_X\mathcal{T}_g$. Given $q \in S(X)$, let $\widetilde{v}_1, \dots, \widetilde{v}_h$ be arbitrary lifts of $v_1,\dots,v_h$ to $T_q \mathcal{Q}^1\mathcal{T}_g$. Contract the Masur-Veech volume form $\mu_q $ on $T_q \mathcal{Q}^1\mathcal{T}_g$ by the vectors $\widetilde{v}_1, \dots, \widetilde{v}_h$. As these vectors are linearly independent and do not belong to $T_qS(X)$, this construction yields a non-zero volume form $s_{X}'(q)$ on $T_qS(X)$. This volume form is independent of the choice of lifts $\widetilde{v}_1, \dots, \widetilde{v}_h$.

Using a local trivialization of the sphere bundle $\pi \colon \mathcal{Q}^1\mathcal{T}_g \to \mathcal{T}_g$, one can show that the volume forms $s_{X}'(q)$ give rise to a smooth volume form $s_X'$ on $S(X)$. As $S(X)$ is compact, its total mass $m_X'$ with respect to this volume form is finite. Consider the volume form $s_X$ on $S(X)$ given by $s_X := s_X'/m_X'$. This volume form is independent of the choice of basis $v_1,\dots,v_h$ of $T_X \mathcal{T}_g$. The smooth measures induced by the volume forms $\{s_X\}_{X \in \mathcal{T}_g}$ constructed this way satisfy (\ref{eq:fiberwise_measures_smooth}).

\subsection*{$\boldsymbol{\mathrm{SO}(2)}$-invariance of the Hubbard-Masur functions.} As a first step towards proving Proposition \ref{prop:hm_inv_sum}, we prove the following invariance property for the smooth fiberwise measures $\{s_X\}_{X \in \mathcal{T}_g}$. The proof will use the notation introduced above.

\begin{lemma}
	\label{lem:rot_inv}
	For every $X \in \mathcal{T}_g$, the volume form $s_X$ on $S(X)$ is $\mathrm{SO}(2)$-invariant.
\end{lemma}

\begin{proof}
	Fix $X \in \mathcal{T}_g$. Let $v_1,\dots,v_h$ be an arbitrary basis of $T_X\mathcal{T}_g$. Given $q \in S(X)$, consider arbitrary lifts $\widetilde{v}_1,\dots,\widetilde{v}_h$ of $v_1,\dots,v_h$ to $T_q \mathcal{Q}^1\mathcal{T}_g$. Let $\theta \in [0,2\pi]$ be arbitrary. As the action of $\mathrm{SO}(2)$ on $\qut$ preserves the fibers of the projection $\pi \colon \qut \to \tt$, the tangent vectors $dr_\theta \widetilde{v}_1,\dots, dr_\theta \widetilde{v}_h$ on $T_{r_\theta q}\mathcal{Q}^1\mathcal{T}_g$ project to $v_1,\dots,v_h$ on $T_X \tt$. As the Masur-Veech volume form $\mu$ on $\qut$  is preserved by $r_\theta$, 
	\begin{align*}
	(r_\theta^* s_X')(q) &= r_\theta^* (s_X'(r_\theta q)) 
	= r_\theta^* (\iota_{dr_\theta \widetilde{v}_1,\dots,dr_\theta \widetilde{v}_h}\mu_{r_\theta q} ) \\
	&= \iota_{\widetilde{v}_1,\dots,\widetilde{v}_h} \mu_q 
	= s_X'(q).
	\end{align*}
	Normalizing the volume forms $r_\theta^* s_X'(q)$ and $s_X'(q)$ by the total fiberwise measure $m_X'$ finishes the proof.
\end{proof}

To prove Proposition \ref{prop:hm_inv_sum} we will also need the following result, which is a direct consequence of \cite[Corollary 5.9]{Du15}; we refer the reader to \S 2 for the notation used in this statement.

\begin{proposition}[\cite{Du15}]
	\label{prop:dum_inv}
	For every $X \in \mathcal{T}_g$, the volume forms $(\Re|_{S(X)})^*\overline{\nu}_X$ and $(\Im|_{S(X)})^*\overline{\nu}_X$ on $S(X) \cap \qutp$ are $\text{SO}(2)$-invariant.
\end{proposition}

Proposition \ref{prop:hm_inv_sum} follows directly from Lemma \ref{lem:rot_inv}, Proposition \ref{prop:dum_inv}, and definitions (\ref{eq:hm_3}) and (\ref{eq:hm_4}) of the Hubbard-Masur functions $\lambda^-,\lambda^+ \colon \qut(\mathbf{1}) \to \mathbf{R}_{>0}$. 

\subsection*{The Hubbard-Masur function.} We now prove Proposition \ref{prop:one_hm_function_sum}, i.e., that the Hubbard-Masur functions $\lambda^-,\lambda^+ \colon \mathcal{Q}^1\mathcal{T}_g(\mathbf{1}) \to \mathbf{R}_{>0}$ coincide everywhere. We will denote by $\lambda \colon \qut(\mathbf{1}) \to \mathbf{R}_{>0}$ the common value of these functions and refer to it as the Hubbard-Masur function.

\begin{proof}[Proof of Proposition \ref{prop:one_hm_function_sum}]
	Fix $q \in \mathcal{Q}^1\mathcal{T}_g(\mathbf{1})$ and let $X :=\pi(q) \in \tt$. By Lemma \ref{lem:rot_inv}, the volume form $s_X$ on $S(X)$ is $\mathrm{SO}(2)$-invariant. By Proposition \ref{prop:hm_inv_sum}, the function $\lambda^+ \colon \mathcal{Q}^1\mathcal{T}_g(\mathbf{1}) \to \mathbf{R}_{>0}$ is $\mathrm{SO}(2)$-invariant. Consider multiplication by $e^{i\pi}$ as an automorphism of $\qut(\mathbf{1})$. Using the fact that $\Re = \Im \circ e^{i\pi}$ on $\qut(\mathbf{1})$ and definition (\ref{eq:hm_4}) of $\lambda^+$ we deduce
	\begin{align*}
	|((\Re|_{S(X)})^*\overline{\nu}_X)(q)| &= |((e^{i\pi})^*(\Im|_{S(X)})^*\overline{\nu}_X)(q) |
	= |(e^{i\pi})^*((\Im|_{S(X)})^*\overline{\nu}_X)(e^{i\pi}q))|\\
	&= |(e^{i\pi})^*(\lambda^+(e^{i\pi}q) \cdot |s_X(e^{i\pi}q)|)|
	= \lambda^+(e^{i\pi}q) \cdot |(e^{i\pi})^*( |s_X(e^{i\pi}q)|)|\\
	&= \lambda^+(q) \cdot |s_X(q)|.
	\end{align*}
	At the same time, definition (\ref{eq:hm_3}) of $\lambda^-$ ensures
	\[
	|((\Re|_{S(X)})^*\overline{\nu}_X)(q)| = \lambda^-(q) \cdot |s_X(q)|.
	\]
	It follows that $\lambda^+(q) = \lambda^{-}(q)$.
\end{proof}

%By Proposition \ref{prop:dum_inv} the volume forms $(\Re|_{S(X)})^*\overline{\nu}$ and  $(\Im|_{S(X)})^*\overline{\nu}$ on $S(X) \cap \qut(\mathbf{1})$ are $\mathrm{SO}(2)$-invariant.

\subsection*{Bounding the Hubbard-Masur function.} We now prove Proposition \ref{prop:HM_bound_sum}. Let $q \in \qutp$ and denote by $\widetilde{X}$ the canonical double cover of $X:= \pi(q) \in \tt$ induced by $q$. Recall that we can identify $T_q \qt(\mathbf{1})$ with $H^1_\text{odd}(\widetilde{X};\mathbf{C})$. Endow  $T_q \qt(\mathbf{1}) =  H^1_\text{odd}(\widetilde{X};\mathbf{C})$ with the Hodge inner product  and $T_X\tt$ with the Teichmüller norm. Denote by $\| d\pi_q \|$ the norm of the linear operator $d\pi_q \colon T_q \qutp \to T_X \tt$. Recall that $\ell_{\min}(q) > 0$ denotes the length of the shortest saddle connection of $q$. The following estimate due to Kahn and Wright \cite{KW20} will play a crucial role in the proof of Proposition \ref{prop:HM_bound_sum}.

\begin{proposition}[\cite{KW20}]
	\label{prop:KW_new}
	For every $q \in \qutp$,
	\[
	\| d\pi_q \| \preceq_g \ell_{\min}(q)^{-1}.
	\]
\end{proposition}

Let $V$ be an $n$-dimensional inner product space. Denote by $\langle \cdot, \cdot \rangle$ the inner product of $V$ and by $\|\cdot \|$ the corresponding norm. On the $1$-dimensional vector space $V^{\wedge n}$ consider the inner product $\langle \cdot , \cdot \rangle$ specified by
\[
\langle v_1 \wedge \dots \wedge v_n, u_1 \wedge \dots \wedge u_n \rangle := \det\left(\langle v_i,u_j\rangle_{i,j=1}^n\right).
\]
Let $\| \cdot \|$ be the induced norm on $V^{\wedge n}$. The proof of Proposition \ref{prop:HM_bound_sum} will use the following fact.

\begin{lemma}
	\label{lem:determinant_new}
	Let $V$ be an $n$-dimensional inner product space. Then, for every $v_1,\dots,v_n \in V$,
	\[
	\| v_1 \wedge \dots \wedge v_n \| \preceq_n \prod_{i=1}^n \| v_i \|.
	\]
\end{lemma}

We are now ready to prove Proposition \ref{prop:HM_bound_sum}; recall that $h := 6g-6$.

\begin{proof}[Proof of Proposition \ref{prop:HM_bound_sum}]
	Fix $\mathcal{K} \subseteq \mathcal{T}_g$ compact and $q \in \qutp \cap \pi^{-1}(\mathcal{K})$. Denote by $\widetilde{X}$ the canonical double cover of $X := \pi(q)$ induced by $q$ and by $\omega$ the canonical square root of the pullback of $q$ to $\widetilde{X}$. Let $u_1 := [\Im(\omega)] \in E^{uu}(q)$. Denote by $\text{span}(\Re(\omega), \Im(\omega))^\perp \subseteq H^1_\text{odd}(\widetilde{X};\mathbf{R})$ the orthogonal complement of $\text{span}(\Re(\omega), \Im(\omega))$ in $H^1_\text{odd}(\widetilde{X};\mathbf{R})$ with respect to the Hodge inner product. Equivalently, one can consider the symplectic complement with respect to the cup intersection pairing. It follows that $\text{span}(\Re(\omega), \Im(\omega))^\perp \subseteq E^{uu}(q)$. Let $u_2,\dots,u_{h-1}$ be an orthonormal basis of $\text{span}(\Re(\omega), \Im(\omega))^{\perp}$ with respect to the Hodge inner product. Recall definition (\ref{eq:hm_2}) of the Hubbard-Masur function:
	\begin{equation}
	\label{eq:piece0}
	|\mathbf{m}_{X}| = \lambda(q) \cdot |(d\pi_q)_*(\mu_{E^{u}(q)})|.
	\end{equation}
	
	We evaluate both sides of (\ref{eq:piece0}) on $d\pi_q u_1 \wedge \dots \wedge d\pi_q u_{h-1} \wedge d\pi_q [\overline{\omega}]$. We first evaluate the right hand side. Recall that $\mu_{E^{u}(q)} = \mu_{E^{uu}(q)} \wedge dt$ when parametrizing $E^{u}(q) = E^{uu}(q) \oplus \mathbf{R} \cdot [\overline{\omega}]$. It follows that
	\begin{equation}
	\label{eq:step1}
	|(d\pi_q)_*(\mu_{E^{u}(q)})|(d\pi_q u_1 \wedge \dots \wedge d\pi_q u_{h-1} \wedge d\pi_q [\overline{\omega}]) = |\mu_{E^{uu}(q)}|(u_1 \wedge \dots \wedge u_{h-1}).
	\end{equation}
	Denote by $\mu_{H_\odd^1(\widetilde{X};\mathbf{R})}$ the volume form induced by the cup intersection pairing on $H_\odd^1(\widetilde{X};\mathbf{R})$. By definition,
	\begin{equation}
	\label{eq:step2}
	|\mu_{E^{uu}(q)}|(u_1 \wedge \dots \wedge u_{h-1}) = |\mu_{H_\odd^1(\widetilde{X};\mathbf{R})}|(u_1 \wedge \dots \wedge u_{h-1} \wedge [\Re(\omega)]).
	\end{equation}
	Recall that the volume form $\mu_{H_\odd^1(\widetilde{X};\mathbf{R})}$ coincides with the volume form induces by the Hodge inner product on $H^1_\text{odd}(\widetilde{X};\mathbf{R})$. As $u_1,\dots,u_{h-1}, [\Re(\omega)]$ is an orthonormal basis of $H^1_\text{odd}(\widetilde{X};\mathbf{R})$,
	\begin{equation}
	\label{eq:step3}
	|\mu_{H_\odd^1(\widetilde{X};\mathbf{R})}|(u_1 \wedge \dots \wedge u_{h-1} \wedge [\Re(\omega)]) = 1.
	\end{equation}
	Putting together (\ref{eq:step1}), (\ref{eq:step2}), and (\ref{eq:step3}) we deduce
	\begin{equation}
	\label{eq:piece1}
	|(d\pi_q)_*(\mu_{E^{u}(q)})|(d\pi_q u_1 \wedge \dots \wedge d\pi_q u_{h-1} \wedge d\pi_q [\overline{\omega}]) = 1.
	\end{equation}
	
	We now evaluate the left hand side of (\ref{eq:piece0}) on $d\pi_q u_1 \wedge \dots \wedge d\pi_q u_{h-1} \wedge d\pi_q [\overline{\omega}]$. Fix a complete Riemannian metric $\langle \cdot , \cdot \rangle$ on $\mathcal{T}_g$. Denote its induced norm by $\| \cdot \|$ and its induced volume form by $\text{vol}$. Consider the operator $\mathrm{vol}_X \colon (T_X\mathcal{T}_g)^{\wedge h} \to \mathbf{R}$. Denote its norm by $\| \mathrm{vol}_X \|$. As $\mathcal{K} \subseteq \tt$ is compact,
	\begin{align}
	|\mathbf{m}_X|(d\pi_q u_1 \wedge \dots \wedge d\pi_q u_{h-1} \wedge d\pi_q [\overline{\omega}])
	&\preceq_\mathcal{K}|\mathrm{vol_X}|(d\pi_q u_1 \wedge \dots \wedge d\pi_q u_{h-1} \wedge d\pi_q [\overline{\omega}]) \label{eq:step4}\\
	&\preceq_\mathcal{K} \| \mathrm{vol}_X \| \cdot \| d\pi_q u_1 \wedge \dots \wedge d\pi_q u_{h-1} \wedge d\pi_q [\overline{\omega}] \| \nonumber\\
	&\preceq_\mathcal{K} \| d\pi_q u_1 \wedge \dots \wedge d\pi_q u_{h-1} \wedge d\pi_q [\overline{\omega}] \|. \nonumber
	\end{align}
	By Lemma \ref{lem:determinant_new},
	\begin{equation}
	\| d\pi_q u_1 \wedge \dots \wedge d\pi_q u_{h-1} \wedge d\pi_q [\overline{\omega}] \| \preceq_g \left(\prod_{i=1}^{h-1} \| d\pi_q u_i \|\right) \cdot \| d\pi_q [\overline{\omega}] \|. \label{eq:step5}
	\end{equation}
	Recall that  $\| \cdot \|_\mathcal{T}$ denotes the Teichmüller norm on $\mathcal{T}_g$. As $\mathcal{K} \subseteq \mathcal{T}_g$ is compact, 
	\begin{equation}
	\left(\prod_{i=1}^{h-1} \| d\pi_q u_i \|\right) \cdot \| d\pi_q [\overline{\omega}] \| \preceq_{\mathcal{K}} \left(\prod_{i=1}^{h-1} \| d\pi_q u_i \|_\mathcal{T} \right) \cdot \| d\pi_q [\overline{\omega}] \|_\mathcal{T} = \prod_{i=1}^{h-1} \| d\pi_q u_i \|_\mathcal{T}. \label{eq:step6}
	\end{equation}
	Denote by $\| \cdot \|_H$ the norm induced by the Hodge inner product on $H^1_\text{odd}(\widetilde{X};\mathbf{R})$. Proposition \ref{prop:KW_new} implies
	\begin{equation}
	\prod_{i=1}^{h-1} \| d\pi_q u_i \|_\mathcal{T}  \preceq_\mathcal{K} \ell_{\min}(q)^{-(h-1)} \cdot \left(\prod_{i=1}^{h-1} \|u_i\|_H \right) . \label{eq:step7}
	\end{equation}
	As $u_1,\dots,u_{h-1}$ is an orthonormal basis of $E^{uu}(q)$,
	\begin{equation}
	\prod_{i=1}^{h-1} \|u_i\|_H = 1. \label{eq:step8}
	\end{equation}
	Putting together (\ref{eq:step4}), (\ref{eq:step5}), (\ref{eq:step6}), (\ref{eq:step7}), and (\ref{eq:step8}) we deduce
	\begin{equation}
	|\mathbf{m}_X|(d\pi_q u_1 \wedge \dots \wedge d\pi_q u_{h-1} \wedge d\pi_q [\overline{\omega}]) \preceq_{\mathcal{K}} \ell_{\min}(q)^{-(h-1)}. \label{eq:piece2}
	\end{equation}
	
	From (\ref{eq:piece0}), (\ref{eq:piece1}), and (\ref{eq:piece2}) we conclude
	\[
	\lambda(q) \preceq_\mathcal{K}  \ell_{\min}(q)^{-(h-1)}. \qedhere
	\]
\end{proof}

\subsection*{Train tracks coordinates of $\boldsymbol{\mf}$.} The rest of this section is devoted to the proof of Proposition \ref{prop:small_HM_integral_sum}. Using a convenient geometric construction, we will reduce this proof to an estimate in train track coordinates. We now introduce the basics of the theory of train tracks on $S_g$ and of the coordinates they induce on $\mf$. For more details we refer the reader to \cite{PH92}.

A train track $\tau$ on $S_g$ is an embedded $1$-complex satisfying the following conditions:
\begin{enumerate}
	\item Each edge of $\tau$ is a smooth path with a well defined tangent vector at each endpoint. All edges at a given vertex are tangent.
	\item For each component $R$ of $S_g \backslash \tau$, the double of $R$ along the smooth part of $\partial R$ has negative Euler characteristic.
\end{enumerate}
The vertices of $\tau$ where three or more edges meet are called switches. By considering the inward pointing tangent vectors of the edges incident to a vertex, one can divide these edges into incoming and outgoing edges. A train track $\tau$ on $S_g$ is said to be maximal if all the components of $S_g \backslash \tau$ are trigons, i.e., the interior of a disc with three non-smooth points on the boundary. 

A singular measured foliation $\eta \in \mf$ is said to be carried by a train track $\tau$ on $S_g$ if it can be obtained by collapsing the complementary regions in $S_g$ of a measured foliation of a tubular neighborhood of $\tau$ whose leaves run parallel to the edges of $\tau$. In this situation, the invariant transverse measure of $\eta$ corresponds to a counting measure $v$ on the edges of $\tau$ satisfying the switch conditions: at every switch of $\tau$ the sum of the measures of the incoming edges equals the sum of the measures of the outgoing edges. Every $\eta \in \mf$ is carried by some maximal train track $\tau$ on $S_g$.

Given a maximal train track $\tau$ on $S_g$, denote by $V(\tau) \subseteq (\mathbf{R}_{\geq0})^{18g-18}$ the $6g-6$ dimensional closed cone of non-negative counting measures on the edges of $\tau$ satisfying the switch conditions. The set $V(\tau)$ can be identified with the closed cone $U(\tau) \subseteq \mf$ of singular measured foliations carried by $\tau$. These identifications give rise to charts on $\mf$ called train track charts. The transition maps between these charts are piecewise integral linear. This allows one to endow $\mf$ with a $6g-6$ dimensional piecewise integral linear structure. The Thurston measure is the unique, up to scaling, piecewise integral linear measure on $\mf$. Recall that $\nu$ denotes the normalization of the Thurston measure induced by the symplectic structure described in \cite[\S 3.2]{PH92}. The Thurston measure has the following scaling property: for every measurable subset $A \subseteq \mf$ and every $t > 0$,
$
\nu(t \cdot A) = t^{6g-6} \cdot \nu(A).
$

\subsection*{Delaunay triangulations of quadratic differentials.} Delaunay triangulations provide a canonically defined family of triangulations with nice geometric properties. Our proof of Proposition \ref{prop:small_HM_integral_sum} will make use of these triangulations in a crucial way. We now summarize the properties that will be relevant for our purposes. We refer the reader to \cite[\S 4]{MS91} for more details.

Every quadratic differential $q$ gives rise to finitely many Delaunay triangulations on its underlying Riemann surface $X$, the edges of which are saddle connections of $q$. The number of possible Delaunay triangulations that $q$ can induce on $X$ is bounded uniformly in terms of the genus of $X$. By a marked triangulation on $S_g$ we mean an isotopy class of triangulations on such surface. Given a marked quadratic differential $q \in \qt$, pulling back any of the Delaunay triangulations induced by $q$ via the marking yields a marked triangulation on $S_g$. These marked triangulations are invariant under complex scaling of quadratic differentials. They induce period coordinates on the stratum of $\qt$ that $q$ belongs to. 

The following lemma highlights two important geometric properties of these triangulations. Recall that, given $q \in \qut$, $\ell_\gamma(q) > 0$ denotes the length of a saddle connection $\gamma$ of $q$ and $\mathrm{diam}(q) > 0$ denotes the diameter of $q$. The proofs of these facts can be found in \cite[\S 4]{MS91} and \cite[Lemma 3.11]{ABEM12}. 

\begin{lemma}
	\label{lem:delaunay}
	Let $q \in \qut$ and $\Delta$ be a Delaunay triangulation of $q$. Then, the following properties hold:
	\begin{enumerate}
		\item If $\gamma$ is a saddle connection of $q$ which belongs to $\Delta$, then $\ell_{\gamma}(q) \preceq_g \mathrm{diam}(q)$.
		\item If $\gamma$ is a saddle connection of $q$ such that $\ell_{\gamma}(q) \leq \sqrt{2} \cdot \ell_{\min}(q)$, then $\gamma$ belongs to $\Delta$.
	\end{enumerate}
\end{lemma}

\subsection*{Integrating the Hubbard-Masur function near the multiple zero locus.} We now prove Proposition \ref{prop:small_HM_integral_sum}. We will need the following  fact about Delaunay triangulations.

\begin{lemma}
	\label{lem:finite_delaunay}
	For every compact subset $\mathcal{K} \subseteq \mathcal{T}_g$, the Delaunay triangulations of marked quadratic differentials in $\pi^{-1}(\mathcal{K})$ induce finitely many marked triangulations on $S_g$.
\end{lemma}

\begin{proof}
	Fix a compact subset $\mathcal{K} \subseteq \tt$. Delaunay triangulations of marked quadratic differentials on $S_g$ have at most $4g-4$ vertices. In particular, these triangulations give rise to finitely many combinatorial types of triangulations on $S_g$. By Lemma \ref{lem:delaunay}, the edges of the Delaunay triangulations of marked quadratic differentials in $\mcg \cdot \pi^{-1}(\mathcal{K})$ have uniformly bounded length. In particular, given a marked triangulation $\Delta$ on $S_g$, the subset of marked quadratic differentials in $\mcg \cdot \pi^{-1}(\mathcal{K})$ having $\Delta$ as a Delaunay triangulation has compact closure. It follows that, as the action of $\mcg$ on $\mathcal{T}_g$ is properly discontinuous, the number of marked triangulations on $S_g$ of a given combinatorial type that can arise as a Delaunay triangulation of a marked quadratic differentials in $\pi^{-1}(\mathcal{K})$ is finite. This finishes the proof.
\end{proof}

The following result combines Lemma \ref{lem:finite_delaunay} with a construction of train tracks dual to triangulations of quadratic differentials \cite[\S4.4]{Mir08a}. We will take advantage of this result and of the explicit nature of the construction used in its proof several times throughout this paper. 

\begin{lemma}
	\label{lem:finite_tt}
	For every $\mathcal{K} \subseteq \mathcal{T}_g$ compact, there exists a finite collection of maximal train tracks $\{\tau_i\}_{i=1}^n$ on $S_g$ with the following property: For every $q \in \pi^{-1}(\mathcal{K}) \cap \qutp$, the singular measured foliation $\Re(q) \in \mf$ is carried by some $\tau_i$, and, for every edge $e$ of $\tau_i$, the corresponding counting measure $v(e)$ is equal to the absolute value of the real part of the holonomy of a saddle connection of a Delaunay triangulation of $q$.
\end{lemma}

\begin{proof}
	Fix $q \in \qut(\mathbf{1})$ and let $\Delta'$ be a Delaunay triangulation of $q$. Denote by $\Delta$ the corresponding marked triangulation on $S_g$.  We first describe a procedure for constructing a train track $\tau$ on $S_g$ dual to $\Delta$ and which carries $\Re(q) \in \mf$. Let $T'$ be a triangle of $\Delta'$. Represent $T'$ as a triangle in the complex plane carrying the diferential $dz^2$. Label the edges of $T'$ by $a,b,c$ so that
	\[
	|\Re(a)| = |\Re(b)| + |\Re(c)|.
	\]
	If one of the edges of $T'$ is vertical, more than one such labelling exists. On $T'$ consider the dual $1$-complex described in Figure \ref{fig:dual_complex}. Delete the edge opposite to $a$ as in Figure \ref{fig:dual_tt}. Consider the corresponding $1$-complexes on all the triangle of $\Delta'$. Joining these complexes along the edges of $\Delta'$ gives rise to a maximal train track $\tau'$ on the underlying Riemann surface. The corresponding maximal train track $\tau$ on $S_g$ carries $\Re(q) \in \mf$. For every edge $e'$ of $\Delta'$, the counting measure of the coresponding edge of $\tau$ is equal to $|\Re(e')|$, the absolute value of the real part of the holonomy of $e'$ with respect to $q$.
	
	\begin{figure}[h]
		\centering
		\begin{subfigure}[b]{0.4\textwidth}
			\centering
			\includegraphics[width=0.6\textwidth]{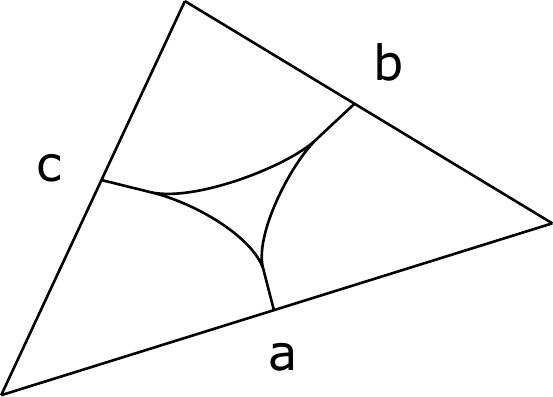}
			\caption{The dual $1$-complex.}
			\label{fig:dual_complex}
		\end{subfigure}
		\quad \quad \quad
		~ %add desired spacing between images, e. g. ~, \quad, \qquad, \hfill etc. 
		%(or a blank line to force the subfigure onto a new line)
		\begin{subfigure}[b]{0.4\textwidth}
			\centering
			\includegraphics[width=0.6\textwidth]{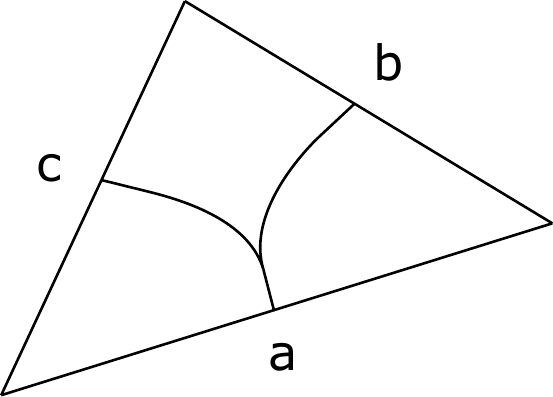}
			\caption{Deleting an edge.}
			\label{fig:dual_tt}
		\end{subfigure}
		\caption{The train track dual to a triangle of a Delaunay triangulation.} \label{fig:dual}
	\end{figure}
	
	Now fix a compact subset $\mathcal{K} \subseteq \tt$. By Lemma \ref{lem:finite_delaunay}, the Delaunay triangulations of marked quadratic differentials in $\pi^{-1}(\mathcal{K})$ induce finitely many marked triangulations on $S_g$. The procedure described above produces finitely many train tracks on $S_g$ for each one of these marked triangulations.
\end{proof}

We are now ready to prove Proposition \ref{prop:small_HM_integral_sum}.

\begin{proof}[Proof of Proposition \ref{prop:small_HM_integral_sum}]
	Fix a compact subset $\mathcal{K} \subseteq \mathcal{T}_g$, $X \in \mathcal{K}$, and $\delta > 0$. Recall that the Thurston measure $\nu$ on $\mf$ gives zero mass to the subset of singular measured foliations having a singularity with more than three prongs. Notice that, for every $q \in \qut \backslash \qut(\mathbf{1})$, $\Re(q) \in \mf$ belongs to this subset. This together with definition (\ref{eq:hm_3}) of the Hubbard-Masur function $\lambda$ implies
	\begin{equation}
	\label{eq:st1}
	\int_{S(X) \backslash p^{-1}(K_\delta(\mathbf{1}))} \lambda(q) \ ds_X(q) \leq  \nu(\{\eta \in \mf \ | \ \mathrm{Ext}_\eta(X) \leq 1, \ \ell_{\min}(\Re|_{Q(X)}^{-1}(\eta)) < \delta\}).
	\end{equation}
	
	Consider the finite collection of maximal train tracks $\{\tau_i\}_{i=1}^n$ on $S_g$ provided by Lemma \ref{lem:finite_tt}. Fix $i \in \{1,\dots,n\}$. Recall that train track coordinates identify the $6g-6$ dimensional closed cone $V(\tau_i) \subseteq (\mathbf{R}_{\geq0})^{18g-18}$ of non-negative counting measures on the edges of $\tau_i$ satisfying the switch conditions with the closed subset $U(\tau_i) \subseteq \mf$ of singular measured foliations carried by $\tau_i$. Denote by $\text{Leb}_{\tau_i}$ the Lebesgue measure on $V(\tau_i)$ and by $\| \cdot\|$ the Euclidean norm on $\mathbf{R}^{18g-18}$. As the function $\mathrm{Ext} \colon \mf \times \tt \to \mathbf{R}_{>0}$ given by $\mathrm{Ext}(\eta,X) := \mathrm{Ext}_\eta(X)$ is positive and continuous, and as the subset $U(\tau_i) \subseteq \mf$ is projectively compact, $\| \eta \| \preceq_{\tau_i} \mathrm{Ext}_\eta(X)$ for every $\eta \in U(\tau_i) =  V(\tau_i)$. Using this fact together with Lemmas \ref{lem:delaunay} and \ref{lem:finite_tt}, we deduce that, for some constant $C = C(\mathcal{K})> 0$ depending only on $\mathcal{K}$,
	\begin{gather}
	\nu(\{\eta \in \mf \ | \ \mathrm{Ext}_\eta(X) \leq 1, \ \ell_{\min}(\Re|_{Q(X)}^{-1}(\eta)) < \delta\}) \label{eq:st2}\\
	\leq \sum_{i=1}^n \sum_{j=1}^{18g-18} \mathrm{Leb}_{\tau_i}(\{v \in V(\tau_i) \ | \ \| v \| \leq C, \ v_j \leq \delta\}). \nonumber
	\end{gather}
	An explicit computation shows that, for every $i \in \{1,\dots,n\}$ and every $j \in \{1,\dots,18g-18\}$,
	\begin{equation}
	\mathrm{Leb}_{\tau_i}(\{v \in V(\tau_i) \ | \ \| v \| \leq C, \ v_j \leq \delta\}) \preceq_{\mathcal{K}} \delta. \label{eq:st3}
	\end{equation}

	Putting together (\ref{eq:st1}), (\ref{eq:st2}), and (\ref{eq:st3}) finishes the proof.
\end{proof}

\section{The Masur-Veech measure in polar coordinates}

\subsection*{Outline of this section.} The goal of this section is to prove Theorem \ref{theo:polar_coordinates_estimate}. The bounds of Forni on the spectral gap of the Kontsevich-Zorich cocycle on $\qutp$ \cite{F02} and the bound of Kahn and Wright \cite{KW20} introduced in Proposition \ref{prop:KW_new} will play a crucial role in the proof.

\subsection*{The spectral gap of the Kontsevich-Zorich cocycle on $\boldsymbol{\qutp}$.} In \cite{F02}, Forni proved variational formulas for the Hodge norm along the Kontsevich-Zorich cocycle on strata of Abelian differentials. By considering canonical double covers, these formulas can be used to prove bounds on the spectral gap of the Kontsevich-Zorich cocycle on $\qutp$. We now give a brief account of these bounds.

Let $q \in \qut(\mathbf{1})$. Recall that $\pi \colon \qut \to \tt$ denotes the standard projection. Denote by $\widetilde{X}$ the canonical double cover of $X := \pi(q)$ induced by $q$ and by $\omega$ the canonical square root of the pullback of $q$ to $\widetilde{X}$. Every marked quadratic differential $a_t q \in \qut(\mathbf{1})$ along the Teichmüller geodesic flow orbit of $q$ induces a norm $\|\cdot\|_{a_t q}$ on $H^1_\odd(\widetilde{X};\mathbf{R})$ by restricting the Hodge inner product coming from the canonical double cover of $a_tq$. Denote by $\text{span}(\Re(\omega), \Im(\omega))^\perp \subseteq H^1_\text{odd}(\widetilde{X};\mathbf{R})$ the symplectic complement of $\text{span}(\Re(\omega), \Im(\omega))$ in $H^1_\text{odd}(\widetilde{X};\mathbf{R})$ with respect to cup intersection form. The following result was proved by Forni in \cite{F02}; a more explicit statement can be found in \cite[Lemma 4.2]{AF08}.

\begin{theorem}[\cite{F02}]
	\label{theo:spectral_gap}
	There exists a continuous function $E \colon \qutp \to [0,1)$ with the following property. Let $q \in \qut(\mathbf{1})$. Denote by $\widetilde{X}$ the canonical double cover of $X := \pi(q)$ induced by $q$ and by $\omega$ the canonical square root of the pullback of $q$ to $\widetilde{X}$. Let $\beta \in \text{span}(\Re(\omega),\Im(\omega))^\perp \subseteq H^1_{\text{odd}}(\widetilde{X};\mathbf{R})$. Then,
	\[
	\bigg\vert \frac{d}{dt}\bigg\vert_{t =0}  \left(\log(\| \beta \|_{a_tq})\right) \bigg\vert \leq E(q).
	\]
\end{theorem}

Recall that, for $q \in \qum$, $\ell_{\min}(q) > 0$ denotes the length of the shortest saddle connections of $q$. Recall that, for every $\epsilon > 0$ and every $t > 0$,
\begin{gather*}
K_\epsilon := \{ q \in \qum \ | \ \ell_{\min}(q) \geq \epsilon\},\\
K_\epsilon(t) := \{q \in \qum  \colon |\{s \in [0,t] \ | \ a_s q \in K_\epsilon \}| \geq (1/2) \thinspace t \}.
\end{gather*}
Recall that $p \colon \qut \to \qum$ denotes the quotient map. The following result is a direct consequence of Theorem \ref{theo:spectral_gap} and the fact that the subsets $K_\epsilon \subseteq \qum$ are compact. 

\newconc{c:good_traj}
\begin{corollary}
	\label{cor:good_trajectory}
	Let $q \in \qut(\mathbf{1})$. Denote by $\widetilde{X}$ the canonical double cover of $X := \pi(q)$ induced by $q$ and by $\omega$ the canonical square root of the pullback of $q$ to $\widetilde{X}$. Suppose $\epsilon > 0$ and $t>0$ are such that $q \in p^{-1}(K_\epsilon(t))$. Then, for every $\beta \in \text{span}(\Re(\omega),\Im(\omega))^\perp \subseteq H^1_{\text{odd}}(\widetilde{X};\mathbf{R})$,
	\[
	e^{-(1 - \useconc{c:good_traj})t} \cdot \| \beta \|_{q} \leq \| \beta \|_{a_t q} \leq e^{(1 - \useconc{c:good_traj})t} \cdot \| \beta \|_{q},
	\]
	where $\useconc{c:good_traj} = \useconc{c:good_traj}(g,\epsilon) > 0$ is a constant depending only on $g$ and $\epsilon$.
\end{corollary}

\begin{remark}
	The spectral gap estimate in Theorem \ref{theo:spectral_gap} does not give information about what happens as we consider quadratic differentials in $\qutp$ approaching the multiple zero locus. A recent result of Frankel \cite[Theorem 1.2]{F20} can be used to quantify this. More precisely, using Frankel's result, one can show that the constant $\useconc{c:good_traj} = \useconc{c:good_traj}(g,\epsilon) > 0$ in Corollary \ref{cor:good_trajectory} can be assumed to satisfy $\useconc{c:good_traj} \succeq_g \ell_{\min}(q)^2$.
\end{remark}

\subsection*{The pullback of the Thurston measure to $\boldsymbol{S(X)}$.} The proof of  Theorem \ref{theo:polar_coordinates_estimate} will use an alternative description of the volume forms $(\Re|_{S(X)})^*\overline{\nu}_X$ and $(\Im|_{S(X)})^*\overline{\nu}_X$ on $S(X) \cap \qutp$ introduced in \S 2. Fix $X \in \tt$. Let $q \in \pi^{-1}(X) \cap \qutp$. Denote by $\widetilde{X}$ the canonical double cover of $X$ induced by $q$. Recall that period coordinates identify $T_q Q(X)$ with $H^{1,0}_\odd(\widetilde{X};\mathbf{C}) \subseteq H^{1}_\odd(\widetilde{X};\mathbf{C})$. Consider the isomorphism $\Re \colon H^{1,0}_\odd(\widetilde{X};\mathbf{C}) \to H^{1}_\odd(\widetilde{X};\mathbf{R})$. Pull back the volume form induced by the cup intersection pairing on  $H^{1}_\odd(\widetilde{X};\mathbf{R})$ to $H^{1,0}_\odd(\widetilde{X};\mathbf{C})$. Contract the restriction of this volume form to $T_qS(X)$ by any vector $V \in T_qQ(X)$ satisfying $d\mathrm{Area}_q(V) = 1$. For instance, take $V := [\omega] \in H^{1,0}_\odd(\widetilde{X};\mathbf{C})$. These volume forms integrate to the smooth volume form $(\Re|_{S(X)})^*\overline{\nu}_X$ on $S(X) \cap \qutp$ \cite[Proof of Theorem 5.8]{Du15}. An analogous description can be provided for the volume form $(\Im|_{S(X)})^*\overline{\nu}_X$ on $S(X) \cap \qutp$.

\subsection*{The Masur-Veech measure in polar coordinates} We are now ready to prove Theorem \ref{theo:polar_coordinates_estimate}. Recall that $h :=6g-6$. Recall that, for every $X \in \mathcal{T}_g$, $\Phi_X \colon S(X) \times \mathbf{R}_{>0} \to \mathcal{T}_g$ denotes the map $\Phi_X(q,t) := \pi(a_tq)$ and that, for every $q \in S(X) \cap \qutp$ and every $t > 0$,
\begin{equation}
\label{eq:polar2}
|\Phi_X^*(\mathbf{m})(q,t)| = \Delta(q,t) \cdot |s_X(q) \wedge dt|,
\end{equation}
where $\Delta \colon (S(X) \cap \qutp) \times \mathbf{R}_{>0} \to \mathbf{R}_{>0}$ is the smooth, positive function we want to bound.

\begin{proof}[Proof of Theorem \ref{theo:polar_coordinates_estimate}]
	Fix $\mathcal{K} \subseteq \mathcal{T}_g$ compact. Let $q \in \mathcal{Q}^1 \mathcal{T}_g(\mathbf{1})$, $t > 0$, and $\epsilon > 0$ be such that $q, a_t q \in \pi^{-1}(\mathcal{K})$ and $q \in p^{-1}(K_\epsilon(t))$. Denote by $\widetilde{X}$ the canonical double cover of $X := \pi(q)$ induced by $q$ and by $\omega$ the canonical square root of the pullback of $q$ to $\widetilde{X}$. Let $A_t \colon H^1_\odd(\widetilde{X};\mathbf{R}) \to H_\odd^1(\widetilde{X};\mathbf{R})$ be the identity map. Endow the source of this map with the Hodge inner product $\langle \cdot, \cdot  \rangle_q$ induced by $q$ and the target with the Hodge inner product $\langle \cdot, \cdot  \rangle_{a_t q}$ induced by $a_tq$. This map corresponds precisely to the Kontsevich-Zorich cocycle on $\qutp$. To prove the desired estimate for $\Delta(q,t)$ we evaluate both sides of (\ref{eq:polar2}) on a convenient basis of $T_{(q,t)}(S(X) \times \mathbf{R}_{>0})$ constructed using a singular value decomposition of $A_t$.
	
	We begin by constructing a convenient basis of $T_{(q,t)}(S(X) \times \mathbf{R}_{>0})$. Recall that $\text{span}(\Re(\omega),\Im(\omega))^\perp \subseteq H^1_{\text{odd}}(\widetilde{X};\mathbf{R})$ is not only a symplectic complement but also an orthogonal complement with respect to either of the Hodge inner products $\langle \cdot, \cdot \rangle_{q}$ or $\langle \cdot, \cdot \rangle_{a_tq}$. Consider the restriction of $A_t$ to $\text{span}(\Re(\omega),\Im(\omega))^\perp$. Denote by $\lambda_2 \geq \dots \geq \lambda_{h-1}> 0$ the singular values of this restriction. As $q \in p^{-1}(K_\epsilon(\mathbf{1}))$, Corollary \ref{cor:good_trajectory} ensures 
	\begin{equation}
	\label{eq:lyapunov}
	e^{-(1 - \useconc{c:good_traj})t} \leq \lambda_{h-1} \leq \lambda_2 \leq e^{(1 - \useconc{c:good_traj})t}.
	\end{equation}
	Consider an orthonormal basis $u_2,\dots,u_{h-1}$ of $\text{span}(\Re(\omega),\allowbreak\Im(\omega))^\perp \subseteq H^1_\odd(\widetilde{X};\mathbf{R})$ with respect to $\langle \cdot, \cdot \rangle_q$ such that $A_t u_2,\dots, A_t u_{h-1}$ are orthogonal vectors with respect to $\langle \cdot, \cdot \rangle_{a_tq}$ and satisfy $\| A_t u_j \|_{a_tq} = \lambda_j$. In other words, consider a singular value decomposition of the restriction of $A_t$ to $\text{span}(\Re(\omega),\allowbreak\Im(\omega))^\perp$. Notice that $\|[\Im(\omega)]\|_q = 1$,  $\|[\Re(\omega)]\|_q = 1$, $\| A_t [\Im(\omega)]\|_{a_tq} = e^t$, and $\| A_t [\Re(\omega)] \|_{a_t q} = e^{-t}$. Let $\lambda_1 := e_1$ and $\lambda_h := e^{-t}$. Then, $\lambda_1 \geq \dots \geq \lambda_h$ are the singular values of $A_t$ and $[\Im(\omega)],u_2,\dots,u_{h-1}, [\Re(\omega)] \in H^1_\odd(\widetilde{X};\mathbf{R})$ is an orthonormal basis with respect to $\langle \cdot, \cdot \rangle_q$ of singular vectors of $A_t$.
	
	Recall that the volume forms on $H^1_\odd(\widetilde{X};\mathbf{R})$ induced by the cup intersection pairing and the Hodge inner products $\langle \cdot, \cdot \rangle_{q}$ and $\langle \cdot, \cdot \rangle_{a_tq}$ coincide. Clearly $A_t$ preserves the cup intersection pairing. It follows that
	\begin{equation}
		\label{eq:det}
		\prod_{j=1}^h \lambda_j = 1.
	\end{equation}
	
	We now use the singular value decomposition of $A_t$ introduced above to construct a convenient basis of $T_qS(X)$. Recall that period coordinates identify $T_q Q(X)$ with $H^{1,0}_\odd(\widetilde{X};\mathbf{C}) \subseteq H^{1}_\odd(\widetilde{X};\mathbf{C})$. A vector $v \in T_q Q(X)$ belongs to $T_qS(X)$ if and only if $d\mathrm{Area}_q(v) = 0$. Let $u_1 := [\Im(\omega)] \in E^u(q)$. Notice that $u_1,\dots,u_{h-1}$ is an orthonormal basis of $E^{uu}(q)$. Recall that $\star$ denotes the Hodge star operator on $H^1(\widetilde{X};\mathbf{C})$. For every $j \in \{1,\dots,h-1\}$, let $s_j := i \thinspace (\star u_j) \in H^1_\text{odd}(\widetilde{X};i\mathbf{R})$. In particular, $s_1 := [i\Re(\omega)]$. As $\star$ and multiplication by $i$ are orthogonal operators with respect to the Hodge inner product $\langle \cdot, \cdot \rangle_q$ on $H^1(\widetilde{X};\mathbf{C})$, $s_1,\dots,s_{h-1}$ is an orthonormal basis of $E^{ss}(q)$. For every $j \in \{1,\dots,h-1\}$, let $e_j := u_j + s_j\in H^{1,0}_\text{odd}(\widetilde{X}) = T_qQ(X)$. A direct computation shows that $d \text{Area}_q(e_j) = 0$. It follows that $e_1,\dots,e_{h-1}$ is a basis of $T_qS(X)$.\\
	
	We evaluate the right hand side of (\ref{eq:polar2}) on the basis $e_1,\dots,e_{h-1},\partial t$ of $T_{(q,t)}(S(X) \times \mathbf{R}_{>0})$. Notice that
	\begin{equation}
	|s_X(q) \wedge dt|(e_1,\dots,e_{h-1},\partial t) = |s_X(q)|(e_1,\dots,e_{h-1}). \label{eq:p1}
	\end{equation}
	By definition (\ref{eq:hm_3}) of the Hubbard-Masur function $\lambda$,
	\begin{equation}
	|s_X(q)|(e_1,\dots,e_{h-1}) = \lambda(q)^{-1} \cdot |((\Re|_{S(X)})^* \overline{\nu}_X)(q)|(e_1 \wedge \dots \wedge e_{h-1}). \label{eq:p2}
	\end{equation}
	As in the proof of Proposition \ref{prop:HM_bound_sum}, denote by $\mu_{H_\odd^1(\widetilde{X};\mathbf{R})}$ the volume form induced by the cup intersection pairing on $H_\odd^1(\widetilde{X};\mathbf{R})$. The alternative definition of the volume form $(\Re|_{S(X)})^*\overline{\nu}_X$ introduced above ensures
	\begin{equation}
	|((\Re|_{S(X)})^* \overline{\nu}_X)(q)|(e_1 \wedge \dots \wedge e_{h-1}) = |\mu_{H^1_\odd(\widetilde{X};\mathbf{R})}|(u_1\wedge \dots \wedge u_{h-1} \wedge [\Re(\omega)]). \label{eq:p3}
	\end{equation}
	As $u_1, \dots u_{h-1},[\Re(\omega)]$ is an orthonormal basis of $H^1_\odd(\widetilde{X};\mathbf{R})$, 
	\begin{equation}
	|\mu_{H^1_\odd(\widetilde{X};\mathbf{R})}|(u_1 \wedge \dots \wedge u_{h-1} \wedge [\Re(\omega)]) = 1. \label{eq:p4}
	\end{equation}
	Putting together (\ref{eq:p1}), (\ref{eq:p2}), (\ref{eq:p3}), and (\ref{eq:p4}) we deduce
	\begin{equation}
	|s_X(q) \wedge dt|(e_1,\dots,e_{h-1},\partial t) = 1. \label{eq:end1}
	\end{equation}
	
	We now evaluate the left hand side of (\ref{eq:polar2}) on the basis $e_1,\dots,e_{h-1},\partial t$ of $T_{(q,t)}(S(X) \times \mathbf{R}_{>0})$. Notice that $\Phi_X =  \pi \circ A_X$, where $A_X \colon S(X) \times \mathbf{R}_{>0} \to \qut$ is the map given by $A_X(q',s) := a_s q'$. Identify the tangent space $T_{a_t q} \qt$ with $H^1_\odd(\widetilde{X};\mathbf{C})$ using the Gauss-Manin connection. Under this identification, $d (A_X)_{(q,t)} e_j = e^t u_j + e^{-t} s_j$ for every $j \in \{1,\dots,h-1\}$ and $d (A_X)_{(q,t)} \partial t = [a_t\overline{\omega}] := e^t [\Re(\omega)] - e^{-t} [\Im(\omega)]$, the tangent direction of the Teichmüller geodesic flow at $a_tq \in \qut$. Consider the decomposition
	\[
	\bigwedge_{j=1}^{h-1} e_j = \sum_{P} \left( \bigwedge_{j \in P} u_j \right) \wedge \left( \bigwedge_{j \notin P} s_j \right),
	\]
	where the sum runs over all subsets $P \subseteq \{1,\dots,h-1\}$. Using this decomposition and the derivatives of $A_X$ computed above we can write
	\begin{gather}
	|\Phi_X^*(\mathbf{m})(q,t)|(e_1\wedge \dots \wedge e_{h-1} \wedge \partial t) \label{eq:p5}\\
	= \sum_{P} |\mathbf{m}_{\pi(a_tq)}| \left(\left( \bigwedge_{j \in P} e^t d\pi_{a_tq} u_j \right) \wedge \left( \bigwedge_{j \notin P} e^{-t} d\pi_{a_tq} s_j \right) \wedge d\pi_{a_tq} [a_t\overline{\omega}]\right). \nonumber
	\end{gather}
	The same estimates used in the proof of Proposition \ref{prop:HM_bound_sum} show that, for every subset $P \subseteq \{1,\dots,h-1\}$,
	\begin{gather}
	|\mathbf{m}_{\pi(a_tq)}| \left(\left( \bigwedge_{j \in P} e^t d\pi_{a_tq} u_j \right) \wedge \left( \bigwedge_{j \notin P} e^{-t} d\pi_{a_tq} s_j \right) \wedge d\pi_{a_tq} [a_t \overline{\omega}]\right)  \preceq_{\mathcal{K}} \ell_{\min}(a_tq)^{-(h-1)} \cdot \left(\prod_{j\in P} e^t \lambda_j\right). \label{eq:p6}
	\end{gather}
	For $P \neq \{1,\dots,h-1\}$, (\ref{eq:lyapunov}) and (\ref{eq:det}) ensure
	\begin{equation}
	\prod_{j\in P} e^t \lambda_j \leq e^{(h- \useconc{c:good_traj})t}. \label{eq:p7}
	\end{equation}
	For $P = \{1,\dots,h-1\}$, definition $(\ref{eq:hm_2})$ of the Hubbard-Masur function $\lambda$ ensures
	\begin{equation}
	 |\mathbf{m}_{\pi(a_tq)}|(e^t d\pi_{a_tq} u_1 \wedge \dots \wedge e^t d\pi_{a_tq}  u_{h-1} \wedge d\pi_{a_t q} [a_t \overline{\omega}]) =  e^{(h-1)t} \cdot\lambda(a_tq)  \cdot |\mu_{\alpha^{u}(a_tq)}|(u_1\wedge\dots\wedge u_{h-1} \wedge  [a_t \overline{\omega}]). \label{eq:p8}
	\end{equation}
	Directly from the definition of the volume form  $\mu_{\alpha^{u}(a_tq)}$ we see that
	\begin{equation}
	|\mu_{\alpha^{u}(a_tq)}|(u_1 \wedge \dots \wedge u_{h-1} \wedge da_t [\overline{\omega}]) = |\mu_{\alpha^{uu}(a_tq)}|(u_1 \wedge \dots \wedge u_{h-1}). \label{eq:p9}
	\end{equation}
	Notice that $d\mathrm{Area}_{a_t q}([e^t \Re(\omega)]) = 1$. It follows from the definition of the volume form $\mu_{\alpha^{u}(a_tq)}$ that
	\begin{equation}
	 |\mu_{\alpha^{u}(a_tq)}|(u_1 \wedge \dots \wedge u_{h-1}) =   |\mu_{H^1_\odd(\widetilde{X};\mathbf{R})}|(u_1 \wedge \dots \wedge u_{h-1} \wedge [e^{t} \Re(\omega)]). \label{eq:p10}
	\end{equation}
	By (\ref{eq:p4}),
	\begin{equation}
	|\mu_{H^1_\odd(\widetilde{X};\mathbf{R})}|(u_1 \wedge \dots \wedge u_{h-1} \wedge [e^{t} \Re(\omega)])
	= e^t. \label{eq:p11}
	\end{equation}
	Putting together (\ref{eq:p5}), (\ref{eq:p6}), (\ref{eq:p7}), (\ref{eq:p8}), (\ref{eq:p9}), (\ref{eq:p10}), and (\ref{eq:p11}) we see that
	\begin{equation}
	|\Phi_X^*(\mathbf{m})(q,t)|(e_1\wedge \dots \wedge e_{h-1} \wedge \partial t) = \lambda(a_tq) \cdot e^{ht} + O_\mathcal{K}(\ell_{\min}(a_tq)^{-(h-1)} \cdot e^{(h-\useconc{c:good_traj}) t}). \label{eq:end2}
	\end{equation}
	
	From (\ref{eq:polar2}), (\ref{eq:end1}), and (\ref{eq:end2}) we deduce
	\[
	\Delta(q,t) = \lambda (q) \cdot \lambda(a_tq) \cdot e^{ht} + O_\mathcal{K}(\lambda(q)   \cdot \ell_{\min}(a_tq)^{-(h-1)} \cdot e^{(h-\useconc{c:good_traj}) t}).
	\]
	Using Proposition \ref{prop:HM_bound_sum} to bound $\lambda(q)$ in the error term finishes the proof.
\end{proof}

\section{Estimates near the multiple zero locus}

\subsection*{Outline of this section.} In this section we prove Theorems \ref{theo:thin_trajectories_sum}, \ref{theo:thin_sector_sum}, and \ref{theo:large_deviations_sum}. These theorems bound the contributions near the multiple zero locus that show up in the proof of Theorem \ref{theo:ball_equidistribution} and which cannot be controlled using Theorem \ref{theo:polar_coordinates_estimate}. The proofs use results from \cite{EMR19}, \cite{ABEM12}, and \cite{A06}.

\subsection*{Thin Teichmüller geodesic segments.} We use work of Eskin, Mirzakhani, and Rafi \cite{EMR19} to bound the measure of the set of Teichmüller geodesic segments that spend half of their time outside a large compact subset. We begin with a brief account of the results in \cite{EMR19} that are relevant for our purposes.

Recall $h:= 6g-6$. Recall that $p \colon \qut \to \qum$ denotes the quotient map, and that, for every $\epsilon> 0$, 
\[
K_\epsilon := \{ q \in \qum \ | \ \ell_{\min}(q) \geq \epsilon\}.
\]
Recall that $B_R(X) \subseteq \tt$ denotes the ball of radius $R > 0$ centered at $X \in \tt$ with respect to the Teichmüller metric. Let $X,Y \in \mathcal{T}_g$, $r > 0$, and $\epsilon > 0$. For every $R > 0$ denote by $N_R(X,Y,r,\epsilon)$ the number of mapping classes $\mc \in \mcg$ such that $\mc.Y \in B_R(X)$ and such that there exist a Teichmüller geodesic segment joining $X$ to $Z \in B_{r}(\mc.Y)$ which spends more than half of the time outside $p^{-1}(K_\epsilon)$. Notice that $N_R(X,Y,r,\epsilon)$ is independent of the markings of $X,Y \in \tt$. The following theorem is a direct consequence of \cite[Theorem 1.7 and Lemma 3.9]{EMR19}.

\newcone{e:EMR} \newconk{k:EMR} \newconr{r:EMR}
\begin{theorem}
	\label{theo:EMR}
	There exist constants $\useconr{r:EMR} = \useconr{r:EMR}(g) > 0$ and $\usecone{e:EMR} = \usecone{e:EMR}(g) > 0$ such that for every compact subset $\mathcal{K} \subseteq \tt$, every $X,Y \in \mathcal{K}$, every $0 < r < \useconr{r:EMR}$, every $0 < \epsilon < \usecone{e:EMR}$, and every $R > 0$,
	\[
	N_R(X,Y,r,\epsilon) \preceq_{\mathcal{K}} e^{(h-\useconk{k:EMR})R},
	\]
	where $\useconk{k:EMR} = \useconk{k:EMR}(g) > 0$ is a constant depending only on $g$.
\end{theorem}

Let $\mathcal{K} \subseteq \tt$ be a compact subset, $X \in \mathcal{K}$, $R>0$, and $\epsilon > 0$. Recall that $B_R(X,\mathcal{K},K_\epsilon) \subseteq \mathcal{T}_g$ denotes the set of all $Y \in B_R(X) \cap \mathrm{Mod}_g \cdot \mathcal{K}$ such that the Teichmüller geodesic segment joining $X$ to $Y$ spends more than half of the time outside $p^{-1}(K_\epsilon)$. Given $A \subseteq \mathcal{T}_g$ and $r > 0$, $\text{Nbhd}_r(A) \subseteq \tt$ will denote the set of points in $\mathcal{T}_g$ at Teichmüller distance at most $r$ from $A$. Recall that $\mathbf{m} := \pi_* \mu$ denotes the pushforward to $\tt$ of the Masur-Veech measure $\mu$ on $\qut$ under the projection $\pi \colon \qut \to \tt$ and that $\widehat{\mathbf{m}}$ denotes the local pushforward of $\mathbf{m}$ to $\mm$. The following result is a stonger version of Theorem \ref{theo:thin_trajectories_sum} that will be used in the proof of Theorem \ref{theo:thin_sector_sum}.

\newcone{e:thint} \newconk{k:thint} \newconr{r:thint}
\begin{theorem}
	\label{theo:thin_trajectories}
	There exist constants $\useconr{r:thint} = \useconr{r:thint}(g) > 0$ and $\usecone{e:thint} = \usecone{e:thint}(g) > 0$ such that for every compact subset $\mathcal{K} \subseteq \tt$, every $X,Y \in \mathcal{K}$, every $0 < r < \useconr{r:EMR}$, every $0 < \epsilon < \usecone{e:thint}$, and every $R > 0$,
	\[
	\mathbf{m}(\mathrm{Nbhd}_{r}(B_R(X,\mathcal{K},K_\epsilon))) \preceq_{\mathcal{K}}  e^{(h-\useconk{k:thint})R},
	\]
	where $\useconk{k:thint} = \useconk{k:thint}(g) > 0$ is a constant depending only on $g$.
\end{theorem}

\begin{proof}
	Let $\useconr{r:thint} = \useconr{r:thint}(g) > 0$ and $\usecone{e:thint} = \usecone{e:thint}(g) > 0$ be as in Theorem \ref{theo:EMR}. Fix $\mathcal{K} \subseteq \tt$ compact, $X \in \mathcal{K}$, $0 < r < \useconr{r:thint}$ and $0 < \epsilon < \usecone{e:thint}$.  Directly from the definitions we see that, for every $Y \in \mathcal{M}_g$,
	\begin{equation*}
	\sum_{\mc \in \text{Mod}_g} \mathbbm{1}_{\text{Nbhd}_{r}(B_R(X,\mathcal{K},K_\epsilon))}(\mc \cdot Y) \leq	N_{R+r}(X,Y,r,\epsilon). %\label{eq:z1}
	\end{equation*}
	Integrating this inequality over $\pi(\text{Nbhd}_{r}(\mathcal{K})) \subseteq \mathcal{M}_g$ with respect to $d\widehat{\mathbf{m}}(Y)$ we get
	\begin{align}
	\int_{\pi(\text{Nbhd}_{r}(\mathcal{K}))} \sum_{\mc \in \text{Mod}_g} \mathbbm{1}_{\text{Nbhd}_{r}(B_R(X,\mathcal{K},K_\epsilon))}(\mc \cdot Y) \thinspace d\widehat{\mathbf{m}}(Y) 
	\leq \int_{\pi(\text{Nbhd}_{r}(\mathcal{K}))} N_{R+r}(X,Y,r,\epsilon) \thinspace d\widehat{\mathbf{m}}(Y). \label{eq:z2}
	\end{align}
	An unfolding argument shows that
	\begin{align}
	\int_{\pi(\text{Nbhd}_{r}(\mathcal{K}))} \sum_{\mc \in \text{Mod}_g} \mathbbm{1}_{\text{Nbhd}_{r}(B_R(X,\mathcal{K},K_\epsilon))}(\mc \cdot Y) \thinspace d\widehat{\mathbf{m}}(Y)  
	&= \int_{\mathcal{T}_g}  \mathbbm{1}_{\text{Nbhd}_{r}(B_R(X,\mathcal{K},K_\epsilon))}(Y) \thinspace d\mathbf{m}(Y) \label{eq:z3}\\
	&= \mathbf{m}(\text{Nbhd}_{r}(B_R(X,\mathcal{K},K_\epsilon))). \nonumber
	\end{align}
	Putting together (\ref{eq:z2}) and (\ref{eq:z3}) we deduce
	\begin{equation}
	\mathbf{m}(\text{Nbhd}_{r}(B_R(X,\mathcal{K},K_\epsilon))) \leq \int_{\pi(\text{Nbhd}_{r}(\mathcal{K}))} N_{R+r}(X,Y,r,\epsilon) \thinspace d\widehat{\mathbf{m}}(Y). \label{eq:z4}
	\end{equation}
	Theorem \ref{theo:EMR} ensures that, for every $Y \in \text{Nbhd}_{r}(\mathcal{K})$,
	\[
	N_{R+r}(X,Y,r,\epsilon) \preceq_{\mathcal{K}} e^{(h-\useconk{k:EMR})(R+r)} \preceq_g e^{(h-\useconk{k:EMR})R}
	\]
	Integrating this inequality over $\pi(\text{Nbhd}_{r}(\mathcal{K})) \subseteq \mathcal{M}_g$ with respect to $d\widehat{\mathbf{m}}(Y)$ we deduce
	\begin{align}
	\int_{\pi(\text{Nbhd}_{r}(\mathcal{K}))} N_{R+r}(X,Y,\epsilon) \thinspace d\widehat{\mathbf{m}}(Y)
	\preceq_\mathcal{K} \widehat{\mathbf{m}}(\pi(\text{Nbhd}_{r}(\mathcal{K}))) \cdot  e^{(h-\useconk{k:EMR})R} 
	\preceq_\mathcal{K} e^{(h-\useconk{k:EMR})R}. \label{eq:z5}
	\end{align}
	Putting together (\ref{eq:z4}) and (\ref{eq:z5}) we conclude
	\[
	\mathbf{m}(\text{Nbhd}_{r}(B_R(X,\mathcal{K},K_\epsilon))) \preceq_{\mathcal{K}} e^{(h-\useconk{k:EMR})R}. \qedhere
	\]
\end{proof}

\subsection*{The Euclidean metric.} The proof of Theorem \ref{theo:thin_sector_sum} will use  a couple of metrics on $\qut$ that interact nicely with period coordinates and the Teichmüller geodesic flow. We now introduce the first of these metrics. Period coordinates give rise to a local Lipschitz class of metrics on $\qut$. This class contains $\mcg$-invariant metrics induced by $\mcg$-invariant families of norms on the fibers of the tangent bundle of $\qut$. Explicit constructions of these metrics can be found in \cite[\S3.5]{ABEM12} and \cite[Definition 3.2]{F17}. For the rest of this paper we fix one such metric, denote it by $d_E$, and refer to it as the Euclidean metric of $\qut$. We denote by $\| \cdot \|_E$ the corresponding norms on the fibers of the tangent bundle of $\qut$. On $\mathbf{C}^{6g-6}$ consider the standard Euclidean norm $\| \cdot \|$. The Euclidean metric satisfies the following fundamental property.

\begin{lemma}
	\label{lem:euc_fund_prop}
	Let $\mathcal{K} \subseteq \tt$ be a compact subset. Consider a period coordinate chart $\Phi \colon U \to \mathbf{C}^{6g-6}$ defined on an open subset $U \subseteq \qut(\mathbf{1})$. Then, for any $q \in U \cap \pi^{-1}(\mathcal{K})$ and any $v \in T_q \qut$,
	\[
	\| v \|_E \asymp_{\Phi,\mathcal{K}} \|d\Phi_q v\|.
	\]
\end{lemma}

Recall that $d_\mathcal{T}$ denotes the Teichmüller metric on $\tt$. The following result, which corresponds to \cite[Lemma 3.9]{ABEM12}, can be proved using a compactness argument. See \cite[Theorem 1.2]{F17} for a related estimate.

\begin{lemma}
	\label{lem:euclidean_teichmuller_1}
	Let $\mathcal{K} \subseteq \tt$ be a compact subset. There exists a continuous function $f_1 \colon \mathbf{R}_{\geq 0} \to \mathbf{R}_{\geq 0}$ with $f_1(r) \to 0$ as $r \to 0$ such that for any pair $q_1,q_2 \in \pi^{-1}(\mathcal{K})$,
	\[
	d_\mathcal{T}(\pi(q_1),\pi(q_2)) \leq f_1(d_E(q_1,q_2)).
	\]
	Moreover, there exists a continuous function $f_2 \colon \mathbf{R}_{\geq 0} \to \mathbf{R}_{\geq 0}$ with $f_2(r) \to 0$ as $r \to 0$ such that for any pair $q_1, q_2 \in \pi^{-1}(\mathcal{K})$ on the same leaf of $\mathcal{F}^{u}$ or $\mathcal{F}^s$,
	\[
	d_E(q_1,q_2) \leq f_2(d_\mathcal{T}(\pi(q_1),\pi(q_2))).
	\]
\end{lemma}

Directly from Lemma \ref{lem:euclidean_teichmuller_1} we deduce the following corollary.

\newcond{euc_teich}
\begin{corollary}
	\label{cor:euclidean_teichmuller_1}
	Let $\mathcal{K} \subseteq \mathcal{T}_g$ compact. There exists a constant $\usecond{euc_teich} = \usecond{euc_teich}(\mathcal{K})>0$ such that if $q_1,q_2 \in \mathcal{Q}^1\mathcal{T}_g$ satisfy $q_1 \in \pi^{-1}(\mathcal{K})$ and $d_E(q_1,q_2) \leq \usecond{euc_teich}$, then $d_\mathcal{T}(\pi(q_1),\pi(q_2)) \leq 1$. In particular, $q_2 \in \pi^{-1}(\mathrm{Nbhd}_1(\mathcal{K}))$.
\end{corollary}

\subsection*{The modified Hodge metric.} The Hodge norm is a convenient norm to consider on the fibers of the tangent bundle of $\qut(\mathbf{1})$ due to how nicely it behaves with respect to the Teichmüller geodesic flow; see Theorem \ref{theo:spectral_gap}, for instance. Issues arise when trying to use the Hodge norm  to define a metric on all of $\qut$. Indeed, the Hodge norm in some directions tangent to the multiple zero locus might vanish. 

In \cite[\S 3.3.2]{ABEM12}, a modification of the Hodge norm is introduced to circumvent these issues. The modified Hodge norm can be used to define a $\mcg$-invariant metric along the leaves of the strongly stable and stongly unstable foliations $\mathcal{F}^{ss}$ and $\mathcal{F}^{uu}$ of $\qut$. We denote this metric by $d_H$ and refer to it as the modified Hodge metric. For the rest of this paper, we denote by $\epsilon_\mathrm{m} > 0$ the Margulis constant of the hyperbolic plane. The construction of the modified Hodge norm, and so the definition of the modified modified Hodge metric, depends on the choice of a parameter $0 < \epsilon_{\mathrm{sb}} \label{e:sb} < \epsilon_{\mathrm{m}}$ used to define \textit{short bases}. 

The following important theorem, corresponding to \cite[Theorem 3.10]{ABEM12}, shows that the modified Hodge metric is comparable to the Euclidean metric. In particular, by Lemma \ref{lem:euclidean_teichmuller_1}, the modified Hodge metric is absolutely continuous with respect to the Teichmüller metric.

\begin{theorem}
	\label{theo:euclidean_hodge_1}
	Let $\mathcal{K} \subseteq \mathcal{T}_g$ compact. For every pair $q_1,q_2 \in \pi^{-1}(\mathcal{K})$ on the same leaf of $\mathcal{F}^{uu}$ or $\mathcal{F}^{ss}$,
	\[
	d_E(q_1,q_2) \preceq_\mathcal{K} d_H(q_1,q_2).
	\]
	Moreover, if  $q_1,q_2 \in \pi^{-1}(\mathcal{K})$ are on the same leaf of $\mathcal{F}^{uu}$ or $\mathcal{F}^{ss}$ and satisfy $d_E(q_1,q_2) < 1$, then
	\[
	d_H(q_1,q_2) \preceq_{\mathcal{K}} d_E(q_1,d_2) \cdot | \log d_E(q_1,q_2)|^{1/2}.
	\]
\end{theorem}

Directly from Lemma \ref{lem:euclidean_teichmuller_1} and Theorem \ref{theo:euclidean_hodge_1} we deduce the following corollary.

\newcond{d:hodge_teich}
\begin{corollary}
	\label{cor:euclidean_hodge_teichmuller_1}
	Let $\mathcal{K} \subseteq \mathcal{T}_g$ be a compact subset. There exists a constant $\usecond{d:hodge_teich}:= \usecond{d:hodge_teich}(\mathcal{K})>0$ such that if $q_1,q_2 \in \mathcal{Q}^1\mathcal{T}_g$ are on the same leaf of $\mathcal{F}^{uu}$ or $\mathcal{F}^{ss}$ and satisfy $q_1 \in \pi^{-1}(\mathcal{K})$ and $d_H(q_1,q_2) \leq \usecond{d:hodge_teich}$, then $d_\mathcal{T}(\pi(q_1),\pi(q_2)) \leq 1$. In particular, $q_2 \in \pi^{-1}(\mathrm{Nbhd}_1(\mathcal{K}))$.
\end{corollary}

The modified Hodge metric behaves nicely with respect to the Teichmüller geodesic flow. Recall that, for every $\epsilon > 0$ and every $t>0$, 
\[
K_\epsilon(t) := \{q \in \qum  \colon |\{s \in [0,t] \ | \ a_s q \in K_\epsilon \}| \geq (1/2)t \}.
\]
The following theorem corresponds to \cite[Theorem 3.15]{ABEM12}. 

\newconcc{cc:hodge_decay_stable} \newconcc{cc:hds2} \newconc{c:hds}
\begin{theorem}
	\label{theo:hodge_distance_decay_stable}
	Let $0 < \epsilon_{\mathrm{sb}} < \epsilon_{\mathrm{m}}$ be the parameter used to define short bases. There exists a constant $\useconcc{cc:hodge_decay_stable} = \useconcc{cc:hodge_decay_stable}(g,\epsilon_\mathrm{sb}) > 0$ such that for every $q_1,q_2 \in \qut$ on the same leaf of $\mathcal{F}^{ss}$ and every $t \geq 0$,
	\[
	d_H(a_{t} q_1, a_{t} q_2) \leq \useconcc{cc:hodge_decay_stable} \cdot d_H(q_1,q_2). 
	\]
	Let $\epsilon > 2  \epsilon_{\mathrm{sb}}$. Consider $q_1,q_2 \in \mathcal{Q}^1\mathcal{T}_g$ on the same leaf of $\mathcal{F}^{ss}$ with $d_H(q_1,q_2) < 1$. Suppose that $t >0$ is such that $q \in p^{-1}(K_{\epsilon}(t))$. Then, for every $0< s < t$,
	\begin{equation*}
	\label{eq:hodge_contraction_stable}
	d_H(a_{s} q_1, a_{s} q_2) \leq \useconcc{cc:hds2} \cdot e^{-\useconc{c:hds}s} \cdot d_H(q_1,q_2),
	\end{equation*}
	where $\useconcc{cc:hds2} = \useconcc{cc:hds2}(g,\epsilon) > 0$ and $\useconc{c:hds} = \useconc{c:hds} (g,\epsilon) > 0$ are constants depending only on $g$ and $\epsilon$.
\end{theorem}

For every $\epsilon > 0$ and every  $t>0$ denote
\[
K_\epsilon(-t) := \{q \in \qum  \colon |\{s \in [-t,0] \ | \ a_s q \in K_\epsilon \}| \geq (1/2)t \}.
\]
For future reference we also state the following version of Theorem \ref{theo:hodge_distance_decay_stable} for strongly unstable leaves.

\newconcc{cc:hdu1}  \newconcc{cc:hdu2}  \newconc{c:hdu}
\begin{theorem}
	\label{theo:hodge_distance_decay_unstable}
	Let $0 < \epsilon_{\mathrm{sb}} < \epsilon_{\mathrm{m}}$ be the parameter used to define short bases. There exists a constant $\useconcc{cc:hdu1} = \useconcc{cc:hdu1}(g,\epsilon_\mathrm{sb}) > 0$ such that for every $q_1,q_2 \in \qut$ on the same leaf of $\mathcal{F}^{uu}$ and every $t \geq 0$,
	\[
	d_H(a_{-t} q_1, a_{-t} q_2) \leq \useconcc{cc:hdu1} \cdot d_H(q_1,q_2). 
	\]
	Let $\epsilon > 2 \epsilon_\mathrm{sb}$. Consider $q_1,q_2 \in \mathcal{Q}^1\mathcal{T}_g$ on the same leaf of $\mathcal{F}^{uu}$ with $d_H(q_1,q_2) < 1$. Suppose that $t >0$ is such that $q \in p^{-1}(K_{\epsilon}(-t))$. Then, for every $0< s < t$,
	\begin{equation*}
	\label{eq:hodge_contraction}
	d_H(a_{-s} q_1, a_{-s} q_2) \leq \useconcc{cc:hdu2} \cdot e^{-\useconc{c:hdu}s} \cdot d_H(q_1,q_2),
	\end{equation*}
	where $\useconcc{cc:hdu2} := \useconcc{cc:hdu2}(g,\epsilon) > 0$ and $\useconc{c:hdu} := \useconc{c:hdu}(g,\epsilon) > 0$ are constants depending only on $g$ and $\epsilon$.
\end{theorem}

\subsection*{Bounding the measure of sectors in Teichmüller space.} We prove a general bound for the measure of a sector in Teichmüller space. This bound will be the main tool used in the proof of Theorem \ref{theo:thin_sector_sum}. Recall that $q_s \colon \mathcal{T}_g \times \mathcal{T}_g \to \mathcal{Q}^1\mathcal{T}_g$ denotes the map which to every pair $X,Y \in \mathcal{T}_g$ assigns the quadratic differential $q_s(X,Y) \in S(X)$ corresponding to the tangent direction at $X$ of the unique Teichmüller geodesic segment from $X$ to $Y$ and that, given $X \in \mathcal{T}_g$ and $V \subseteq S(X)$, 
\[
\text{Sect}_V(X) := \{Y \in \mathcal{T}_g \ | \ q_s(X,Y) \in V \}.
\] 
Recall that $\nu$ denotes the Thurston measure on $\mf$. Denote by $\overline{\nu}$ the function which to every measurable subset $A \subseteq \mf$ assigns the value $\overline{\nu}(A) = \nu([0,1] \cdot A)$. Consider the map $\Re \colon \qut \to \mf$. Given $W \subseteq \mathcal{Q}^1\mathcal{T}_g$ and $s > 0$, denote by $W(s) \subseteq \mathcal{Q}^1\mathcal{T}_g$ the set of $q_1 \in \mathcal{Q}^1\mathcal{T}_g$ such that there exists $q_2 \in W$ on the same leaf of $\mathcal{F}^{uu}$ as $q_1$ satisfying $d_H(q_1,q_2)<s$.  We prove the following general bound. 

\newconk{k:thin_sect_gen} \newconn{n:thin_sect_gen} \newconr{r:thin_sect_gen}
\begin{theorem}
	\label{theo:thin_sector_general}
	There exists a constant $\useconr{r:thin_sect_gen} = \useconr{r:thin_sect_gen}(g) > 0$ such that for every compact subset $\mathcal{K} \subseteq \tt$, every $X \in \mathcal{K}$, every measurable subset $V \subseteq S(X)$, every $0 < r < \useconr{r:thin_sect_gen}$, and every $R > 0$,
	\begin{align*}
	\mathbf{m}(\text{Nbhd}_{r}(B_R(X) \cap \text{Sect}_{V}(X) \cap \mathrm{Mod}_g \cdot \mathcal{K})) 
	\preceq_\mathcal{K} \overline{\nu}(\Re(V( \useconn{n:thin_sect_gen} e^{-\useconk{k:thin_sect_gen}  R}))) \cdot e^{hR} + e^{(h-\useconk{k:thin_sect_gen})R},
	\end{align*}
	where $\useconn{n:thin_sect_gen} = \useconn{n:thin_sect_gen}(g) > 0$ and $\useconk{k:thin_sect_gen} = \newconk{k:thin_sect_gen}(g) > 0$ are constants depending only on $g$.
\end{theorem}

To prove Theorem \ref{theo:thin_sector_general} we will use the following lemma, which corresponds to a special case of \cite[Lemma 4.1]{ABEM12}. This lemma is a consequence of Lemma \ref{lem:euclidean_teichmuller_1} and Theorems \ref{theo:euclidean_hodge_1} and \ref{theo:hodge_distance_decay_stable}.

\begin{lemma}
	\label{lem:crucial_m_bound_new}
	Let $\mathcal{K} \subseteq \mathcal{T}_g$ be a compact subset, $U \subseteq \pi^{-1}(\mathcal{K})$ be a measurable subset, and $t > 0$. Denote $W := a_t U \cap \mathrm{Mod}_g \cdot \pi^{-1}(\mathcal{K})$. Then,
	\[
	\mathbf{m}(\text{Nbhd}_2(\pi(W))) \preceq_\mathcal{K} \overline{\nu}(\Re(W(1))).
	\]
\end{lemma}

The following lemma will also be used in the proof Theorem \ref{theo:thin_sector_general}.

\newconcc{cc:tas} \newconc{c:tas}
\begin{lemma}
	\label{lem:thin_annular_sectors}
	Let $0 < \epsilon_\mathrm{sb} < \epsilon_\mathrm{m}$ be the parameter used in the definition of short bases and $\usecone{e:thint} = \usecone{e:thint}(g) > 0$ be as in Theorem \ref{theo:thin_trajectories}. Assume $2 \epsilon_{\mathrm{sb}} < \usecone{e:thint}$. Let $\mathcal{K} \subseteq \tt$ compact, $X \in \mathcal{K}$, $V \subseteq S(X)$ measurable, $2 \epsilon_{\mathrm{sb}} < \epsilon < \usecone{e:thint}$, and $t > 0$. Denote $W := a_t (V \cap K_\epsilon(t)) \cap \mathrm{Mod}_g \cdot \pi^{-1}(\mathcal{K})$. Then, 
	\[
	\mathbf{m}(\mathrm{Nbhd}_2(\pi(W))) \preceq_\mathcal{K}  \overline{\nu}(\Re(V(\useconcc{cc:tas}e^{-\useconc{c:tas} t}))) \cdot e^{ht},
	\]
	where $\useconcc{cc:tas} = \useconcc{cc:tas}(g,\epsilon) > 0$ and $\useconc{c:tas} = \useconc{c:tas}(g,\epsilon) > 0$ are constants depending only on $g$ and $\epsilon$.
\end{lemma}

\begin{proof}
	Fix $2 \epsilon_{\mathrm{sb}} < \epsilon < \usecone{e:thint}$. Let $\useconcc{cc:hdu2} = \useconcc{cc:hdu2}(g,\epsilon) > 0$ and $\useconc{c:hdu} = \useconc{c:hdu}(\epsilon,g) > 0$ be as in Theorem \ref{theo:hodge_distance_decay_unstable}. Fix $\mathcal{K} \subseteq \tt$ compact, $X \in \mathcal{K}$, $V \subseteq S(X)$ measurable, and $t > 0$. Denote $W := a_t (V \cap K_\epsilon(t)) \cap \text{Mod}_g \cdot \pi^{-1}(\mathcal{K})$. Lemma \ref{lem:crucial_m_bound_new} ensures
	\begin{equation}
	\label{eq:x1}
	\mathbf{m}(\text{Nbhd}_2(\pi(W))) \preceq_\mathcal{K} \overline{\nu}(\Re(W(1))). 
	\end{equation}
	We claim $W(1) \subseteq a_t V(\useconcc{cc:hdu2} e^{-\useconc{c:hdu}t})$. Let $a_t q_1 \in W(1)$ be arbitrary. By the definition of $W(1)$, there exists $q_2 \in V \cap K_\epsilon(t)$ in the same leaf of $\mathcal{F}^{uu}$ as $q_1$ such that $d_H(a_tq_1,a_tq_2)<1$. It follows from Theorem \ref{theo:hodge_distance_decay_unstable} that
	\begin{align*}
	d_H(q_1,q_2) &= d_H(a_{-t} a_t q_1 , a_{-t} a_t q_2) 
	\leq \useconcc{cc:hdu2} \cdot e^{-\useconc{c:hdu}t} \cdot d_H(a_tq_1, a_tq_2)
	< \useconcc{cc:hdu2} \cdot e^{-\useconc{c:hdu}t}.
	\end{align*}
	This proves the claim. In particular, 
	\begin{equation}
	\overline{\nu}(\Re(W(1))) \leq \overline{\nu}(\Re(a_t V(\useconcc{cc:hdu2}e^{-\useconcc{c:hdu}t}))). \label{eq:x2}
	\end{equation}
	By the scaling property of the Thurston measure, 
	\begin{equation}
	\overline{\nu}(\Re(a_t V(\useconcc{cc:hdu2} e^{-\useconc{c:hdu}t}))) = e^{ht} \cdot \overline{\nu}(\Re(V(\useconcc{cc:hdu2} e^{-\useconc{c:hdu}t}))). \label{eq:x3}
	\end{equation}
	Putting together (\ref{eq:x1}), (\ref{eq:x2}), and (\ref{eq:x3}) we conclude
	\[
	\mathbf{m}(\text{Nbhd}_2(\pi(W))) \preceq_\mathcal{K} \overline{\nu}(\Re(V(\useconcc{cc:hdu2} e^{\useconc{c:hdu}t}))) \cdot e^{ht}. \qedhere
	\]
\end{proof}

We are now ready to prove Theorem \ref{theo:thin_sector_general}. 

\begin{proof}[Proof of Theorem \ref{theo:thin_sector_general}.]
	Let $\useconr{r:thint} =  \useconr{r:thint}(g) > 0$ be as in Theorem \ref{theo:thin_trajectories}. Consider $\useconr{r:thin_sect_gen} = \useconr{r:thin_sect_gen}(g) := \min\{\useconr{r:thint}/2,1\} >0$. Fix $0 < r < \useconr{r:thin_sect_gen}$. Let $0 < \epsilon_\mathrm{sb} < \epsilon_\mathrm{m}$ be the parameter used in the definition of short bases and $\usecone{e:thint} = \usecone{e:thint}(g) > 0$ be as in Theorem \ref{theo:thin_trajectories}. Assume $2 \epsilon_{\mathrm{sb}} < \usecone{e:thint}$. Fix $2 \epsilon_{\mathrm{sb}} < \epsilon = \epsilon(g) < \usecone{e:thint}$. Let $\useconcc{cc:tas} = \useconcc{cc:tas}(g,\epsilon) > 0$ and $\useconc{c:tas} = \useconc{c:tas}(\epsilon,g) > 0$ be as in Theorem \ref{theo:hodge_distance_decay_unstable}. Fix $\mathcal{K} \subseteq \tt$ compact. Let $\usecond{d:hodge_teich} = \usecond{d:hodge_teich}(\mathcal{K}) > 0$ be as in Corollary \ref{cor:euclidean_hodge_teichmuller_1}. Fix $X \in \mathcal{K}$, $V \subseteq S(X)$ measurable, and $R > 0$.  Consider the following decomposition
	\begin{gather}
	\text{Nbhd}_{r}(B_R(X) \cap \text{Sect}_{V}(X) \cap \text{Mod}_g \cdot \mathcal{K}) \label{eq:y1}\\
	= \ \text{Nbhd}_{r}(B_R(X) \cap \text{Sect}_{V}(X) \cap \text{Mod}_g \cdot \mathcal{K}) \backslash \text{Nbhd}_{2r}(B_R(X,\mathcal{K},K_\epsilon)) \nonumber\\ 
	\cup \thinspace \text{Nbhd}_{2r}(B_R(X,\mathcal{K},K_\epsilon)) \nonumber .
	\end{gather}
	As $0 < r < \useconr{r:thint}/2$, Theorem \ref{theo:thin_trajectories} ensures
	\begin{equation}
	\mathbf{m}(\text{Nbhd}_{2r}(B_R(X,\mathcal{K},K_\epsilon))) \preceq_\mathcal{K}  e^{(h-\useconk{k:thint})R}. \label{eq:y2}
	\end{equation}
	Notice that
	\begin{gather}
	\text{Nbhd}_{r}(B_R(X) \cap \text{Sect}_{V}(X) \cap \text{Mod}_g \cdot \mathcal{K}) \backslash \text{Nbhd}_{2r}(B_R(X,\mathcal{K},K_\epsilon)) \label{eq:y3}\\
	\subseteq \bigcup_{n=0}^{\lceil R/r \rceil} \text{Nbhd}_{2r}(\pi(a_{nr}(V \cap K_\epsilon(nr)) \cap \text{Mod}_g \cdot \pi^{-1}(\text{Nbhd}_{r}(\mathcal{K})))). \nonumber
	\end{gather}
	See Figure \ref{fig:annuli}. We decompose this union into three regimes. Let $h^\dagger := (h-1)/h$. Consider the integers
	\[
	n_1 = n_1(\mathcal{K},r) := \max \left\lbrace\left\lceil -\frac{1}{\useconc{c:tas} r} \log \left( \frac{\usecond{d:hodge_teich}}{\useconcc{cc:tas}}\right) \right\rceil, 0 \right\rbrace , \quad
	n_2 = n_2(g,r,R) := \left \lceil \frac{h^\dagger R}{r} \right \rceil,	
	\]
	so that $\useconcc{cc:tas} e^{-\useconc{c:tas}n r} \leq \usecond{d:hodge_teich}$ for every $n \geq n_1$ and $n r \geq h^\dagger R$ for every $n \geq n_2$. Write
	\begin{gather}
	\bigcup_{n=0}^{\lceil R/r \rceil} \text{Nbhd}_{2r}(\pi(a_{nr_0} (V \cap K_\epsilon(nr)) \cap \text{Mod}_g \cdot \pi^{-1}(\text{Nbhd}_{r}(\mathcal{K})))) \label{eq:y4}\\
	\subseteq \bigcup_{n=0}^{n_1} \text{Nbhd}_{2r}(\pi(a_{nr} (V  \cap K_\epsilon(nr)) \cap \text{Mod}_g \cdot \pi^{-1}(\text{Nbhd}_{r}(\mathcal{K})))) \nonumber\\
	\cup \bigcup_{n= n_1}^{n_2} \text{Nbhd}_{2r}(\pi(a_{nr} (V  \cap K_\epsilon(nr)) \cap \text{Mod}_g \cdot \pi^{-1}(\text{Nbhd}_{r}(\mathcal{K})))) \nonumber\\
	\ \cup \bigcup_{n= n_2}^{\lceil R/r \rceil} \text{Nbhd}_{2r}(\pi(a_{nr} (V \cap K_\epsilon(nr)) \cap \text{Mod}_g \cdot \pi^{-1}(\text{Nbhd}_{r}(\mathcal{K})))). \nonumber
	\end{gather}
	We bound the measure of each term in this decomposition separately. To control the first term notice that
	\begin{align*}
	\bigcup_{n=0}^{n_1} \text{Nbhd}_{2r}(\pi(a_{nr} (V \cap K_\epsilon(nr)) \cap \text{Mod}_g \cdot \pi^{-1}(\text{Nbhd}_{r}(\mathcal{K}))))
	\subseteq \text{Nbhd}_{(n_1+2)r}(\mathcal{K}).
	\end{align*}
	In particular, using the definition of $n_1$ and the fact that $0 < r < 1$ we deduce
	\begin{align}
	\mathbf{m} \left(\bigcup_{n=0}^{n_1} \text{Nbhd}_{2r}(\pi(a_{nr} (V  \cap K_\epsilon(nr)) \cap \text{Mod}_g \cdot \pi^{-1}(\text{Nbhd}_{r}(\mathcal{K}))))\right)
	\leq \mathbf{m} \left(\text{Nbhd}_{(n_1+2)r}(\mathcal{K})\right)
	\preceq_{\mathcal{K}} 1. \label{eq:y5}
	\end{align}
	To control the second term of the decomposition in (\ref{eq:y4}) we write
	\begin{gather}
	\mathbf{m}\left( \bigcup_{n=n_1}^{n_2} \text{Nbhd}_{2r}(\pi(a_{nr} (V \cap K_\epsilon(nr)) \cap \text{Mod}_g \cdot \pi^{-1}(\text{Nbhd}_{r}(\mathcal{K}))))\right) \label{eq:y6} \\
	\leq  \sum_{n=n_1}^{n_2} \mathbf{m} \left(\text{Nbhd}_{2r}(\pi(a_{nr} (V \cap K_\epsilon(nr)) \cap \text{Mod}_g \cdot \pi^{-1}(\text{Nbhd}_{r}(\mathcal{K}))))\right). \nonumber
	\end{gather}
	As $0 < r <1$, Lemma \ref{lem:thin_annular_sectors} ensures
	\begin{gather}
	\sum_{n=n_1}^{n_2} \mathbf{m} \left(\text{Nbhd}_{2r}(\pi(a_{nr} (V \cap K_\epsilon(nr)) \cap \text{Mod}_g \cdot \pi^{-1}(\text{Nbhd}_{r}(\mathcal{K}))))\right) \label{eq:y7} \\
	\preceq_\mathcal{K} \sum_{n=n_1}^{n_2} \overline{\nu}(\Re(V(\useconcc{cc:tas}e^{-\useconc{c:tas}nr}))) \cdot e^{hnr}. \nonumber
	\end{gather}
	Using Lemma \ref{cor:euclidean_hodge_teichmuller_1} and the definition of $n_1$ we bound
	\begin{align}
	\sum_{n=n_1}^{n_2} \overline{\nu}(\Re(V(\useconcc{cc:tas}e^{-\useconc{c:tas}nr}))) \cdot e^{hnr} \leq \overline{\nu}(\Re(\pi^{-1}(\text{Nbhd}_1(\mathcal{K}))))  \cdot \sum_{n=n_1}^{n_2} e^{hnr} \preceq_{\mathcal{K}} e^{(h-1)R}. \label{eq:y8}
	\end{align}
	Putting together (\ref{eq:y6}), (\ref{eq:y7}), and (\ref{eq:y8}) we deduce
	\begin{equation}
	\mathbf{m}\left( \bigcup_{n=n_1}^{n_2} \text{Nbhd}_{2r}(\pi(a_{nr} (V \cap K_\epsilon(nr)) \cap \text{Mod}_g \cdot \pi^{-1}(\text{Nbhd}_{r}(\mathcal{K}))))\right) \preceq_{\mathcal{K}} e^{(h-1)R}. \label{eq:y9}
	\end{equation}
	To control the third term of the decomposition in (\ref{eq:y4}) we write
	\begin{gather}
	\mathbf{m}\left( \bigcup_{n=n_2}^{\lceil R/r \rceil} \text{Nbhd}_{2r}(\pi(a_{nr} (V \cap K_\epsilon(nr)) \cap \text{Mod}_g \cdot \pi^{-1}(\text{Nbhd}_{r}(\mathcal{K}))))\right) \label{eq:y10}\\
	\leq \sum_{n=n_2}^{\lceil R/r \rceil} \mathbf{m} \left(\text{Nbhd}_{2r}(\pi(a_{nr} (V \cap K_\epsilon(nr)) \cap \text{Mod}_g \cdot \pi^{-1}(\text{Nbhd}_{r}(\mathcal{K}))))\right). \nonumber
	\end{gather}
	As $0 < r < 1$, Lemma \ref{lem:thin_annular_sectors} ensures
	\begin{gather}
	\sum_{n=n_2}^{\lceil R/r \rceil} \mathbf{m} \left(\text{Nbhd}_{2r}(\pi(a_{nr} (V\cap K_\epsilon(nr)) \cap \text{Mod}_g \cdot \pi^{-1}(\text{Nbhd}_{r}(\mathcal{K}))))\right) \label{eq:y11} \\
	\preceq_\mathcal{K} \sum_{n=n_2}^{\lceil R/r \rceil}  \overline{\nu}(\Re(V(\useconcc{cc:tas}e^{-\useconc{c:tas}nr}))) \cdot e^{hnr}. \nonumber
	\end{gather}
	Using the definition of $n_2$ we bound
	\begin{align}
	\sum_{n=n_2}^{\lceil R/r \rceil}  \overline{\nu}(\Re(V(\useconcc{cc:tas}e^{-\useconc{c:tas}nr}))) \cdot e^{hnr} &\leq \overline{\nu}(\Re(V(\useconcc{cc:tas}e^{-\useconc{c:tas}h^\dagger R}))) \cdot\sum_{n=n_2}^{\lceil R/r \rceil}  e^{hnr} \label{eq:y12}\\
	&\preceq_g \overline{\nu}(\Re(V(\useconcc{cc:tas}e^{-\useconc{c:tas}h^\dagger R}))) \cdot e^{hR}. \nonumber
	\end{align}
	Putting together (\ref{eq:y10}), (\ref{eq:y11}), and (\ref{eq:y12}) we deduce
	\begin{equation}
	\mathbf{m}\left( \bigcup_{n=n_2}^{\lceil R/r \rceil} \text{Nbhd}_{2r}(\pi(a_{nr} (V \cap K_\epsilon(nr)) \cap \text{Mod}_g \cdot \pi^{-1}(\text{Nbhd}_{r}(\mathcal{K}))))\right) \preceq_{\mathcal{K}} \overline{\nu}(\Re(V(\useconcc{cc:tas}e^{-\useconc{c:tas}h^\dagger R}))) \cdot e^{hR}. \label{eq:y13}
	\end{equation}
	Let $\useconn{n:thin_sect_gen} = \useconn{n:thin_sect_gen}(g) := \useconcc{cc:tas} > 0$ and $\useconk{k:thin_sect_gen} = \useconk{k:thin_sect_gen}(g)  := \min\{\useconk{k:thint}, 1, \useconc{c:tas}h^\dagger\} > 0$. Putting together (\ref{eq:y1}), (\ref{eq:y2}), (\ref{eq:y3}), (\ref{eq:y4}), (\ref{eq:y5}), (\ref{eq:y9}), and (\ref{eq:y13}) we conclude
	\[
	\mathbf{m}(\text{Nbhd}_{r}(B_R(X) \cap \text{Sect}_{V}(X) \cap \text{Mod}_g \cdot \mathcal{K})) 
	\preceq_\mathcal{K} \overline{\nu}(\Re(V( \useconn{n:thin_sect_gen} e^{-\useconk{k:thin_sect_gen}  R}))) \cdot e^{hR} + e^{(h-\useconk{k:thin_sect_gen})R}. \qedhere
	\]
\end{proof}

\begin{figure}[h!]
	\centering
	\includegraphics[width=.25\textwidth]{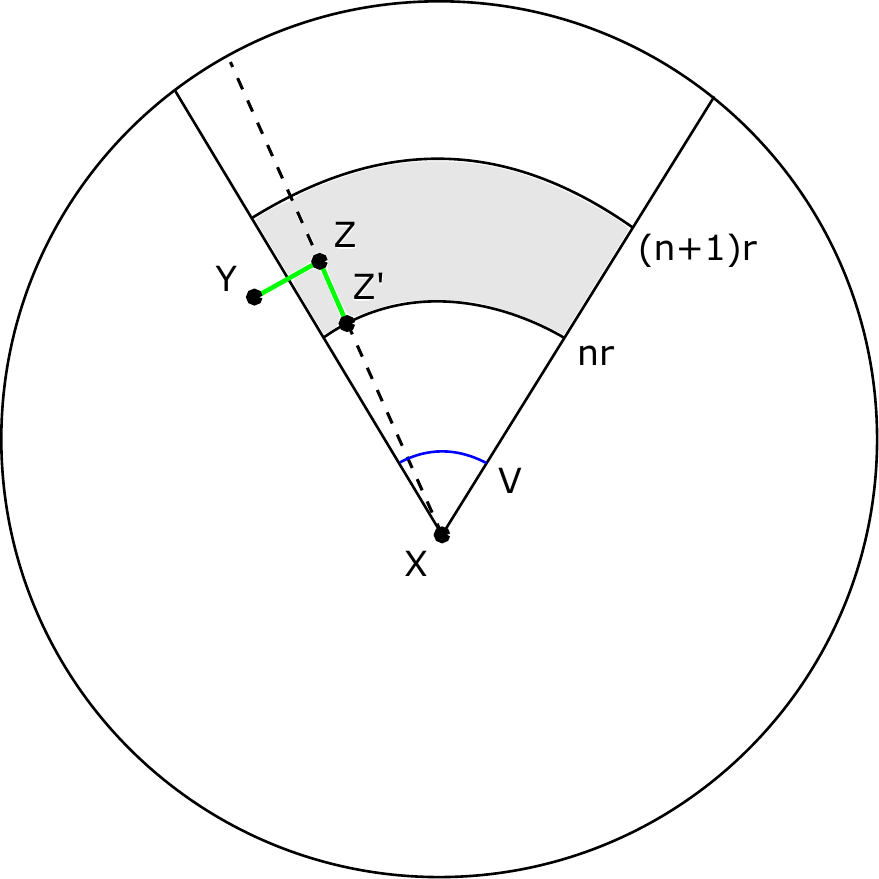}
	\caption{Picture of (\ref{eq:y3}). Straight lines represent Teichmüller geodesics. The green segment represents a piecewise smooth path of Teichmüller length $< 2r$ joining a point $Y \in \text{Nbhd}_{r}(B_R(X) \cap \text{Sect}_{V}(X) \cap \text{Mod}_g \cdot \mathcal{K}) \backslash \text{Nbhd}_{2r}(B_R(X,\mathcal{K},K_\epsilon))$ to a point $Z' \in \pi(a_{nr}(V \cap K_\epsilon(nr)) \cap \text{Mod}_g \cdot \pi^{-1}(\text{Nbhd}_{r}(\mathcal{K})))$.} \label{fig:annuli} 
\end{figure}

\subsection*{The function $\boldsymbol{\ell_{\min}(q)}$ and the Euclidean metric.} As an important step towards proving Theorem \ref{theo:thin_sector_sum}, we study how the function $\ell_{\min}(q)$ varies along short paths of $\qutp$ with respect to the Euclidean metric. Given a piecewise smooth path $\rho \colon [0,1] \to \qut$, denote by $\ell_E(\rho) = \int_0^1 \| \rho'(t) \|_E \thinspace dt$ its length with respect to the Euclidean metric.

\newcond{d:sys_euc} \newconcc{cc:sys_euc}
\begin{proposition}
	\label{prop:systole_euclidean}
	Let $\mathcal{K} \subseteq \mathcal{T}_g$ be a compact subset. There exist constants $\usecond{d:sys_euc} = \usecond{d:sys_euc}(\mathcal{K}) > 0$ and $\useconcc{cc:sys_euc} = \useconcc{cc:sys_euc}(\mathcal{K}) > 0$ such that for every piecewise smooth path $\rho \colon [0,1] \to \qutp$ with $\ell_E(\rho) < \usecond{d:sys_euc}$ and $\rho(0) \in \pi^{-1}(\mathcal{K})$, 
	\[
	\ell_{\min}(\rho(1)) \leq \ell_{\min}(\rho(0)) +  \useconcc{cc:sys_euc} \cdot \ell_E(\rho).
	\]
\end{proposition}

\begin{proof}
	Fix $\mathcal{K} \subseteq \tt$ compact. By Lemma \ref{lem:finite_delaunay}, the Delaunay triangulations of marked quadratic differentials in $\pi^{-1}(\mathrm{Nhbd}_1(\mathcal{K})) \cap \qutp$ induce finitely many marked triangulations on $S_g$. Each one of these triangulations gives rise to finitely many period coordinate charts $\Phi \colon U \to \mathbf{C}^{6g-6}$ defined on open subsets $U \subseteq \qutp$, one per homologically independent set of edges of the triangulation on the corresponding canonical double cover. By Lemma \ref{lem:euc_fund_prop}, for any such period coordinate chart, any $q \in U \cap \pi^{-1}(\mathrm{Nbhd}_1(\mathcal{K}))$, and any $v \in T_q \qut$, 
	\begin{equation}
	\label{eq:euc_comparison}
	\| v \|_E \asymp_{K} \|d\Phi_q v\|.
	\end{equation}
	
	 Let $\usecond{euc_teich} = \usecond{euc_teich}(\mathcal{K}) > 0$ be as in Corollary \ref{cor:euclidean_teichmuller_1}. Consider a piecewise smooth path $\rho \colon [0,1] \to \qutp$ with $\ell_E(\rho) < \usecond{euc_teich}$ and $\rho(0) \in \pi^{-1}(\mathcal{K})$. By Corollary \ref{cor:euclidean_teichmuller_1}, $\rho([0,1]) \subseteq \pi^{-1}(\mathrm{Nbhd}_1(\mathcal{K}))$. Let $t_0 \in [0,1]$ be arbitrary. There exists a relatively open interval $I \subseteq [0,1]$ containing $t_0$ and a finite collection $\gamma_1,\dots,\gamma_n$ of saddle connections of $\rho(t_0)$ such that $\ell_{\min}(\rho(t)) = \min_{i=1,\dots,n} \ell_{\gamma_i}(\rho(t))$ for every $t \in I$ and $\ell_{\gamma_i}(\rho(t)) \leq \sqrt{2} \cdot \ell_{\min}(\rho(t_0))$ for every $i \in \{1,\dots,n\}$ and every $t \in I$. In particular, by Lemma \ref{lem:delaunay}, the saddle connections $\gamma_1,\dots,\gamma_n$ are edges of every Delaunay triangulation of $\rho(t)$ for every $t \in I$. 
	
	Fix $i \in \{1,\dots,n\}$. Consider the canonical double cover induced by $\rho(t_0)$ on its underlying Riemann surface and denote by $H_1^\odd(\widetilde{X},\widetilde{\Sigma};\mathbf{Z})$ the odd part of the corresponding homology group. Using $\gamma_i$ and the Delaunay triangulation of $\rho(t_0)$ to choose an appropriate basis of $H_1^\odd(\widetilde{X},\widetilde{\Sigma};\mathbf{Z})$, one obtains a period coordinate chart $\Phi \colon U \to \mathbf{C}^{6g-6}$ among the finitely many described above such that
	\[
	\frac{d}{dt} \bigg \vert_{t = t_0}  \ell_{\gamma_i}(\rho(t)) \preceq \| d\Phi_{\rho(t_0)} \rho'(t_0) \|.
	\]
	By (\ref{eq:euc_comparison}),
	\[
	\| d\Phi_{\rho(t_0)} \rho'(t_0) \| \preceq_\mathcal{K} \| \rho'(t_0) \|_E.
	\]
	A direct application of the fundamental theorem of calculus finishes the proof.
\end{proof}

Recall that, for every $\epsilon > 0$, $K_\epsilon(\mathbf{1})$ denote the intersection of $K_\epsilon \subseteq \qum$ with the principal stratum $\qum(\mathbf{1})$. Directly from Proposition \ref{prop:systole_euclidean}, Theorem \ref{theo:euclidean_hodge_1}, and Corollary \ref{cor:euclidean_hodge_teichmuller_1}, we deduce the following result.

\newcond{d:thick} \newconcc{cc:thick}
\begin{corollary}
	\label{cor:thickening}
	Let $\mathcal{K} \subseteq \mathcal{T}_g$ be a compact subset. There exist constants $\usecond{d:thick} := \usecond{d:thick}(\mathcal{K}) > 0$ and $\useconcc{cc:thick} := \useconcc{cc:thick}(\mathcal{K}) > 0$ such that for every $X \in \mathcal{K}$, every $\epsilon > 0$, and every $0 < \delta < \usecond{d:thick}$, if $V := S(X) \backslash p^{-1}(K_\epsilon(\mathbf{1}))$, then
	\[
	V(\delta) \subseteq \pi^{-1}(B_1(X)) \backslash p^{-1}(K_{\epsilon+\useconcc{cc:thick} \delta}(\mathbf{1})).
	\]
\end{corollary}

\subsection*{Bounding the measure of thin sectors.} The final ingredient needed to prove Theorem \ref{theo:thin_sector_sum} is the following bound, which follows by the same arguments introduced in the proof of Proposition \ref{prop:small_HM_integral_sum}.

\begin{proposition}
	\label{prop:small_thu_measure}
	Let $\mathcal{K} \subseteq \mathcal{T}_g$ be a compact subset. Then, for every $\delta > 0$, 
	\[
	\overline{\nu}(\Re(\pi^{-1}(\mathcal{K}) \backslash p^{-1}(K_\delta(\mathbf{1})))) \preceq_\mathcal{K} \delta.
	\]
\end{proposition}

Theorem \ref{theo:thin_sector_general}, Corollary \ref{cor:thickening}, and Proposition \ref{prop:small_thu_measure} yield the following strong version of Theorem \ref{theo:thin_sector_sum}.

\newconk{k:thin_sect_strong} \newconr{r:thin_sect_strong}
\begin{theorem}
	\label{theo:thin_sector_strong}
	There exists a constant $\useconr{r:thin_sect_strong} =  \useconr{r:thin_sect_strong}(g)  >0$ such that for every compact subset $\mathcal{K} \subseteq \mathcal{T}_g$, every $X \in \mathcal{K}$, every $\delta > 0$, every $0 < r < \useconr{r:thin_sect_strong}$, and every $R > 0$, if $V := p^{-1}(K_\delta(\mathbf{1})) \cap S(X)$, then
	\[
	\mathbf{m}(\mathrm{Nbhd}_{r}(B_R(X) \cap \mathrm{Sect}_{V}(X) \cap \mathrm{Mod}_g \cdot \mathcal{K}))
	\preceq_{\mathcal{K}} \delta \cdot e^{hR} + e^{(h - \useconk{k:thin_sect_strong}) R},
	\]
	where $\useconk{k:thin_sect_strong} := \useconk{k:thin_sect_strong}(g) >0$ is a constant depending only on $g$.
\end{theorem}

\subsection*{Large deviations of the Teichmüller geodesic flow.} We finish this section by proving Theorem \ref{theo:large_deviations_sum}. Consider the function $u \colon \qum(\mathbf{1}) \to \mathbf{R}_{>0}$ given by $u(q) := \max\{1,\ell_{\min}(q)^{-1}\}$. For every $q \in \qum(\mathbf{1})$ and every $\epsilon > 0$ denote 
\[
J_{q,\epsilon}(t) := \{\theta \in [0,2\pi] \ | \ r_\theta q \in \qum \backslash K_\epsilon(t)\}.
\]
The following theorem, corresponding to \cite[Theorem 4.4]{EMM15}, is a consequence of \cite[Theorems 2.1, 2.2, and 2.3]{A06}. The use of the function $u$ in the following statement is justified in \cite[Lemma 6.2]{AG13}.

\newconk{k:ath1} \newconk{k:ath2} \newconn{n:ath} \newcone{e:ath}
\begin{theorem}
	\label{theo:EMM_Athreya}
	There exist constants $\usecone{e:ath} = \usecone{e:ath}(g) > 0$, $\useconn{n:ath} = \useconn{n:ath}(g) > 0$, and $\useconk{k:ath1} = \useconk{k:ath1}(g) > 0$ such that for any $0 < \epsilon < \usecone{e:ath}$ and any $t > 0$, if $q \in \qum(\mathbf{1})$ satisfies $\log u(q) \leq \useconn{n:ath} + \useconk{k:ath1} t$, then 
	\[
	|J_{q,\epsilon}(t)| \leq e^{-\useconk{k:ath1}t}.
	\]
\end{theorem}

Recall that $\widehat{\mu}$ denotes the Masur-Veech measure on $\qum$. To prove Theorem \ref{theo:large_deviations_sum} we will use Theorem \ref{theo:EMM_Athreya} and the following result due to Eskin and Masur \cite[Lemma 5.5]{EM93}.

\begin{lemma}
	\label{lem:systole_integrable}
	The function $u \colon \qum(\mathbf{1}) \to \mathbf{R}_{>0}$ belongs to $L^1(\qum(\mathbf{1}),\widehat{\mu})$.
\end{lemma}

We are now ready to prove Theorem \ref{theo:large_deviations_sum}.

\begin{proof}[Proof of Theorem \ref{theo:large_deviations_sum}]
	Let $\usecone{e:ath} = \usecone{e:ath}(g) > 0$, $\useconn{n:ath} = \useconn{n:ath}(g) > 0$, and $\useconk{k:ath1} = \useconk{k:ath1}(g) > 0$ be as in Theorem \ref{theo:EMM_Athreya}. Fix $0 < \epsilon < \usecone{e:ath}$ and $t > 0$. Notice that, for any $q \in \qum(\mathbf{1})$, the condition $\log u(q) \leq \useconn{n:ath} + \useconk{k:ath1} t$ in Theorem \ref{theo:EMM_Athreya} is equivalent to the condition $u(q) \leq e^{\useconn{n:ath} + \useconk{k:ath1}t}$. Consider the subsets 
	\begin{align*}
	A &:= \{q \in \qum(\mathbf{1}) \ | \ u(q) \leq  e^{\useconn{n:ath} + \useconk{k:ath1}t} \},\\
	B &:= \{q \in \qum(\mathbf{1}) \ | \ u(q) > e^{\useconn{n:ath} + \useconk{k:ath1}t} \}.
	\end{align*}
	Denote $K_\epsilon^c(t) := \qum(\mathbf{1}) \backslash K_\epsilon(t)$. Using the fact that $\widehat{\mu}$ is $\text{SO}(2)$-invariant and Fubini's theorem we write
	\begin{align}
	\widehat{\mu}(K_\epsilon^c(t)) &= \int_{\mathcal{Q}^1\mathcal{M}_g} \mathbbm{1}_{K_\epsilon^c(t)}(q) \thinspace d\widehat{\mu}(q) 
	= \frac{1}{2\pi} \int_0^{2\pi} \int_{\mathcal{Q}^1\mathcal{M}_g} \mathbbm{1}_{K_\epsilon^c(t)}(r_\theta q) \thinspace d\widehat{\mu}(q) \thinspace d\theta \label{eq:a1}\\
	&= \int_{\mathcal{Q}^1\mathcal{M}_g} \frac{1}{2\pi} \int_0^{2\pi} \mathbbm{1}_{K_\epsilon^c(t)}(r_\theta q) \thinspace d\theta \thinspace d\widehat{\mu}(q). \nonumber
	\end{align}
	We split this integral into two pieces,
	\begin{align}
	&\int_{\mathcal{Q}^1\mathcal{M}_g} \frac{1}{2\pi} \int_0^{2\pi} \mathbbm{1}_{K_\epsilon^c(t)}(r_\theta q) \thinspace d\theta \thinspace d\widehat{\mu}(q) \label{eq:a2} \\
	&= \int_{A} \frac{1}{2\pi} \int_0^{2\pi} \mathbbm{1}_{K_\epsilon^c(t)}(r_\theta q) \thinspace d\theta \thinspace d\widehat{\mu}(q) + \int_{B} \frac{1}{2\pi} \int_0^{2\pi} \mathbbm{1}_{K_\epsilon^c(t)}(r_\theta q) \thinspace d\theta \thinspace d\widehat{\mu}(q). \nonumber
	\end{align}
	Lemma \ref{lem:systole_integrable} ensures $u \in L^1(\qum(\mathbf{1}),\widehat{\mu})$. In particular, by Markov's inequality,
	\begin{equation}
	\int_{B} \frac{1}{2\pi} \int_0^{2\pi} \mathbbm{1}_{K_\epsilon^c(t)}(r_\theta q) \thinspace d\theta \thinspace d\widehat{\mu}(q) \leq \mu(B) \leq \|u\|_{L^1(\widehat{\mu})} \cdot e^{-\useconn{n:ath} - \useconk{k:ath1}t} \preceq_g e^{-\useconk{k:ath1}t}. \label{eq:a3}
	\end{equation}
	Directly from the definitions we see that
	\begin{equation}
	\int_{A} \frac{1}{2\pi} \int_0^{2\pi} \mathbbm{1}_{K_\epsilon^c(t)}(r_\theta q) \thinspace d\theta \thinspace d\widehat{\mu}(q) = \int_{A} \frac{1}{2\pi} \int_0^{2\pi} \mathbbm{1}_{J_{q,\epsilon}(t)}(\theta) \thinspace d\theta \thinspace d\widehat{\mu}(q). \label{eq:a4}
	\end{equation}
	By Theorem \ref{theo:EMM_Athreya},
	\begin{align}
	\int_{A} \frac{1}{2\pi} \int_0^{2\pi} \mathbbm{1}_{J_{q,\epsilon}(t)}(\theta) \thinspace d\theta \thinspace d\widehat{\mu}(q) = \int_{A} \frac{1}{2\pi} \thinspace |J_{q,\epsilon}(t)| \thinspace d\widehat{\mu}(q)
	\leq \int_{A} \frac{1}{2\pi} \thinspace e^{-\useconk{k:ath1}t} \thinspace d\widehat{\mu}(q)
	\preceq_g e^{-\useconk{k:ath1}t}. \label{eq:a5}
	\end{align}
	Putting together (\ref{eq:a1}),  (\ref{eq:a2}),  (\ref{eq:a3}),  (\ref{eq:a4}), and  (\ref{eq:a5}) we conclude
	\[
	\widehat{\mu}(\qum(\mathbf{1}) \backslash K_\epsilon(t)) \preceq_g e^{-\useconk{k:ath1}t}. \qedhere
	\]
\end{proof}

\section{Effective mean equidistribution of sectors in Teichmüller space}

\subsection*{Outline of this section.} In this section we state and prove effective mean equidistribution theorems for sectors of Teichmüller space. Following Thurston, we interpret $\pmf$ as the boundary at infinity of $\mathcal{T}_g$ and consider two cases depending on what type of objects on $\pmf$ are used to cut out the sectors. We first prove an effective mean equidistribution theorem for sectors of $\tt$ cut out by piecewise $\mathcal{C}^1$ functions on $\pmf$ and then, using an approximation argument, deduce an analogous theorem for sectors cut out by cubes of $\pmf$ in Dehn-Thurston coordinates. These results correspond to Theorems \ref{theo:sector_equidistribution} and \ref{theo:sector_equidistribution_box}, respectively. The key to proving these theorems is understanding how the smooth structure of $\qut$ and the piecewise linear structure of $\pmf$ interact through the map $[\Re] \colon \qut \to \pmf$.

\subsection*{Piecewise $\boldsymbol{\mathcal{C}^1}$ functions on $\boldsymbol{\pmf}$.} Recall that, although $\mf$ cannot be endowed with a natural smooth structure, it admits a natural piecewise integral linear structure induced by train track coordinates. As these coordinates conjugate the natural $\mathbf{R}_{>0}$ scaling actions, $\pmf$ inherits a natural piecewise projective integral linear structure. It makes sense then to consider the class $\PC^1(\pmf)$ of piecewise $\mathcal{C}^1$ functions $\psi \colon \pmf \to \mathbf{R}$. This will be the first class of objects we consider to cut out sectors of $\mathcal{T}_g$. We now introduce a norm to quantify their regularity.

Let $\tau$ be a maximal train track on $S_g$. Recall that $U(\tau) \subseteq \mf$ denotes the the closed cone of singular measured foliations carried by $\tau$ and that $V(\tau) \subseteq (\mathbf{R}_{\geq 0})^{18g-18}$ denotes the closed cone of non-negative counting measures on the edges of $\tau$ satisfying the switch conditions. Denote by $TU(\tau) \subseteq T\mf$ the subset of tangent vectors of $U(\tau) \subseteq \mf$ that point into $U(\tau)$. Pulling back the standard Riemannian Euclidean metric through the identification $U(\tau) = V(\tau)$ yields a fiberwise norm on $TU(\tau)$ which we denote by $\| \cdot \|_{\tau}$.

Identify the projectivization $PU(\tau) \subseteq \pmf$ with the projectivization $PV(\tau) \subseteq \mathrm{P}(\mathbf{R}_{\geq0})^{18g-18} $ and subsequently identify this projectivization with the affine submanifold $\VV(\tau) \subseteq V(\tau)$ of points in $V(\tau)$ of unit $L^1$-norm. Denote by $TPU(\tau) \subseteq TP\mf$ the subset of tangent vectors of $PU(\tau) \subseteq \pmf$ that point into $PU(\tau)$. Pulling back the standard Riemannian Euclidean metric through the identification $PU(\tau) = \VV(\tau)$ yields a fiberwise norm on $TPU(\tau)$ which we also denote by $\| \cdot \|_{\tau}$. 

For every $[\eta] \in PU(\tau)$ let $T_{[\eta]}PU(\tau) := T_{[\eta]} \pmf \cap PU(\tau)$. Given a piecewise $\mathcal{C}^1$ functions $\psi \colon \pmf \to \mathbf{R}_{\geq0}$ and $[\eta] \in PU(\tau)$, denote
\[
\|d\psi_{[\eta]}\|_{\tau} := \sup_{v \in T_{[\eta]}PU(\tau)} \frac{|d\psi_{[\eta]}v|}{\|v\|_{\tau}}.
\]

Let $\{\tau_i\}_{i=1}^k$ be a finite collection of maximal train tracks on $S_g$ such that $T\mf = \bigcup_{i=1}^k TU(\tau_i)$. On $\PC^1(\pmf)$ consider the norm
\begin{equation*}
%\label{eq:norm}
\|\psi\|_{\PC^1} := \|\psi\|_\infty + \sup_{i=1,\dots,k} \ \sup_{[\eta] \in PU(\tau_i)} \|d\psi_{[\eta]}\|_{\tau_i}.
\end{equation*}
As the transition maps between train track coordinate charts are piecewise linear, this definition specifies a Lipschitz class of norms on $\PC^1(\pmf)$. 

\subsection*{Effective mean equidistribution of sectors of $\boldsymbol{\tt}$ cut out by functions in $\boldsymbol{\mathcal{PC}^1(\pmf)}$.}  Before stating the first main theorem of this section we introduce some notation. Recall that $\mathbf{m} := \pi_* \mu$ denotes the pushforward to $\tt$ of the Masur-Veech measure $\mu$ on $\qut$ under the projection $\pi \colon \qut \to \tt$. Recall that $B_R(X) \subseteq \tt$ denotes the ball of radius $R > 0$ centered at $X \in \tt$ with respect to the Teichmüller metric. Recall that $q_s \colon \mathcal{T}_g \times \mathcal{T}_g \to \mathcal{Q}^1\mathcal{T}_g$ denotes the map which to every pair $X,Y \in \mathcal{T}_g$ assigns the quadratic differential $q_s(X,Y) \in S(X)$ corresponding to the tangent direction at $X$ of the unique Teichmüller geodesic segment from $X$ to $Y$. Fix $X \in \tt$ and $V \subseteq S(X)$ measurable. Recall that
\[
\mathrm{Sect}_V(X) := \{Y \in \tt \ | \ q_s(X,Y) \in V \}.
\]
For every $R > 0$, denote by $\mathbf{m}_{X,V}^R$ the restriction of the measure $\mathbf{m}$ to the set $B_R(X) \cap \mathrm{Sect}_V(X) \subseteq \tt$. Recall that $\Phi_X \colon S(X) \times \mathbf{R}_{>0} \to \mathcal{T}_g$ denotes the map $\Phi_X(q,t) := \pi(a_tq)$ and that the volume form $\mathbf{m}$ on $\mathcal{T}_g$ can be described in polar coordinates as
\begin{equation*}
%\label{eq:polar_new}
|\Phi_X^*(\mathbf{m})(q,t)| = \Delta(q,t) \cdot |s_X(q) \wedge dt|,
\end{equation*}
where $\Delta \colon S(X) \cap \mathcal{Q}^1\mathcal{T}_g(\mathbf{1}) \times \mathbf{R}_{>0} \to \mathbf{R}_{>0}$ is a positive, smooth function. In terms of this description,
\[
\mathbf{m}_{X,V}^R = (\Phi_X)_*(\Delta(q,t) \cdot\mathbbm{1}_{V}(q) \cdot \mathbbm{1}_{(0,R)}(t) \cdot |s_X(q) \wedge dt|).
\]
More generally, given an arbitrary non-negative, measurable function $\varphi \colon \mathcal{Q}^1\mathcal{T}_g \to \mathbf{R}_{\geq 0}$ and $R> 0$, consider the measure $\mathbf{m}_{X,\varphi}^R$ on $\tt$ given by
\[
\mathbf{m}_{X,\varphi}^R := (\Phi_X)_*(\Delta(q,t) \cdot \varphi(q) \cdot \mathbbm{1}_{(0,R)}(t) \cdot |s_X(q) \wedge dt|).
\]
Denote by $\widehat{\mathbf{m}}_{X,\varphi}^R$ the pushforward of $\mathbf{m}_{X,\varphi}^R$ to $\qum$. The measures $\widehat{\mathbf{m}}_{X,\varphi}^R$ keep track of how the sector centered at $X$ and cut out by $\varphi$ wraps around $\mm$. 

We will be particularly interested in the case $\varphi = \psi \circ [\Re]$ for $\psi \colon \pmf \to \mathbf{R}_{\geq 0}$ a measurable function in a suitable class and $[\Re] \colon \qut \to \mf$ the map $[\Re](q) := [\Re(q)]$. Given $X \in \mathcal{T}_g$, $\psi \colon \pmf \to \mathbf{R}_{\geq 0}$ measurable, and $R > 0$, denote $\mathbf{m}_{X,\psi}^R := \mathbf{m}_{X,\psi \circ [\Re]}^R$ and $\widehat{\mathbf{m}}_{X,\psi}^R := \widehat{\mathbf{m}}_{X,\psi \circ [\Re]}^R$.  

For the first half of this section we focus our attention on functions $\psi \in \mathcal{PC}^1(\pmf)$. Recall that $\up \colon \tt \to \mm$ denotes the quotient map and that $\widehat{\mathbf{m}}$ denotes the local pushforward to $\mm$ of the measure $\mathbf{m}$ on $\tt$. Recall that $h := 6g-6$ and that $\Lambda_g > 0$ denotes the Hubbard-Masur constant introduced in \S 2. The following result generalizes Theorem \ref{theo:ball_equidistribution} to sectors of $\tt$ cut out by functions in $\mathcal{PC}^1(\pmf)$.

\newconk{k:sect_equid}
\begin{theorem}
	\label{theo:sector_equidistribution}
	Let $\mathcal{K} \subseteq \mathcal{T}_g$ be a compact subset, $\phi_1 \in L^\infty(\tt,\mathbf{m})$ be an essentially bounded function with $\esupp(\phi_1) \subseteq \mathcal{K}$, and $\psi_1  \in \mathcal{PC}^1(\pmf)$ non-negative. Then, for every essentially bounded function $\phi_2 \in L^\infty(\mm,\widehat{\m})$ with $\esupp(\phi_2) \subseteq \up(\mathcal{K})$ and every $R > 0$,
	\begin{gather*}
	 \int_{\tt} \phi_1(X) \left( \int_{\mm} \phi_2(Y) \thinspace d\widehat{\mathbf{m}}_{X,\psi_1 }^R(Y) \right) d\mathbf{m}(X) \\
	=  \frac{\Lambda_g}{h \cdot \widehat{\mathbf{m}}(\mathcal{M}_g)} \cdot \left(\int_{\qut} \psi_1([\Re(q)]) \thinspace \phi_1(\pi(q)) \thinspace \lambda(q) \thinspace d\mu(q)\right) \cdot \left(\int_{\mm} \phi_2(Y)  \thinspace d\widehat{\mathbf{m}}(Y) \right) \cdot e^{hR} \\
	+ O_\mathcal{K}\left( \|\psi_1 \|_{\mathcal{PC}^1} \cdot \|\phi_1\|_\infty \cdot \|\phi_2\|_\infty \cdot e^{(h-\useconk{k:sect_equid})R}\right),
	\end{gather*}
	where $\useconk{k:sect_equid} =\useconk{k:sect_equid}(g) > 0$ is a constant depending only on $g$.
\end{theorem}

\subsection*{A preliminary effective mean equidistribution theorem for sectors of $\boldsymbol{\tt}$.} To prove Theorem \ref{theo:sector_equidistribution} we first introduce a preliminary effective mean equidistribution theorem for sectors of $\tt$ cut out by observables $\varphi \colon \qut \to \mathbf{R}_{>0}$ in a suitable class and then show that functions in $\mathcal{PC}^1(\pmf)$ give rise to observables in such class when precomposed with the map $[\Re] \colon \qut \to \pmf$. 

Recall that $\mathbf{S}^1 \subseteq \mathbf{C}$ denotes the unit circle. Denote by $\mathcal{R}(\qut,\mu) \subseteq L^2(\qut,\mu)$ the subset of functions $\varphi \in L^2(\qut,\mu)$ such that the map $e^{i\theta} \in \mathbf{S}^1 \mapsto \varphi_\circ r_\theta \in L^2(\qut,\mu)$ is Lipschitz. Denote by $\|\varphi\|_{\mathrm{Lip}(\mu)}$ the minimal Lipschitz constant of such map. On $\mathcal{R}(\qut,\mu)$ consider the norm
\[
\|\varphi\|_{\mathcal{R}(\mu)} := \|\varphi\|_{L^2(\mu)} + \| \varphi \|_{\mathrm{Lip}(\mu)}.
\]
The class of local Ratner observables $\mathcal{R}_\mathrm{loc}(\qut,\mu) \subseteq L^2(\qut,\mu)$ is the set of functions $\varphi \in L^2(\qut,\mu)$ such that $\varphi \cdot \mathbbm{1}_{\pi^{-1}(\mathcal{K})} \in \mathcal{R}(\qut,\mu)$ for every compact subset $\mathcal{K} \subseteq \tt$. 

Carefully following the proof of Theorem \ref{theo:ball_equidistribution} and making suitable modifications yields the following result.

\newconk{k:sect_equid_gen}
\begin{theorem}
	\label{theo:sector_equidistribution_general}
	Let $\mathcal{K} \subseteq \mathcal{T}_g$ be a compact subset, $\phi_1 \in L^\infty(\tt,\mathbf{m})$ be an essentially bounded function with $\esupp(\phi_1) \subseteq \mathcal{K}$, and $\varphi_1  \in \mathcal{R}_{\mathrm{loc}}(\qut,\mu)$ non-negative. Then, for every essentially bounded function $\phi_2 \in L^\infty(\mm,\widehat{\m})$ with $\esupp(\phi_2) \subseteq \up(\mathcal{K})$ and every $R > 0$,
	\begin{gather*}
	\int_{\tt} \phi_1(X) \left( \int_{\mm} \phi_2(Y) \thinspace d\widehat{\mathbf{m}}_{X,\varphi_1}^R(Y) \right) d\mathbf{m}(X) \\
	=  \frac{\Lambda_g}{h \cdot \widehat{\mathbf{m}}(\mathcal{M}_g)} \cdot \left(\int_{\qut} \varphi_1(q) \thinspace \phi_1(\pi(q)) \thinspace \lambda(q) \thinspace d\mu(q)\right) \cdot \left(\int_{\mm} \phi_2(Y)  \thinspace d\widehat{\mathbf{m}}(Y) \right) \cdot e^{hR} \\
	+ O_\mathcal{K}\left( \|\varphi_1 \cdot \mathbbm{1}_{\pi^{-1}(\mathcal{K})} \|_{\mathcal{R}(\mu)} \cdot \|\phi_1\|_\infty \cdot \|\phi_2\|_\infty \cdot e^{(h-\useconk{k:sect_equid_gen})R}\right),
	\end{gather*}
	where $\useconk{k:sect_equid_gen} =\useconk{k:sect_equid_gen}(g) > 0$ is a constant depending only on $g$.
\end{theorem}

\subsection*{$\boldsymbol{\mathcal{R}_\mathrm{loc}(\qut,\mu)}$ observables from $\boldsymbol{\mathcal{PC}^1(\pmf)}$ functions.} Theorem \ref{theo:sector_equidistribution} will follow directly from Theorem \ref{theo:sector_equidistribution_general} once we show that functions in $\mathcal{PC}^1(\pmf)$ give rise to functions in $\mathcal{R}_\mathrm{loc}(\qut,\mu)$ when precomposed with the map $[\Re] \colon \qut \to \pmf$. The following proposition addresses this matter.

\begin{proposition}
	\label{prop:ratner_observable_new}
	Let $\psi \in \mathcal{PC}^1(\pmf)$. Then, $\psi \circ [\Re] \in \mathcal{R}_\mathrm{loc}(\qut,\mu)$ and, for every $\mathcal{K} \subseteq \tt$ compact,
	\[
	\|(\psi \circ [\Re]) \cdot \mathbbm{1}_{\pi^{-1}(\mathcal{K})}\|_{\mathcal{R}(\mu)} \preceq_{\mathcal{K}} \|\psi\|_{\mathcal{PC}^1}.
	\]
\end{proposition}

To prove Proposition \ref{prop:ratner_observable_new} we study the image through the map $[\Re] \colon \qut \to \pmf$ of the $\text{SO}(2)$-orbits of $\mathcal{Q}^1\mathcal{T}_g(\mathbf{1})$. The restriction $\Re|_{\qutp} \colon \qutp \to \mf$ is a piecewise $\mathcal{C}^1$ map \cite[Lemma 4.3]{Mir08a}. In particular, for every $q \in \qutp$, the path $\theta \in \mathbf{S}^1 \mapsto \Re(r_\theta q) \in \mf$ is piecewise $\mathcal{C}^1$. The following proposition bounds the norm of the tangent vectors of  paths $\theta \in \mathbf{S}^1 \mapsto [\Re(r_\theta q)] \in \pmf$ with $q \in \qutp$.

\begin{proposition}
	\label{prop:derivative bound}
	Let $\mathcal{K} \subseteq \mathcal{T}_g$ be a compact subset and $\mathcal{\tau}$ be a maximal train track on $S_g$. Suppose that $q \in \pi^{-1}(\mathcal{K}) \cap \qutp$ is such that 
	$
	(d/d\theta) \vert_{\theta =0^+} \thinspace \Re(r_\theta q) \in TU(\tau). 
	$
	Then,
	\[
	\bigg\| \frac{d}{d\theta} \bigg\vert_{\theta =0^+} [\Re(r_\theta q)]\bigg\|_{\tau} \preceq_{\mathcal{K},\tau} 1.
	\]	
\end{proposition}

\begin{proof}
	Fix $\mathcal{K} \subseteq \mathcal{T}_g$ compact and $\tau$ a maximal train track on $S_g$. Let $q \in \pi^{-1}(\mathcal{K}) \cap \qutp$ such that $(d/d\theta) \vert_{\theta =0^+} \thinspace \Re(r_\theta q) \in TU(\tau)$. Consider the finite collection of maximal train tracks $\{\tau_i\}_{i=1}^n$ on $S_g$ provided by Lemma \ref{lem:finite_tt}. The explicit construction of these train tracks ensures
	$(d/d\theta) \vert_{\theta =0^+} \thinspace \Re(r_\theta q) \in TU(\tau_i)$ for some $i \in \{1,\dots,n\}$. Moreover, as the counting measures on the edges of $\tau_i$ correspond to the absolute value of the real part of the holonomy of the edges the corresponding Delaunay triangulation, Lemma \ref{lem:delaunay} ensures, 
	\[
	\bigg\| \frac{d}{d\theta} \bigg\vert_{\theta =0^+}\thinspace \Re(r_\theta q) \bigg\|_{\tau_i} \preceq_{\mathcal{K}} 1.
	\]
	As the transition maps between train track coordinate charts are piecewise linear,
	\[
	\bigg\| \frac{d}{d\theta} \bigg\vert_{\theta =0^+} \thinspace \Re(r_\theta q) \bigg\|_{\tau} \preceq_{\mathcal{K}} 1.
	\]
	Recall that, for every $q \in \qut$, $\mathrm{Ext}_{\pi(q)}(\Re(q)) = 1$. This condition provides uniform lower and upper bounds depending only on $\mathcal{K}$ on the $L^1$-norm of points $\Re(q) \in U(\tau) = V(\tau)$ for $q \in\pi^{-1}(\mathcal{K})$. The map which scales points in $V(\tau)$ to $\overline{V}(\tau)$ has derivatives bounded uniformly in terms of $\mathcal{K}$ under these conditions. Projecting a vector on $\overline{V}(\tau)$ to $T\overline{V}(\tau)$ can only reduce its norm. The proposition follows.
\end{proof}

A direct application of the fundamental theorem of calculus and Proposition \ref{prop:derivative bound} over a finite collection $\{\tau_i\}_{i=1}^n$ of maximal train tracks on $S_g$ such that $T\mf = \bigcup_{i=1}^n TU(\tau_i)$ yields the following corollary. 

\begin{corollary}
	\label{cor:orbit_functions}
	Let $\mathcal{K} \subseteq \mathcal{T}_g$ compact and $\psi \in \mathcal{PC}^1(\pmf)$. Then, for every $q \in \pi^{-1}(\mathcal{K})\cap \qutp$, the function $\theta \in \mathbf{S}^1 \mapsto \psi([\Re(r_\theta q)]) \in \mathbf{R}$ is Lipschitz with Lipschitz constant $\preceq_\mathcal{K} \| \psi \|_{\mathcal{PC}^1}$.
\end{corollary}

We are now ready to prove Proposition \ref{prop:ratner_observable_new}. 

\begin{proof}[Proof of Proposition \ref{prop:ratner_observable_new}]
	Fix a compact subset $\mathcal{K} \subseteq \tt$ and a piecewise $\mathcal{C}^1$ function $\psi \colon \pmf \to \mathbf{R}$. We show the map
	$
	\theta \in \mathbf{S}^1 \mapsto ((\psi \circ [\Re]) \cdot \mathbbm{1}_{\pi^{-1}(\mathcal{K})}) \circ r_\theta = (\psi \circ [\Re] \circ r_\theta) \cdot \mathbbm{1}_{\pi^{-1}(\mathcal{K})} \in L^2(\qut,\mu)
	$
	is Lipschitz. As $\qutp \subseteq \qut$ is a full measure subset with respect to $\mu$, Corollary \ref{cor:orbit_functions} ensures that, for every $\theta_1,\theta_2 \in \mathbf{S}^1$, 
	\begin{align*}
	&\|(\psi \circ [\Re] \circ r_{\theta_1}) \cdot \mathbbm{1}_{\pi^{-1}(\mathcal{K})} - (\psi \circ [\Re] \circ r_{\theta_2}) \cdot \mathbbm{1}_{\pi^{-1}(\mathcal{K})}\|_{L^2(\mu)} \\
	&= \left(\int_{\pi^{-1}(\mathcal{K})} \left(\psi ([\Re(r_{\theta_1} q)]) - \psi([\Re(r_{\theta_2} q)])\right)^2 d\mu(q) \right)^{1/2} \\
	&\preceq_{\mathcal{K}} \| \psi \|_{\mathcal{PC}^1} \cdot  |\theta_1-\theta_2|.
	\end{align*}
	The proposition follows directly from this computation.
\end{proof}

Theorem \ref{theo:sector_equidistribution} now follows directly from Theorem \ref{theo:sector_equidistribution_general} and Proposition \ref{prop:ratner_observable_new}.

\subsection*{Dehn-Thurston coordinates.} We now consider sectors of $\mathcal{T}_g$ cut out by cubes of $\pmf$. Dehn-Thurston coordinates  parametrize $\mf$ in terms of the intersection and twisting numbers of singular measured foliations with respect to the components of a pair of pants decomposition of $S_g$; see \cite[\S2.6]{PH92} for details. Any set of Dehn-Thurston coordinates provides a continuous, piecewise linear identification of $\mf$ with $\mathbf{R}^{6g-6}$ which conjugates the natural $\mathbf{R}_{>0}$ scaling actions and maps the Thurston measure $\nu$ to a constant multiple of the Lebesgue measure. This identification induces a continuous, projectively piecewise linear identification of $\pmf$ with $\mathbf{S}^{6g-7} \subseteq \mathbf{R}^{6g-6}$, the set of points in $\mathbf{R}^{6g-6}$ with unit Euclidean norm. For the rest of this section, we fix a set of Dehn-Thurston coordinates, consider the corresponding identifications $\mf = \mathbf{R}^{6g-6}$ and $\pmf = \mathbf{S}^{6g-7}$, and endow $\mathbf{S}^{6g-7}$ with the restriction of the Riemannian Euclidean metric. When an implicit constant depends on the choice of Dehn-Thurston coodinates we add the subscript $\mathrm{DT}$. We consider cubes $\mathcal{B} \subseteq \pmf = \mathbf{S}^{6g-7}$ with closed and/or open facets.

\subsection*{Effective mean equidistribution of sectors of $\boldsymbol{\tt}$ cut out by cubes of $\boldsymbol{\pmf}$.} For every $X \in \tt$, every measurable subset $\mathcal{U} \subseteq \pmf$, and every $R > 0$, denote $\mathbf{m}_{X,\mathcal{U}}^R := \mathbf{m}_{X,\mathbbm{1}_{\mathcal{U}}}^R$ and $\widehat{\mathbf{m}}_{X,\mathcal{U}}^R := \widehat{\mathbf{m}}_{X,\mathbbm{1}_{\mathcal{U}}}^R$. Notice  $\mathbf{m}_{X,\mathcal{U}}^R$ is precisely the restriction of the measure $\mathbf{m}$ to the intersection of $B_R(X) \subseteq \tt$ with the sector
\[
\mathrm{Sect}_\mathcal{U}(X) := \{Y \in \tt \ | \ [\Re(q_s(X,Y))] \in \mathcal{U} \}.
\]
The following result generalizes Theorem \ref{theo:ball_equidistribution} to sectors of $\tt$ cut out by cubes of $\pmf$.

\newconk{k:sect_equid_box} 
\begin{theorem}
	\label{theo:sector_equidistribution_box}
	Let $\mathcal{K} \subseteq \mathcal{T}_g$ be a compact subset, $\phi_1 \in L^\infty(\tt,\mathbf{m})$ be an essentially bounded function with $\esupp(\phi_1) \subseteq \mathcal{K}$, and $\mathcal{B} \subseteq \pmf$ be a cube. Then, for every essentially bounded function $\phi_2 \in L^\infty(\mm,\widehat{\m})$ with $\esupp(\phi_2) \subseteq \up(\mathcal{K})$ and every $R > 0$,
	\begin{gather*}
	\int_{\tt} \phi_1(X) \left( \int_{\mm} \phi_2(Y) \thinspace d\widehat{\mathbf{m}}_{X,\mathcal{B}}^R(Y) \right) d\mathbf{m}(X) \\
	=  \frac{\Lambda_g}{h \cdot \widehat{\mathbf{m}}(\mathcal{M}_g)} \cdot \left(\int_{\qut} \mathbbm{1}_{\mathcal{B}}([\Re(q)]) \thinspace \phi_1(\pi(q)) \thinspace \lambda(q) \thinspace d\mu(q)\right) \cdot \left(\int_{\mm} \phi_2(Y)  \thinspace d\widehat{\mathbf{m}}(Y) \right) \cdot e^{hR} \\
	+ O_{\mathcal{K},\mathrm{DT}}\left( \|\phi_1\|_\infty \cdot \|\phi_2\|_\infty \cdot e^{(h-\useconk{k:sect_equid_box})R}\right),
	\end{gather*}
	where $\useconk{k:sect_equid_box} = \useconk{k:sect_equid_box}(g) > 0$ is a constant depending only on $g$.
\end{theorem}

To prove Theorem \ref{theo:sector_equidistribution_box} we use Theorem \ref{theo:sector_equidistribution} and an approximation argument. Denote by $\mathrm{vol}$ the standard volume form on $\pmf = \mathbf{S}^{6g-7}$. The following proposition is the approximation tool needed we will need.

\begin{proposition}
	\label{prop:approx}
	Let $\mathcal{B} \subseteq \pmf$ be a cube. For every $0 < \delta < 1$ there exists a pair of piecewise $\mathcal{C}^1$ functions $\psi_{\mathcal{B},\delta}^{\mathrm{in}}, \psi_{\mathcal{B},\delta}^{\mathrm{out}} \colon \pmf \to [0,1]$  with the following properties:
	\begin{enumerate}
		\item $\psi_{\mathcal{B},\delta}^{\mathrm{in}} \leq \mathbbm{1}_{\mathcal{B}} \leq \psi_{\mathcal{B},\delta}^{\mathrm{out}}$,
		\item $\|\psi_{\mathcal{B},\delta}^{\mathrm{in}}\|_{\mathcal{PC}^1}, \|\psi_{\mathcal{B},\delta}^{\mathrm{out}}\|_{\mathcal{PC}^1} \preceq_{\mathrm{DT}} \delta^{-h}$,
		\item 	$
		\int_{\pmf} (\psi_{\mathcal{B},\delta}^{\mathrm{out}} - \psi_{\mathcal{B},\delta}^{\mathrm{in}}) \thinspace d\mathrm{vol} \preceq_{\mathrm{DT}} \delta.
		$
	\end{enumerate}
\end{proposition}

We now prove Theorem \ref{theo:sector_equidistribution_box}.

\begin{proof}[Proof of Theorem \ref{theo:sector_equidistribution_box}]
	Fix $\mathcal{K} \subseteq \mathcal{T}_g$ compact, $\phi_1 \in L^\infty(\tt,\mathbf{m})$ with $\esupp(\phi_1) \subseteq \mathcal{K}$, $\phi_2 \in L^\infty(\mm,\widehat{\m})$ with $\esupp(\phi_2) \subseteq \up(\mathcal{K})$, and $\mathcal{B} \subseteq \pmf$ a cube. Let $R > 0$ and $0 < \delta = \delta(R) < 1$, to be fixed later. Consider the piecewise $\mathcal{C}^1$ approximations $\psi_{\mathcal{B},\delta}^{\mathrm{in}}, \psi_{\mathcal{B},\delta}^{\mathrm{out}} \colon \pmf \to [0,1]$ of the indicator function $\mathbbm{1}_{\mathcal{B}}$ provided by Proposition \ref{prop:approx}. Denote $\psi_1 := \psi_{\mathcal{B},\delta}^{\mathrm{in}}$ and $\psi_2 := \psi_{\mathcal{B},\delta}^{\mathrm{out}}$.  Notice that, by the first property of Proposition \ref{prop:approx},
	\begin{gather}
		\int_{\tt} \phi_1(X) \left( \int_{\mm} \phi_2(Y) \thinspace d\widehat{\mathbf{m}}_{X,\psi_1}^R(Y) \right) d\mathbf{m}(X) \leq 	\int_{\tt} \phi_1(X) \left( \int_{\mm} \phi_2(Y) \thinspace d\widehat{\mathbf{m}}_{X,\mathcal{B}}^R(Y) \right) d\mathbf{m}(X),   \label{eq:cube_1} \\
		\int_{\tt} \phi_1(X) \left( \int_{\mm} \phi_2(Y) \thinspace d\widehat{\mathbf{m}}_{X,\mathcal{B}}^R(Y) \right) d\mathbf{m}(X) \leq 	\int_{\tt} \phi_1(X) \left( \int_{\mm} \phi_2(Y) \thinspace d\widehat{\mathbf{m}}_{X,\psi_2}^R(Y) \right) d\mathbf{m}(X). \label{eq:cube_2}
	\end{gather}
	By Theorem \ref{theo:sector_equidistribution} and the second property of Proposition \ref{prop:approx}, 
	\begin{gather}
		 \int_{\tt} \phi_1(X) \left( \int_{\mm} \phi_2(Y) \thinspace d\widehat{\mathbf{m}}_{X,\psi_1 \circ [\Re]}^R(Y) \right) d\mathbf{m}(X) \label{eq:cube_3} \\
		=  \frac{\Lambda_g}{h \cdot \widehat{\mathbf{m}}(\mathcal{M}_g)} \cdot \left(\int_{\qut} \psi_1([\Re(q)]) \thinspace \phi_1(\pi(q)) \thinspace \lambda(q) \thinspace d\mu(q)\right) \cdot \left(\int_{\mm} \phi_2(Y)  \thinspace d\widehat{\mathbf{m}}(Y) \right) \cdot e^{hR}  \nonumber\\
		+ O_{\mathcal{K},\mathrm{DT}}\left(\|\phi_1\|_\infty \cdot  \|\phi_2\|_\infty \cdot \delta^{-h} \cdot e^{(h-\useconk{k:sect_equid})R}\right).  \nonumber
	\end{gather}
	Analogously,
	\begin{gather}
		\int_{\tt} \phi_1(X) \left( \int_{\mm} \phi_2(Y) \thinspace d\widehat{\mathbf{m}}_{X,\psi_2 \circ [\Re]}^R(Y) \right) d\mathbf{m}(X) \label{eq:cube_4} \\
		=  \frac{\Lambda_g}{h \cdot \widehat{\mathbf{m}}(\mathcal{M}_g)} \cdot \left(\int_{\qut} \psi_2([\Re(q)]) \thinspace \phi_1(\pi(q)) \thinspace \lambda(q) \thinspace d\mu(q)\right) \cdot \left(\int_{\mm} \phi_2(Y)  \thinspace d\widehat{\mathbf{m}}(Y) \right) \cdot e^{hR} \nonumber\\
		+ O_{\mathcal{K},\mathrm{DT}}\left(\|\phi_1\|_\infty \cdot \|\phi_2\|_\infty \cdot \delta^{-h} \cdot e^{(h-\useconk{k:sect_equid})R}\right). \nonumber
	\end{gather}
	Putting together (\ref{eq:cube_1}), (\ref{eq:cube_2}), (\ref{eq:cube_3}), and (\ref{eq:cube_4}) we deduce
	\begin{gather}
		\int_{\tt} \phi_1(X) \left( \int_{\mm} \phi_2(Y) \thinspace d\widehat{\mathbf{m}}_{X,\mathcal{B}}^R(Y) \right) d\mathbf{m}(X) \label{eq:cube_5}\\
		=  \frac{\Lambda_g}{h \cdot \widehat{\mathbf{m}}(\mathcal{M}_g)} \cdot \left(\int_{\qut} \mathbbm{1}_{\mathcal{B}}([\Re(q)]) \thinspace \phi_1(\pi(q)) \thinspace \lambda(q) \thinspace d\mu(q)\right) \cdot \left(\int_{\mm} \phi_2(Y)  \thinspace d\widehat{\mathbf{m}}(Y) \right) \cdot e^{hR} \nonumber \\
		+ O_g\left(\|\phi_1\|_\infty \cdot \|\phi_2\|_{\infty} \cdot \left(\int_{\pi^{-1}(\mathcal{K})} (\psi_2([\Re(q)] - \psi_1([\Re(q)]) \thinspace \lambda(q) \thinspace d\mu(q)\right)  \cdot e^{hR}\right) \nonumber \\
		+ O_{\mathcal{K},DT}\left(\|\phi_1\|_\infty \cdot \|\phi_2\|_\infty \cdot \delta^{-h} \cdot e^{(h-\useconk{k:sect_equid})R}\right). \nonumber
	\end{gather}
	Using (\ref{eq:fiberwise_measures}) we write
	\begin{gather}
	\int_{\pi^{-1}(\mathcal{K})} (\psi_2([\Re(q)]) - \psi_1([\Re(q)])) \thinspace \lambda(q) \thinspace d\mu(q) \label{eq:cube_6}\\
	 = \int_{\mathcal{K}} \left( \int_{S(X)} (\psi_2([\Re(q)]) - \psi_1([\Re(q)])) \thinspace \lambda(q) \thinspace ds_X(q) \right) d\mathbf{m}(X). \nonumber
	\end{gather}
	Let $X \in \mathcal{K}$. Recall $E_X := \{\eta \in \mf \ | \ \mathrm{Ext}_\eta(X) = 1\}$. Identify $\pmf$ with $E_X$ and consider the measure $\overline{\nu}_X$ on $E_X$ introduced in \S 2 as a measure on $\pmf$. By definition (\ref{eq:hm_3}) of the Hubbard-Masur function $\lambda$,
	\begin{equation}
	\label{eq:cube_7}
	\int_{S(X)} (\psi_2([\Re(q)]) - \psi_1([\Re(q)])) \thinspace \lambda(q) \thinspace ds_X(q) = \int_{\pmf} (\psi_2([\eta]) - \psi_1([\eta])) \thinspace d\overline{\nu}_X([\eta]).
	\end{equation}
	Recall that the Thurston measure $\nu$ on $\mf = \mathbf{R}^{6g-6}$ is equal, up to a multiplicative constant, to the Lebesgue measure in Dehn-Thurston coordinates. As $\pmf$ is compact and as the function $\mathrm{Ext} \colon \mf \times \tt \to \mathbf{R}_{>0}$ given by $\mathrm{Ext}(\eta,X) := \mathrm{Ext}_\eta(X)$ is positive and continuous, the measure $\overline{\nu}_X$ on $\pmf = \mathbf{S}^{6g-7}$ is bounded by $\mathrm{vol}$ up to a multiplicative constant depending only on $\mathcal{K}$ and the choice of Dehn-Thurston coordinates. This fact together with the third property of Proposition \ref{prop:approx} implies
	\begin{equation}
	\label{eq:cube_8}
	\int_{\pmf} \psi_2([\eta]) - \psi_1([\eta]) \thinspace d\overline{\nu}_X([\eta]) \preceq_{\mathcal{K},\mathrm{DT}} \int_{\pmf} \psi_2([\eta]) - \psi_1([\eta]) \thinspace d\mathrm{vol}([\eta]) \preceq_{\mathrm{DT}} \delta.
	\end{equation}
	Putting together (\ref{eq:cube_6}), (\ref{eq:cube_7}), and (\ref{eq:cube_8}) we deduce
	\begin{equation}
	\label{eq:cube_9}
	\int_{\pi^{-1}(\mathcal{K})} (\psi_2([\Re(q)]) - \psi_1([\Re(q)])) \thinspace \lambda(q) \thinspace d\mu(q) \preceq_{\mathcal{K},\mathrm{DT}} \delta.
	\end{equation}
	From (\ref{eq:cube_5}) and (\ref{eq:cube_9}) we deduce
	\begin{gather}
		\int_{\tt} \phi_1(X) \left( \int_{\mm} \phi_2(Y) \thinspace d\widehat{\mathbf{m}}_{X,\mathcal{B}}^R(Y) \right) d\mathbf{m}(X) \label{eq:ending}\\
		=  \frac{\Lambda_g}{h \cdot \widehat{\mathbf{m}}(\mathcal{M}_g)} \cdot \left(\int_{\qut} \mathbbm{1}_{\mathcal{B}}([\Re(q)]) \thinspace \phi_1(\pi(q)) \thinspace \lambda(q) \thinspace d\mu(q)\right) \cdot \left(\int_{\mm} \phi_2(Y)  \thinspace d\widehat{\mathbf{m}}(Y) \right) \cdot e^{hR} \nonumber \\
		+ O_{\mathcal{K},DT}\left(\|\phi_1\|_\infty \cdot \|\phi_2\|_\infty \cdot \left(\delta \cdot e^{hR} + \delta^{-h} \cdot e^{(h-\useconk{k:sect_equid})R}\right)\right). \nonumber
	\end{gather}
	Let $\delta = \delta(R) := e^{-\eta R}$ with $\eta = \eta(g) >0$ small enough so that $h\eta < \useconk{k:sect_equid}$. Denote $\useconk{k:sect_equid_box} = \useconk{k:sect_equid_box}(g) := \min\{\eta, \useconk{k:sect_equid}- h\eta\} > 0$. From (\ref{eq:ending}) we conclude that, for every $R > 0$,
	\begin{gather*}
	\int_{\tt} \phi_1(X) \left( \int_{\mm} \phi_2(Y) \thinspace d\widehat{\mathbf{m}}_{X,\mathcal{B}}^R(Y) \right) d\mathbf{m}(X) \\
	=  \frac{\Lambda_g}{h \cdot \widehat{\mathbf{m}}(\mathcal{M}_g)} \cdot \left(\int_{\qut} \mathbbm{1}_{\mathcal{B}}([\Re(q)]) \thinspace \phi_1(\pi(q)) \thinspace \lambda(q) \thinspace d\mu(q)\right) \cdot \left(\int_{\mm} \phi_2(Y)  \thinspace d\widehat{\mathbf{m}}(Y) \right) \cdot e^{hR} \\
	+ O_{\mathcal{K},\mathrm{DT}}\left( \|\phi_1\|_\infty \cdot \|\phi_2\|_\infty \cdot e^{(h-\useconk{k:sect_equid_box})R}\right). \qedhere
	\end{gather*}
\end{proof}

Let $\mathcal{D}_g \subseteq \mathcal{T}_g$ be a measurable fundamental domain for the $\mcg$ action on $\tt$. Denote by $s \colon \mathcal{D}_g \to \mathcal{M}_g$ the restriction to $\mathcal{D}_g$ of the quotient map $\underline{\smash{p}} \colon \tt \to \mm$. The following result is an immediate consequence of Theorem \ref{theo:sector_equidistribution_box} and the proper discontinuity of the $\mcg$ action on $\tt$; compare to Theorem \ref{theo:ball_equidistribution_extra}.

\newconk{k:sebe} 
\begin{theorem}
	\label{theo:sector_equidistribution_box_extra}
	Let $\mathcal{K} \subseteq \mathcal{T}_g$ be a compact subset, $\phi_1,\phi_2 \in L^\infty(\tt,\mathbf{m})$ be essentially bounded functions with $\esupp(\phi_1), \esupp(\phi_2) \subseteq \mathcal{K}$, and $\mathcal{B} \subseteq \pmf$ be a cube. Then, for every $R > 0$,
	\begin{gather*}
	\sum_{\mc \in \mcg} \int_{\tt} \phi_1(X) \left( \int_{\mm} \phi_2(\mc.s^{-1}(Y)) \thinspace d\widehat{\mathbf{m}}_{X,\mathcal{B}}^R(Y) \right) d\mathbf{m}(X) \\
	=  \frac{\Lambda_g}{h \cdot \widehat{\mathbf{m}}(\mathcal{M}_g)} \cdot \left(\int_{\qut} \mathbbm{1}_{\mathcal{B}}([\Re(q)]) \thinspace \phi_1(\pi(q)) \thinspace \lambda(q) \thinspace d\mu(q)\right) \cdot \left(\int_{\tt} \phi_2(Y)  \thinspace d\mathbf{m}(Y) \right) \cdot e^{hR} \\
	+ O_{\mathcal{K},\mathrm{DT}}\left( \|\phi_1\|_\infty \cdot \|\phi_2\|_\infty \cdot e^{(h-\useconk{k:sebe})R}\right),
	\end{gather*}
	where $\useconk{k:sebe} = \useconk{k:sebe}(g) > 0$ is a constant depending only on $g$.
\end{theorem}

\section{Effective lattice point count in sectors of Teichmüller space}

\subsection*{Outline of this section.} In this section we state and prove effective lattice point count theorems for sectors of Teichmüller space. As in \S 8, we interpret $\pmf$ as the boundary at infinity of $\mathcal{T}_g$ and consider two cases depending on what type of objects on $\pmf$ are used to cut out the sectors. We first prove an effective lattice point count theorem for sectors of Teichmüller space cut out by cubes of $\pmf$ in Dehn-Thurston coordinates and then, using an approximation argument, deduce an analogous theorem for sectors cut out by piecewise $\mathcal{C}^1$ functions on $\pmf$. These results correspond to Theorems \ref{theo:sector_count_box} and \ref{theo:sector_count_smooth}, respectively. 

\subsection*{Statement of the main theorems.} As in \S 8, we fix a set of Dehn-Thurston coordinates of $\mf$, consider the corresponding identifications $\mf = \mathbf{R}^{6g-6}$ and $\pmf = \mathbf{S}^{6g-7}$, and endow $\mathbf{S}^{6g-7}$ with the restriction of the Riemannian Euclidean metric. We remind the reader that when an implicit constant depends on the choice of Dehn-Thurston coodinates we add the subscript $\mathrm{DT}$. As in \S 8, we consider cubes $\mathcal{B} \subseteq \pmf = \mathbf{S}^{6g-7}$ with closed and/or open facets.

Recall that $q_s \colon \mathcal{T}_g \times \mathcal{T}_g \to \mathcal{Q}^1\mathcal{T}_g$ denotes the map which to every pair $X,Y \in \mathcal{T}_g$ assigns the quadratic differential $q_s(X,Y) \in S(X)$ corresponding to the tangent direction at $X$ of the unique Teichmüller geodesic segment from $X$ to $Y$ and that, for every $X \in \mathcal{T}_g$ and every measurable subset $\mathcal{U} \subseteq \mathcal{PMF}$,
\[
\mathrm{Sect}_\mathcal{U}(X) := \{Y \in \mathcal{T}_g \ | \ [\Re(q_s(X,Y))] \in \mathcal{U}\}.
\]
Let $X,Y \in \tt$ and $\mathcal{U} \subseteq \pmf$ measurable. For every $R > 0$ consider the counting function
\begin{align*}
F_R(X,Y,\mathcal{U}) &:= \# \{\mc \in \mcg \ | \ \mc.Y \in B_R(X) \cap \mathrm{Sect}_{\mathcal{U}}(X)\}\\
&\phantom{:}= \sum_{\mc \in \mcg} \mathbbm{1}_{B_R(X)}(\mathbf{g}.Y) \cdot \mathbbm{1}_{\mathcal{U}}([\Re(q_s(X,\mc.Y))]).
\end{align*}
Recall that $\mathbf{m} := \pi_* \mu$ denotes the pushforward to $\tt$ of the Masur-Veech measure $\mu$ on $\qut$ under the projection $\pi \colon \qut \to \tt$ and that $\widehat{\mathbf{m}}$ denotes the local pushforward of $\mathbf{m}$ to $\mm$. Recall that $h := 6g-6$ and that $\Lambda_g > 0$ denotes the Hubbard-Masur constant introduced in \S 2. The following result generalizes Theorem \ref{theo:count} to sectors of $\tt$ cut out by cubes of $\pmf$.

\newconk{sect_count_box} 
\begin{theorem}
	\label{theo:sector_count_box}
	Let $\mathcal{K} \subseteq \tt$ compact, $X, Y \in \mathcal{K}$, and $\mathcal{B} \subseteq \pmf$ a cube. Then, for every $R > 0$,
	\begin{align*}
	F_R(X,Y,\mathcal{B}) = \frac{\Lambda_g}{h \cdot \widehat{\mathbf{m}}(\mathcal{M}_g)} \cdot \left(\int_{S(X)} \mathbbm{1}_{\mathcal{B}}([\Re(q)]) \thinspace \lambda(q) \thinspace ds_X(q)\right) \cdot e^{hR} 
	+ O_\mathcal{K,\mathrm{DT}}\left(e^{(h-\useconk{sect_count_box})R}\right),
	\end{align*}
	where $\useconk{sect_count_box} = \useconk{sect_count_box}(g) > 0$ is a constant depending only on $g$.
\end{theorem}

Let $X,Y \in \tt$ and $\psi \colon \pmf \to \mathbf{R}_{\geq 0}$ measurable. For every $R > 0$ consider the counting function
\[
F_R(X,Y,\psi) := \sum_{\mc \in \mcg} \mathbbm{1}_{B_R(X)}(\mathbf{g}.Y) \cdot \psi([\Re(q_s(X,\mc.Y))]).
\]
The following result generalizes Theorem \ref{theo:count} to sectors of $\tt$ cut out by functions in $\mathcal{PC}^1(\pmf)$.

\newconk{k:sect_count_smooth} 
\begin{theorem}
	\label{theo:sector_count_smooth}
	Let $\mathcal{K} \subseteq \tt$ compact, $X, Y \in \mathcal{K}$, and $\psi \in \mathcal{PC}^1(\pmf)$ non-negative. Then, for every $R > 0$,
	\begin{align*}
	F_R(X,Y,\psi) = \frac{\Lambda_g}{h \cdot \widehat{\mathbf{m}}(\mathcal{M}_g)} \cdot \left(\int_{S(X)} \psi([\Re(q)]) \thinspace \lambda(q) \thinspace ds_X(q)\right) \cdot e^{hR} 
	+ O_\mathcal{K}\left(\|\psi\|_{\mathcal{PC}^1} \cdot e^{(h-\useconk{k:sect_count_smooth})R}\right),
	\end{align*}
	where $\useconk{k:sect_count_smooth}  = \useconk{k:sect_count_smooth}(g) > 0$ is a constant depending only on $g$.
\end{theorem}

\subsection*{Outline of the proofs of the main theorems.} We first prove Theorem \ref{theo:sector_count_box}  and then deduce Theorem \ref{theo:sector_count_smooth} using approximation arguments. To prove Theorem \ref{theo:sector_count_box} we follow the same general structure of the proof of Theorem \ref{theo:count}. In particular, averaging and unfolding arguments will reduce the proof to an application of Theorem \ref{theo:sector_equidistribution_box_extra}. Comparing the counting functions $F_R(X,Y,\mathcal{B})$ as $X$ and $Y$ vary in small neighborhoods of $\mathcal{T}_g$ will be much more complicated in this case; this is an essential ingredient in the averaging step of the proof. Recall that $d_\mathcal{T}$ denotes the Teichmüller metric on $\tt$. We devote the first half of this section to the proof of the following bound.

\newcond{d:sect_comp_0} \newconk{k:sect_comp_0}  
\begin{proposition}
	\label{prop:sector_comparison_0}
	For every compact subset $\mathcal{K} \subseteq \mathcal{T}_g$ there exists a constant $\usecond{d:sect_comp_0} = \usecond{d:sect_comp_0}(\mathcal{K}) > 0$ with the following property. Let $\mathcal{B} \subseteq \pmf$ a cube, $0 < \delta < \delta_3$, $X,Y \in \mathcal{K}$, and $X',Y' \in \mathcal{T}_g$ with $d_\mathcal{T}(X,X') \leq \usecond{d:sect_comp_0}$ and $d_\mathcal{T}(Y,Y') \leq \delta$. Then, for every $R > 0$,
	\begin{align*}
	F_R(X,Y,\mathcal{B}) - F_{R+2\delta}(X',Y',\mathcal{B}) 
	\preceq_{\mathcal{K},\mathrm{DT}} e^{(h-\useconk{k:sect_comp_0})R},
	\end{align*}
	where $\useconk{k:sect_comp_0}  = \useconk{k:sect_comp_0} (g) > 0$ is a constant depending only on $g$.
\end{proposition}

\subsection*{Varying the endpoint of sector counts.} We begin by studying how the counting functions $F_R(X,Y,\mathcal{\mathcal{B}})$ vary as we move $Y$ in a small neighborhood of $\mathcal{T}_g$. More specifically, we prove the following bound.

\newconr{r:sect_comp_1}  \newconk{k:sect_comp_1}  
\begin{proposition}
	\label{prop:sector_comparison_1}
	There exists a constant $\useconr{r:sect_comp_1} = \useconr{r:sect_comp_1}(g) > 0$ with the following property. Let $\mathcal{K} \subseteq \mathcal{T}_g$ compact, $\mathcal{B} \subseteq \pmf$ a cube, $0 < \delta < \useconr{r:sect_comp_1}$, $X,Y \in \mathcal{K}$, and $Y' \in \mathcal{T}_g$ with $d_\mathcal{T}(Y,Y') < \delta$. Then, for every $R > 0$,
	\begin{align*}
	F_R(X,Y,\mathcal{B}) - F_{R+\delta}(X,Y',\mathcal{B}) 
	\preceq_{\mathcal{K},\mathrm{DT}} e^{(h-\useconk{k:sect_comp_1} )R},
	\end{align*}
	where $\useconk{k:sect_comp_1}  = \useconk{k:sect_comp_1} (g) > 0$ is a constant depending only on $g$.
\end{proposition}

The main techincal tool that will be used in the proof of Proposition \ref{prop:sector_comparison_1} is Theorem \ref{theo:thin_sector_general}. We will also need the following lemma, which is a direct consequence of the proper discontinuity of the action of $\mcg$ on $\tt$. This lemma covers for the some of the orbifold issues that arise when working on $\mathcal{M}_g$. 

\newcond{d:embed}
\begin{lemma}
	\label{lem:embedded_balls}
	Let $\mathcal{K} \subseteq \mathcal{T}_g$ be a compact subset. There exists a constant $\usecond{d:embed} = \usecond{d:embed}(\mathcal{K}) > 0$ such that for every $X \in \mathcal{K}$ and every $0 < \delta < \usecond{d:embed}$,
	\[
	\# \{\mc \in \mcg \ | \ B_{\delta}(\mc.X) \cap B_{\delta}(X) \neq \emptyset\} \preceq_{\mathcal{K}} 1.
	\]
\end{lemma}

Recall that, for every $W \subseteq \mathcal{Q}^1\mathcal{T}_g$ and every $s > 0$, $W(s) \subseteq \mathcal{Q}^1\mathcal{T}_g$ denotes the set of $q_1 \in \mathcal{Q}^1\mathcal{T}_g$ such that there exists $q_2 \in W$ on the same leaf of $\mathcal{F}^{uu}$ as $q_1$ satisfying $d_H(q_1,q_2)<s$. Consider the map $[\Re] \colon \qut \to \pmf$ given by $[\Re](q) = [\Re(q)]$. Recall that $\nu$ denotes the Thurston measure on $\mf$ and that $\overline{\nu}$ denotes the function which to every measurable subset $A \subseteq \mf$ assigns the value $\overline{\nu}(A) = \nu([0,1] \cdot A)$. Recall that, for any $X \in \mathcal{T}_g$ and any $V \subseteq S(X)$,
\[
\mathrm{Sect}_V(X) := \{Y \in \mathcal{T}_g \ | \ q_s(X,Y) \in V\}.
\]
 As a first step towards proving Proposition \ref{prop:sector_comparison_1}, we prove the following bound. 

\newconr{r:sect_prelim} \newconn{n:sect_prelin} \newconk{k:sect_prelim}
\begin{proposition}
	\label{prop:sector_diff_prelim_3}
	There exists a constant $\useconr{r:sect_prelim}  = \useconr{r:sect_prelim} (g) > 0$ with the following property. Let $\mathcal{K} \subseteq \mathcal{T}_g$ compact, $\mathcal{U} \subseteq \pmf$ measurable, $0 < \delta < \useconr{r:sect_prelim} $, $X,Y \in \mathcal{K}$, and $Y' \in \mathcal{T}_g$ with $d_\mathcal{T}(Y,Y') < \delta$. Denote $V := [\Re]^{-1}(\partial \mathcal{U}) \cap S(X)$. Then, for every $R > 0$,
	\begin{align*}
	F_R(X,Y,\mathcal{U}) - F_{R+\delta}(X,Y',\mathcal{U}) 
	\preceq_{\mathcal{K}} \overline{\nu}(\Re(V(\useconn{n:sect_prelin} e^{- \useconk{k:sect_prelim}  R}))) \cdot e^{hR} + e^{(h- \useconk{k:sect_prelim})R},
	\end{align*}
	where $\useconn{n:sect_prelin} = \useconn{n:sect_prelin}(g) > 0$ and $ \useconk{k:sect_prelim} =  \useconk{k:sect_prelim}(g) > 0$ are constants depending only on $g$.
\end{proposition}

\begin{proof}
	Let $\useconr{r:thin_sect_gen} = \useconr{r:thin_sect_gen}(g) > 0$ be as in Theorem \ref{theo:thin_sector_general} and $r = r(g) := \useconr{r:thin_sect_gen}/4 > 0$. Consider $\useconr{r:sect_prelim}  =  \useconr{r:sect_prelim} (g) := r > 0 $. Fix $\mathcal{K} \subseteq \mathcal{T}_g$ compact, $\mathcal{U} \subseteq \pmf$ measurable, $0 < \delta < \useconr{r:sect_prelim} $, $X,Y \in \mathcal{K}$, and $Y' \in \mathcal{T}_g$ with $d_\mathcal{T}(Y,Y') <\delta$. Denote $V :=   [\Re]^{-1}(\partial \mathcal{B}) \cap S(X)$. By the triangle inequality and the $\mcg$ invariance of $d_{\mathcal{T}}$, for every $\mc \in \mcg$ and every $R > 0$, if $\mc.Y \in B_R(X) \cap \mathrm{Sect}_{\mathcal{U}}(X)$ but $\mc.Y' \notin B_{R+\delta}(X) \cap \mathrm{Sect}_{\mathcal{U}}(X)$, then
	\begin{equation}
	\mc.Y' \in \mathrm{Nbhd}_{r}(B_{R+r}(X) \cap \mathrm{Sect}_{V}(X) \cap \mcg \cdot \mathrm{Nbhd}_{r}(\mathcal{K})). \label{eq:pic}
	\end{equation}
	See Figure \ref{fig:bd_0}. Let $\usecond{d:embed} = \usecond{d:embed}(\mathcal{K}) > 0$ be as in Lemma \ref{lem:embedded_balls} and $r' = r'(\mathcal{K}):= \min\{\usecond{d:embed}/2,r\} > 0$. Consider Teichmüller metric balls of radius $r'$ centered at every point $\mathbf{g}.Y'$ satisfying (\ref{eq:pic}). It follows from (\ref{eq:pic}) and Lemma \ref{lem:embedded_balls} that
	\begin{equation}
	F_R(X,Y,\mathcal{U}) - F_{R+\delta}(X,Y',\mathcal{U}) \preceq_{\mathcal{K}} \mathbf{m}(\mathrm{Nbhd}_{2r}(B_{R+r}(X) \cap \text{Sect}_{V}(X) \cap \mathrm{Mod}_g \cdot \mathrm{Nbhd}_{r}(\mathcal{K}))). \label{eq:w1}
	\end{equation}
	As $0 < 2r < \useconr{r:thin_sect_gen}$, Theorem \ref{theo:thin_sector_general} ensures
	\begin{equation}
	\mathbf{m}(\mathrm{Nbhd}_{2r}(B_{R+r}(X) \cap \text{Sect}_{V}(X) \cap \mathrm{Mod}_g \cdot \mathrm{Nbhd}_{r}(\mathcal{K}))) \preceq_{\mathcal{K}} \overline{\nu}(\Re(V(\useconn{n:thin_sect_gen} e^{-\useconk{k:thin_sect_gen} R}))) \cdot e^{hR} + e^{(h-\useconk{k:thin_sect_gen})R}. \label{eq:w2}
	\end{equation}
	Putting together (\ref{eq:w1}) and (\ref{eq:w2}) we conclude
	\[
	F_R(X,Y,\mathcal{U}) - F_{R+\delta}(X,Y',\mathcal{U}) \preceq_{\mathcal{K}} \overline{\nu}(\Re(V(\useconn{n:thin_sect_gen} e^{-\useconk{k:thin_sect_gen} R}))) \cdot e^{hR} + e^{(h-\useconk{k:thin_sect_gen})R}. \qedhere
	\]
\end{proof}

\begin{figure}[h!]
	\centering
	\includegraphics[width=.25\textwidth]{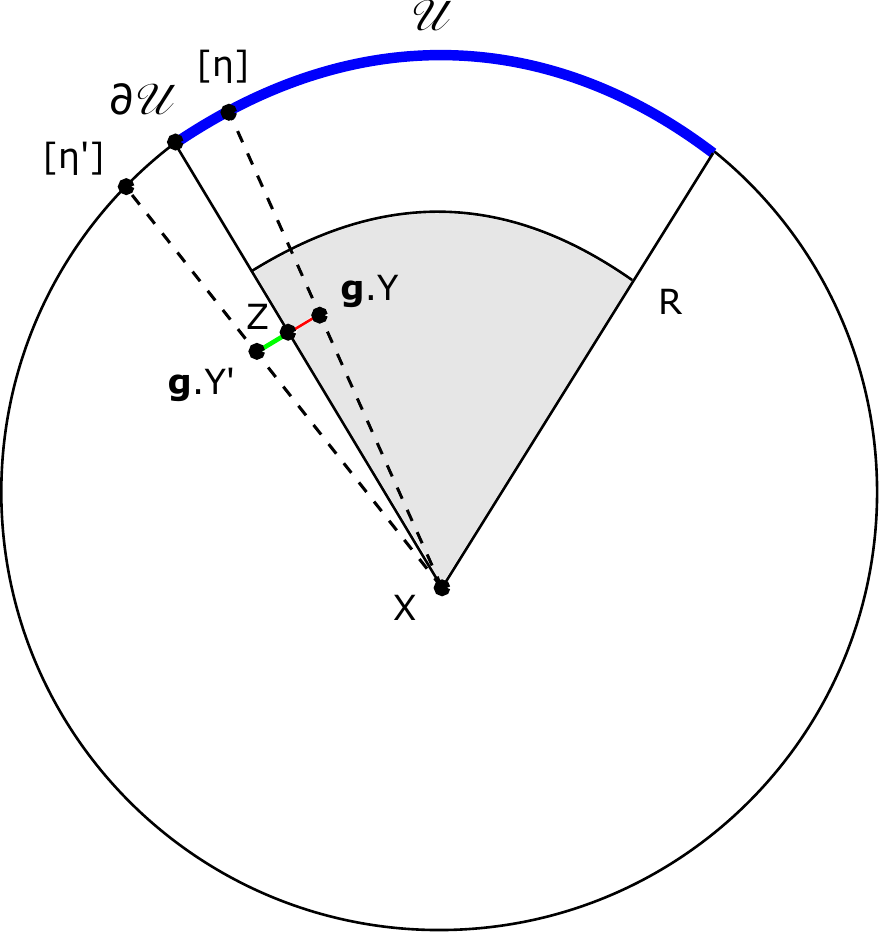}
	\caption{Picture of (\ref{eq:pic}). Straight lines represent Teichmüller geodesics. The green segment represents a piecewise smooth path of Teichmüller length $< r$ joining $\mc.Y'$ to a point $Z \in B_{R+r}(X) \cap \mathrm{Sect}_{V}(X) \cap \mcg \cdot \mathrm{Nbhd}_{r}(\mathcal{K})$.} \label{fig:bd_0}
\end{figure}

Let $X \in \mathcal{T}_g$, $\mathcal{B} \subseteq \pmf$ a cube, and $V := [\Re]^{-1}(\partial \mathcal{B}) \cap S(X)$. To finish the proof of Proposition \ref{prop:sector_comparison_1} we need to estimate $\overline{\nu}(\Re(V(\delta)))$ for $\delta > 0$ sufficiently small. Recall that $\|\cdot \|_E$ denotes the family of fiberwise norms on $T\qut$ inducing the Euclidean metric. The following result is the main tool needed to prove the desired bound on $\overline{\nu}(\Re(V(\delta)))$; the proof uses the notation introduced in \S 8.

\begin{proposition}
	\label{prop:Lipschitz_metric}
	The map $[\Re] \colon \qut(\mathbf{1}) \to \pmf$ is locally Lipschitz in the following sense. 	Let $\mathcal{K} \subseteq \mathcal{T}_g$ be a compact subset and $\mathcal{\tau}$ be a maximal train track on $S_g$. Suppose $q \in \pi^{-1}(\mathcal{K}) \cap \qutp$ and $v \in T_q \qut(\mathbf{1})$ are such that $d \Re_q v \in TU(\tau)$. Then,
	\[
	\| d[\Re]_q v \|_{\tau} \preceq_{\mathcal{K},\tau} \|v\|_E.
	\]	 
\end{proposition}

\begin{proof}
	Fix $\mathcal{K} \subseteq \mathcal{T}_g$ compact and $\tau$ a maximal train track on $S_g$. Let $q \in \pi^{-1}(\mathcal{K}) \cap \qutp$ and $v \in T_q \qut(\mathbf{1})$ be such that $d\Re_q v \in TU(\tau)$. Consider the finite collection of maximal train tracks $\{\tau_i\}_{i=1}^n$ on $S_g$ provided by Lemma \ref{lem:finite_tt}. The explicit construction of these train tracks ensures that $d_q \Re(v) \in TU(\tau_i)$ for some $i \in \{1,\dots,n\}$. Moreover, as the counting measures on the edges of $\tau_i$ correspond to the absolute value of the real part of the holonomy of the edges the corresponding  Delaunay triangulation, Lemma \ref{lem:euc_fund_prop} ensures
	\[
	\| d\Re_q v\|_{\tau_i} \preceq_{\mathcal{K}} \|v\|_E.
	\]	 
	As transition maps between train track coordinate charts are piecewise linear,
	\[
	\| d\Re_q v \|_{\tau} \preceq_{\mathcal{K},\tau} \|v\|_E. 
	\]	
	Recall that, for every $q \in \qut$, $\mathrm{Ext}_{\pi(q)}(\Re(q)) = 1$. This condition provides uniform lower and upper bounds depending only on $\mathcal{K}$ on the $L^1$-norm of points $\Re(q) \in U(\tau) = V(\tau)$ for $q \in\pi^{-1}(\mathcal{K})$. The map which scales points in $V(\tau)$ to $\overline{V}(\tau)$ has derivatives bounded uniformly in terms of $\mathcal{K}$ under these conditions. Projecting a vector on $\overline{V}(\tau)$ to $T\overline{V}(\tau)$ can only reduce its norm. The proposition follows.
\end{proof}

Using Proposition \ref{prop:Lipschitz_metric} we prove the following bound.

\newcond{d:cube_thick}
\begin{proposition}
	\label{prop:cube_thickening}
	For every $\mathcal{K} \subseteq \mathcal{T}_g$ compact there exists a constant $\usecond{d:cube_thick} =  \usecond{d:cube_thick}(\mathcal{K}) > 0$ with the following property. Let $X \in \mathcal{K}$, $\mathcal{B} \subseteq \pmf$ a cube, and $V :=   [\Re]^{-1}(\partial \mathcal{B}) \cap S(X)$. Then, for every $0 < \delta < \usecond{d:cube_thick}$,
	\[
	\overline{\nu}(\Re(V(\delta)) \preceq_{\mathcal{K},\mathrm{DT}} \delta.
	\]
\end{proposition}

\begin{proof}
	Fix $\mathcal{K} \subseteq \mathcal{T}_g$ compact. Let $\usecond{d:hodge_teich} = \usecond{d:hodge_teich}(\mathcal{K}) >0$ be as in Corollary \ref{cor:euclidean_hodge_teichmuller_1}. Fix $X \in \mathcal{K}$, $\mathcal{B} \subseteq \pmf$ a cube, and $0 < \delta < \usecond{d:hodge_teich}$. Denote $V :=   [\Re]^{-1}(\partial \mathcal{B}) \cap S(X)$. By Corollary \ref{cor:euclidean_hodge_teichmuller_1},
	\begin{equation}
	\label{eq:bound_1}
	V(\delta) \subseteq \pi^{-1}(\mathrm{Nbhd}_1(\mathcal{K})).
	\end{equation}   
	Recall that we endow $\pmf := \mathbf{S}^{6g-7}$ with the restriction of the Riemannian Euclidean metric. Denote by $\mathrm{Nbhd}_{\delta}(\partial \mathcal{B}) \subseteq \pmf$ the set of points at distance at most $\delta$ from $\partial \mathcal{B}$. We claim that, for some constant $C = C(\mathcal{K}) > 0$ depending only on $\mathcal{K}$,
	\begin{equation}
	\label{eq:bound_2}
	V(\delta) \subseteq (\qut(\mathbf{1}) \backslash p^{-1}(K_{C \delta}(\mathbf{1}))) \cup [\Re]^{-1}(\mathrm{Nbhd}_{C\delta}(\partial \mathcal{B})).
	\end{equation}
	
	Indeed, let $q_1 \in V$ and $q_2 \in \qut$ on the same leaf of $\mathcal{F}^{uu}$ as $q_1$ with $d_H(q_1,q_2) \leq \delta$. Consider a piecewise smooth path $\rho \colon [0,1] \to \qut$ such that $\rho(0) = q_1$, $\rho(1) = q_2$, and $\ell_E(\rho) \leq 2\delta$, where, we recall, $\ell_E(\rho) > 0$ denotes the length of $\rho$ with respect to the Euclidean metric. Two things can happen: either $\rho$ intersects the multiple zero locus at some point or $\rho$ remains in the principal stratum at all times. In the first case, Proposition \ref{prop:systole_euclidean} ensures $q_2 \in p^{-1}(K_{C'\delta}(\mathbf{1}))$ for some constant $C' = C'(\mathcal{K}) > 0$ depending only on $\mathcal{K}$. In the second case, Proposition \ref{prop:Lipschitz_metric} ensures $q_2 \in [\Re]^{-1}(\mathrm{Nbhd}_{C'' \delta}(\partial \mathcal{B}))$ for some constant $C'' = C''(\mathcal{K}) > 0$ depending only on $\mathcal{K}$. Letting $C = C(\mathcal{K}) := \max \{C',C''\} > 0$ proves the claim.
	
	From (\ref{eq:bound_1}) and (\ref{eq:bound_2}) we deduce
	\begin{equation*}
	\label{eq:bound_3}
	V(\delta) \subseteq (\pi^{-1}(\mathrm{Nbhd}_1(\mathcal{K}))\backslash p^{-1}(K_{C \delta}(\mathbf{1}))) \cup (\pi^{-1}(\mathrm{Nbhd}_1(\mathcal{K})) \cap[\Re]^{-1}(\mathrm{Nbhd}_{C\delta}(\partial \mathcal{B}))).
	\end{equation*}
	In particular,
	\begin{equation}
	\label{eq:bound_4}
	\overline{\nu}(\Re(V(\delta))) \leq \overline{\nu}(\Re(\pi^{-1}(\mathrm{Nbhd}_1(\mathcal{K})) \backslash p^{-1}(K_{C\delta}(\mathbf{1}))) + \overline{\nu}(\Re( \pi^{-1}(\mathrm{Nbhd}_1(\mathcal{K})) \cap [\Re]^{-1}(\mathrm{Nbhd}_{C\delta}(\partial \mathcal{B}))).
	\end{equation}
	By Proposition \ref{prop:small_thu_measure},
	\begin{equation}
	\label{eq:bound_5}
	\overline{\nu}(\Re(\pi^{-1}(\mathrm{Nbhd}_1(\mathcal{K})) \backslash p^{-1}(K_{C\delta}(\mathbf{1}))) \preceq_{\mathcal{K}} \delta.
	\end{equation}
	Denote by $\| \cdot \|$ the Euclidean norm on $\mf = \mathbf{R}^{6g-6}$. Recall that, for every $q \in \qut$, $\mathrm{Ext}_{\pi(q)}(\Re(q)) = 1$. As $\pmf$ is compact and as the function $\mathrm{Ext} \colon \mf \times \tt \to \mathbf{R}_{>0}$ given by $\mathrm{Ext}(\eta,X) := \mathrm{Ext}_\eta(X)$ is positive and continuous, there exists a constant $D = D(\mathcal{K},\mathrm{DT}) > 0$ depending only on $\mathcal{K}$ and the choice of Dehn-Thurston coordinates such that 
	\begin{gather}
	\Re(\pi^{-1}(\mathrm{Nbhd}_1(\mathcal{K})) \cap [\Re]^{-1}(\mathrm{Nbhd}_{C\delta}(\partial \mathcal{B})))
	 \subseteq \{\eta \in \mf \ | \ \|\eta \| \leq D,  \ [\eta] \in \mathrm{Nbhd}_{C\delta}(\partial \mathcal{B}) \}. \label{eq:quick}
	\end{gather}
	As the Thurston measure $\nu$ on $\mf = \mathbf{R}^{6g-6}$ is equal, up to a multiplicative constant, to the Lebesgue measure in Dehn-Thurston coordinates, (\ref{eq:quick}) implies
	\begin{equation}
	\label{eq:bound_6}
	\overline{\nu}(\Re(\pi^{-1}(\mathrm{Nbhd}_1(\mathcal{K})) \cap [\Re]^{-1}(\mathrm{Nbhd}_{C\delta}(\partial \mathcal{B})))) \preceq_{\mathcal{K},\mathrm{DT}} \delta.
	\end{equation}
	Putting together (\ref{eq:bound_4}), (\ref{eq:bound_5}), and (\ref{eq:bound_6}) we conclude
	\[
	\overline{\nu}(\Re(V(\delta)) \preceq_{\mathcal{K},\mathrm{DT}} \delta. \qedhere
	\]
\end{proof}

We are now ready to prove Proposition \ref{prop:sector_comparison_1}.

\begin{proof}[Proof of Proposition \ref{prop:sector_comparison_1}]
	Let $\useconr{r:sect_prelim}  =  \useconr{r:sect_prelim} (g)  > 0$, $\useconn{n:sect_prelin} = \useconn{n:sect_prelin}(g) > 0$, and $\useconk{k:sect_prelim} = \useconk{k:sect_prelim}(g) > 0$ be as in Proposition \ref{prop:sector_diff_prelim_3}. Fix  $\mathcal{K} \subseteq \mathcal{T}_g$ compact.  Let $\usecond{d:cube_thick} = \usecond{d:cube_thick}(\mathcal{K}) > 0$ be as in Proposition \ref{prop:cube_thickening}. Denote $R_0 = R_0(\mathcal{K}) := \max\{(-1/\useconk{k:sect_prelim}) \log(\usecond{d:cube_thick}/\useconn{n:sect_prelin}),0\} \geq 0$ so that $\useconn{n:sect_prelin} e^{-\useconk{k:sect_prelim}R} < \usecond{d:cube_thick}$ for every $R > R_0$. Fix $\mathcal{B} \subseteq \pmf$ a cube, $0 < \delta < \useconr{r:sect_prelim} $, $X,Y \in \mathcal{K}$, and $Y' \in \mathcal{T}_g$ with $d_\mathcal{T}(Y,Y') < \delta$. By Proposition \ref{prop:sector_diff_prelim_3},
	\begin{equation}
	F_R(X,Y,\mathcal{B}) - F_{R+\delta}(X,Y',\mathcal{B}) 
	\preceq_{\mathcal{K}} \overline{\nu}(\Re(V(\useconn{n:sect_prelin} e^{-\useconk{k:sect_prelim}R}))) \cdot e^{hR} + e^{(h-\useconk{k:sect_prelim})R}, \label{eq:h1}
	\end{equation}
	For every $R > R_0$, Proposition \ref{prop:cube_thickening} ensures
	\begin{equation}
	\overline{\nu}(\Re(V(\useconn{n:sect_prelin} e^{-\useconk{k:sect_prelim}R})) \preceq_{\mathcal{K},\mathrm{DT}} e^{-\useconk{k:sect_prelim}R}. \label{eq:h2}
	\end{equation}
	Putting together (\ref{eq:h1}) and (\ref{eq:h2})  we deduce that, for every $R > R_0$,
	\[
	F_R(X,Y,\mathcal{B}) - F_{R+\delta}(X,Y',\mathcal{B}) 
	\preceq_{\mathcal{K}}  e^{(h-\useconk{k:sect_prelim})R}.
	\]
	The same bound holds for every $R > 0$ by increasing the implicit constant.
\end{proof}

\subsection*{Varying the origin of sector counts.} We now study how the counting functions $F_R(X,Y,\mathcal{\mathcal{B}})$ vary as we move $X$ in a small neighborhood of $\mathcal{T}_g$. More specifically, we prove the following bound.

\newcond{d:sect_comp_2} \newconk{k:sect_comp_2}  
\begin{proposition}
	\label{prop:sector_comparison_2}
	For every compact subset $\mathcal{K} \subseteq \mathcal{T}_g$ there exists a constant $\usecond{d:sect_comp_2} = \usecond{d:sect_comp_2}(\mathcal{K}) > 0$ with the following property. Let $\mathcal{B} \subseteq \pmf$ a cube, $0 < \delta < \usecond{d:sect_comp_2}$, $X,Y \in \mathcal{K}$, and $X' \in \mathcal{T}_g$ with $d_\mathcal{T}(X,X') < \delta$. Then, for every $R > 0$,
	\begin{align*}
	F_R(X,Y,\mathcal{B}) - F_{R+\delta}(X',Y,\mathcal{B}) 
	\preceq_{\mathcal{K},\mathrm{DT}} e^{(h-\useconk{k:sect_comp_2})R},
	\end{align*}
	where $\useconk{k:sect_comp_2}  = \useconk{k:sect_comp_2}(g) > 0$ is a constant depending only on $g$.
\end{proposition}

To prove Proposition \ref{prop:sector_comparison_2} we introduce some notation. Consider the map $q_1 \colon \mathcal{T}_g \times \pmf \to \mathcal{Q}^1\mathcal{T}_g$ which to every $X \in \tt$ and every $[\eta] \in \pmf$ assigns the unique quadratic differential $q_1 := q_1(X,[\eta]) \in S(X)$ such that $[\Re(q_1)] = [\eta]$. Consider also the map $t_0 \colon \tt \times \tt \times \pmf \to \mathbf{R}$ which to every pair $X,X' \in \tt$ and every $[\eta] \in \pmf$ assigns the unique $t_0 := t_0(X,X',[\eta]) \in \mathbf{R}$ such that  $a_{t_0} \thinspace q_1(X',[\lambda])$ is in the same leaf of $\mathcal{F}^{ss}$ as $q_1(X,[\lambda])$. Finally, consider the map $q_2 \colon \tt \times \tt \times \pmf \to \mathbf{R}$ given by $q_2(X,X',[\eta]) := a_{t_0(X,X',[\eta])} \thinspace q_1(X',[\eta])$. See Figure \ref{fig:maps}. These maps are continuous. Recall that $d_H$ denotes the modified Hodge metric along the leaves of $\mathcal{F}^{ss}$ and $\mathcal{F}^{uu}$. A compactness argument using Lemma \ref{lem:euclidean_teichmuller_1} and Theorem \ref{theo:euclidean_hodge_1} yields the following result.

\begin{figure}[h!]
	\centering
	\includegraphics[width=.25\textwidth]{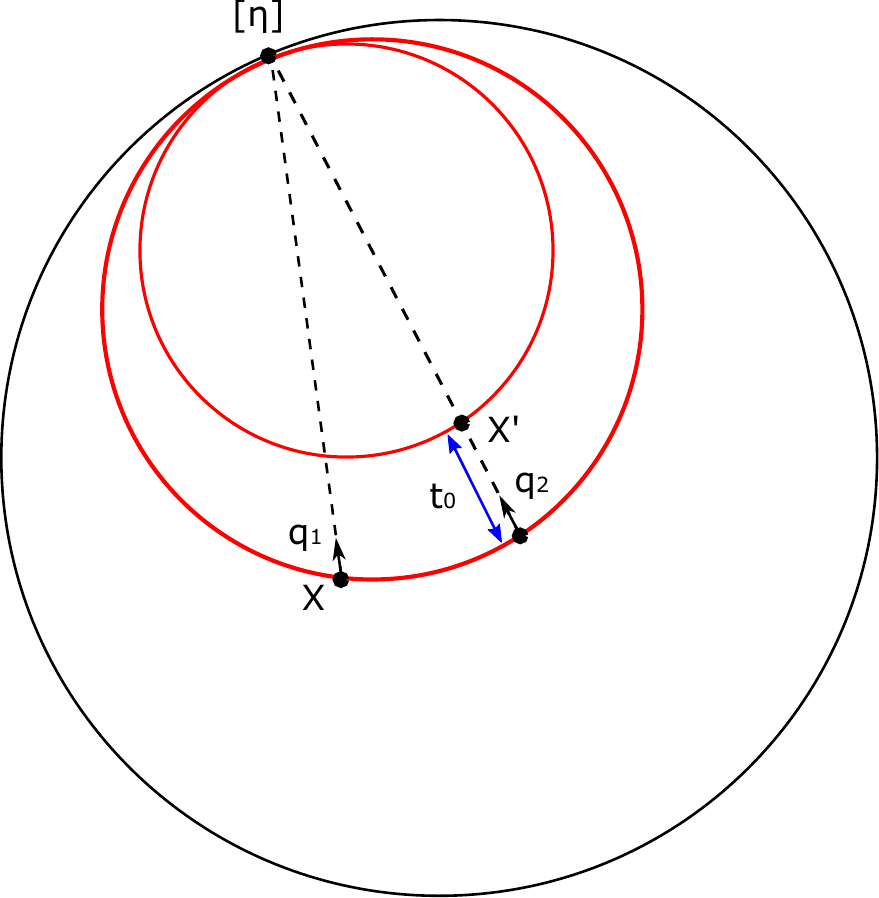}
	\caption{The maps $q_1(X,[\eta])$, $t_0(X,X',[\eta])$, and $q_2(X,X',[\eta])$. Straight lines represent Teichmüller geodesics and circles tangent to the boundary represent leaves of $\mathcal{F}^{ss}$.} \label{fig:maps} 
\end{figure}

\newcond{d:dir_bd} 
\begin{proposition}
	\label{prop:uniform_direction_bounds}
	Let $\mathcal{K} \subseteq \mathcal{T}_g$ be a compact subset and $\epsilon > 0$. There exists a constant $\usecond{d:dir_bd} = \usecond{d:dir_bd}(\mathcal{K},\epsilon) > 0$ such that for every $X\in \mathcal{K}$, every $X' \in \tt$ with $d_\mathcal{T}(X,X') < \usecond{d:dir_bd}$, and every  $[\eta] \in \pmf$, 
	\begin{gather*}
	t_0(X,X',[\eta]) < \epsilon,\\
	d_H(q_1(X,[\eta]), q_2(X,X',[\eta]) < \epsilon.
	\end{gather*}
\end{proposition}

We will also need the following slightly stronger version of Corollary \ref{cor:euclidean_hodge_teichmuller_1}.

\newcond{d:bdd} 
\begin{corollary}
	\label{cor:euclidean_hodge_teichmuller_2}
	Let $\mathcal{K} \subseteq \mathcal{T}_g$ be a compact subset and $r > 0$. There exists a constant $\usecond{d:bdd} = \usecond{d:bdd} (\mathcal{K},r)>0$ such that if $q,q' \in \mathcal{Q}^1\mathcal{T}_g$ are on the same leaf of $\mathcal{F}^{uu}$ or $\mathcal{F}^{ss}$ and satisfy $q \in \pi^{-1}(\mathcal{K})$ and $d_H(q,q') \leq \usecond{d:bdd} $, then $d_\mathcal{T}(\pi(q),\pi(q')) \leq r$. In particular, $q' \in \pi^{-1}(\mathrm{Nbhd}_r(\mathcal{K}))$.
\end{corollary}

As a first step towards proving Proposition \ref{prop:sector_comparison_2}, we prove the following bound.

\newcond{d:sect_dif_4} \newconn{n:sect_dif_4} \newconk{k:sect_dif_4} 
\begin{proposition}
	\label{prop:sector_diff_prelim_4}
 	For every compact subset $\mathcal{K} \subseteq \mathcal{T}_g$ there exists a constant $\usecond{d:sect_dif_4} = \usecond{d:sect_dif_4}(\mathcal{K}) > 0$ with the following property. Let $\mathcal{U} \subseteq \pmf$ measurable, $0 < \delta < \usecond{d:sect_dif_4}$, $X,Y \in \mathcal{K}$, and $X' \in \mathcal{T}_g$ with $d_\mathcal{T}(X,X') \leq \delta$. Denote $V := [\Re]^{-1}(\partial \mathcal{U}) \cap S(X)$. Then, for every $R > 0$,
	\begin{align*}
	F_R(X,Y,\mathcal{U}) - F_{R+\delta}(X',Y,\mathcal{U}) 
	\preceq_{\mathcal{K}} \overline{\nu}(\Re(V(\useconn{n:sect_dif_4}e^{-\useconk{k:sect_dif_4}   R}))) \cdot e^{hR} + e^{(h-\useconk{k:sect_dif_4})R},
	\end{align*}
	where $\useconn{n:sect_dif_4} = \useconn{n:sect_dif_4}(g) > 0$ and $\useconk{k:sect_dif_4}  = \useconk{k:sect_dif_4} (g) > 0$ are constants depending only on $g$.
\end{proposition}

\begin{proof}
	Let $\useconr{r:thin_sect_gen} = \useconr{r:thin_sect_gen}(g) >0$ be as in Theorem \ref{theo:thin_sector_general} and $r = r(g) := \useconr{r:thin_sect_gen}/4 > 0$. Fix $\mathcal{K} \subseteq \mathcal{T}_g$ compact. Let $\usecond{d:bdd}  =\usecond{d:bdd}(\mathcal{K}, r) > 0$ be as in Corollary \ref{cor:euclidean_hodge_teichmuller_2}. By Corollary \ref{cor:euclidean_hodge_teichmuller_2} and the $\mcg$ invariance of $d_H$ and $d_\mathcal{T}$, if $q \in \pi^{-1}(\text{Mod}_g \cdot \mathcal{K})$ and $q' \in \mathcal{Q}^1\mathcal{T}_g$ are on the same leaf of $\mathcal{F}^{ss}$ with $d_H(q,q') < \usecond{d:bdd}$, then 
	\begin{equation}
	d_\mathcal{T}(\pi(q),\pi(q')) < r. \label{eq:small}
	\end{equation}
	Let $0 < \epsilon_{\mathrm{sb}} < \epsilon_{\mathrm{m}}$ be the parameter used to define short bases and $\useconcc{cc:hodge_decay_stable} = \useconcc{cc:hodge_decay_stable}(g,\epsilon_{\mathrm{sb}}) > 0$ be as in Theorem \ref{theo:hodge_distance_decay_stable}. By Theorem \ref{theo:hodge_distance_decay_stable}, for every $q,q' \in \qut$ on the same leaf of $\mathcal{F}^{ss}$ and every $t > 0$,
	\begin{equation}
	\label{eq:non_exp}
	d_H(a_{t} q, a_{t} q') \leq \useconcc{cc:hodge_decay_stable} \cdot d_H(q,q'). 
	\end{equation}
	 Let $\epsilon = \epsilon(\mathcal{K},\epsilon_{\mathrm{sb}}) := \min\{\usecond{d:bdd}/2\useconcc{cc:hodge_decay_stable},r\} > 0$, $\usecond{d:dir_bd} = \usecond{d:dir_bd}(\mathcal{K},\epsilon) > 0$ be as in Proposition \ref{prop:uniform_direction_bounds}, and $\usecond{d:sect_dif_4} = \usecond{d:sect_dif_4}(\mathcal{K},\epsilon)  := \min\{r,\usecond{d:dir_bd}\} > 0$. 
	
	Fix $\mathcal{U} \subseteq \pmf$ measurable, $0 < \delta < \usecond{d:sect_dif_4}$, $X,Y \in \mathcal{K}$, and $X' \in \mathcal{T}_g$ with $d_\mathcal{T}(X,X') < \delta$. Denote $V := [\Re]^{-1}(\partial \mathcal{U}) \cap S(X)$. By the triangle inequality, the $\mcg$ invariance of $d_\mathcal{T}$, (\ref{eq:small}), and (\ref{eq:non_exp}), for every $\mc \in \mcg$ and every $R > 0$, if $\mc.Y \in B_R(X) \cap \mathrm{Sect}_{\mathcal{U}}(X)$ but $\mc.Y \notin B_{R+\delta}(X') \cap \mathrm{Sect}_{\mathcal{U}}(X')$, then
	\begin{equation}
	\label{eq:boundary_point_2}
	\mc.Y \in \text{Nbhd}_{2r}(B_{R+ 2r}(X') \cap \text{Sect}_{V}(X') \cap \text{Mod}_g \cdot \text{Nbhd}_{2r}(\mathcal{K})).
	\end{equation}
	See Figure \ref{fig:boundary_2}. Let $\usecond{d:embed} = \usecond{d:embed}(\mathcal{K}) > 0$ be as in Lemma \ref{lem:embedded_balls} and $r' = r'(\mathcal{K}):= \min\{\usecond{d:embed}/2,r\} > 0$. Consider Teichmüller metric balls of radius $r'$ centered at every point $\mathbf{g}.Y$ satisfying (\ref{eq:boundary_point_2}). It follows from (\ref{eq:boundary_point_2}) and Lemma \ref{lem:embedded_balls} that
	\begin{equation}
	F_R(X,Y,\mathcal{U}) - F_{R+\delta}(X',Y,\mathcal{U}) \preceq_{\mathcal{K}} \mathbf{m}(\mathrm{Nbhd}_{3r}(B_{R+2r}(X') \cap \text{Sect}_{V}(X') \cap \mathrm{Mod}_g \cdot \mathrm{Nbhd}_{2r}(\mathcal{K}))). \label{eq:ff1}
	\end{equation}
	As $0 < 3r < \useconr{r:thin_sect_gen}$, Theorem \ref{theo:thin_sector_general} ensures
	\begin{equation}
	\mathbf{m}(\mathrm{Nbhd}_{3r}(B_{R+2r}(X') \cap \text{Sect}_{V}(X') \cap \mathrm{Mod}_g \cdot \mathrm{Nbhd}_{2r}(\mathcal{K}))) \preceq_{\mathcal{K}} \overline{\nu}(\Re(V(\useconn{n:thin_sect_gen} e^{-\useconk{k:thin_sect_gen} R}))) \cdot e^{hR} + e^{(h-\useconk{k:thin_sect_gen})R}. \label{eq:ff2}
	\end{equation}
	Putting together (\ref{eq:ff1}) and (\ref{eq:ff2}) we conclude 
	\[
	F_R(X,Y,\mathcal{U}) - F_{R+\delta}(X',Y,\mathcal{U}) \preceq_{\mathcal{K}} \overline{\nu}(\Re(V(\useconn{n:thin_sect_gen} e^{-\useconk{k:thin_sect_gen} R}))) \cdot e^{hR} + e^{(h-\useconk{k:thin_sect_gen})R}. \qedhere
	\]
\end{proof}

\begin{figure}[h!]
	\centering
	\begin{subfigure}[b]{0.25\textwidth}
		\centering
		\includegraphics[width=1\textwidth]{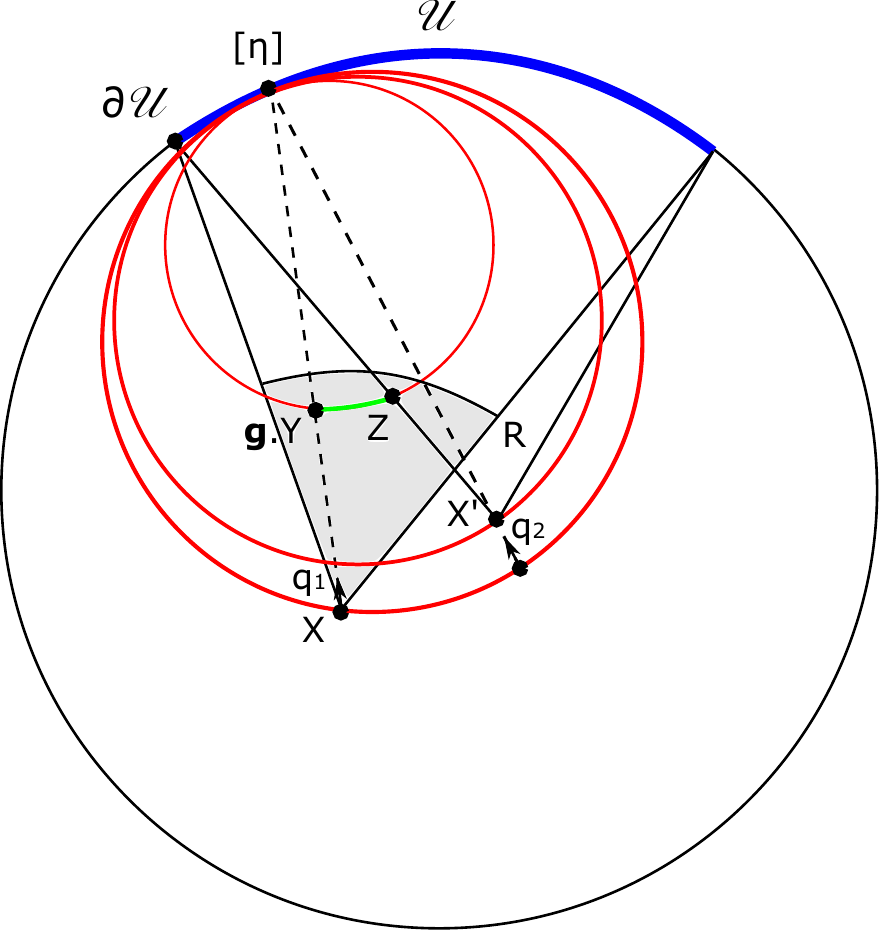}
		\caption{Case 1: $X < X' < \mc.Y$.}
		\label{fig:case_1}
	\end{subfigure}
	\quad \quad \quad
	~ %add desired spacing between images, e. g. ~, \quad, \qquad, \hfill etc. 
	%(or a blank line to force the subfigure onto a new line)
	\begin{subfigure}[b]{0.25\textwidth}
		\centering
		\includegraphics[width=1\textwidth]{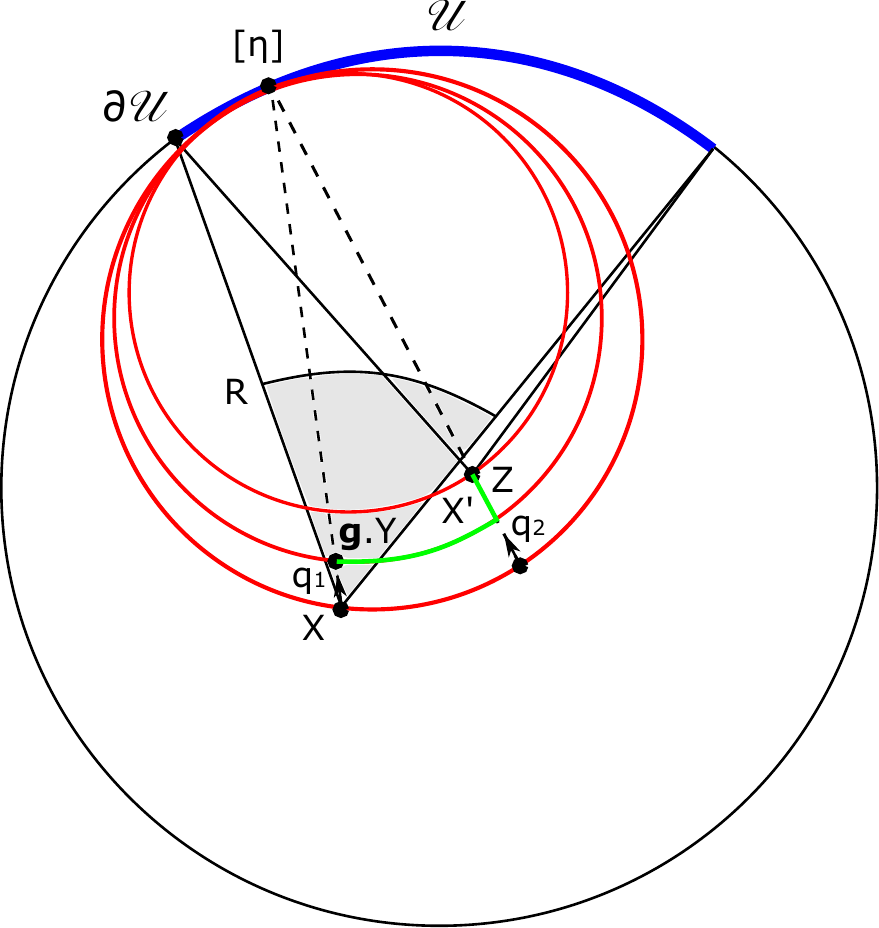}
		\caption{Case 2: $X < \mc.Y < X'$.}
		\label{fig:case_2}
	\end{subfigure}
	\quad \quad \quad
	~ %add desired spacing between images, e. g. ~, \quad, \qquad, \hfill etc. 
	%(or a blank line to force the subfigure onto a new line)
	\begin{subfigure}[b]{0.25\textwidth}
		\centering
		\includegraphics[width=1\textwidth]{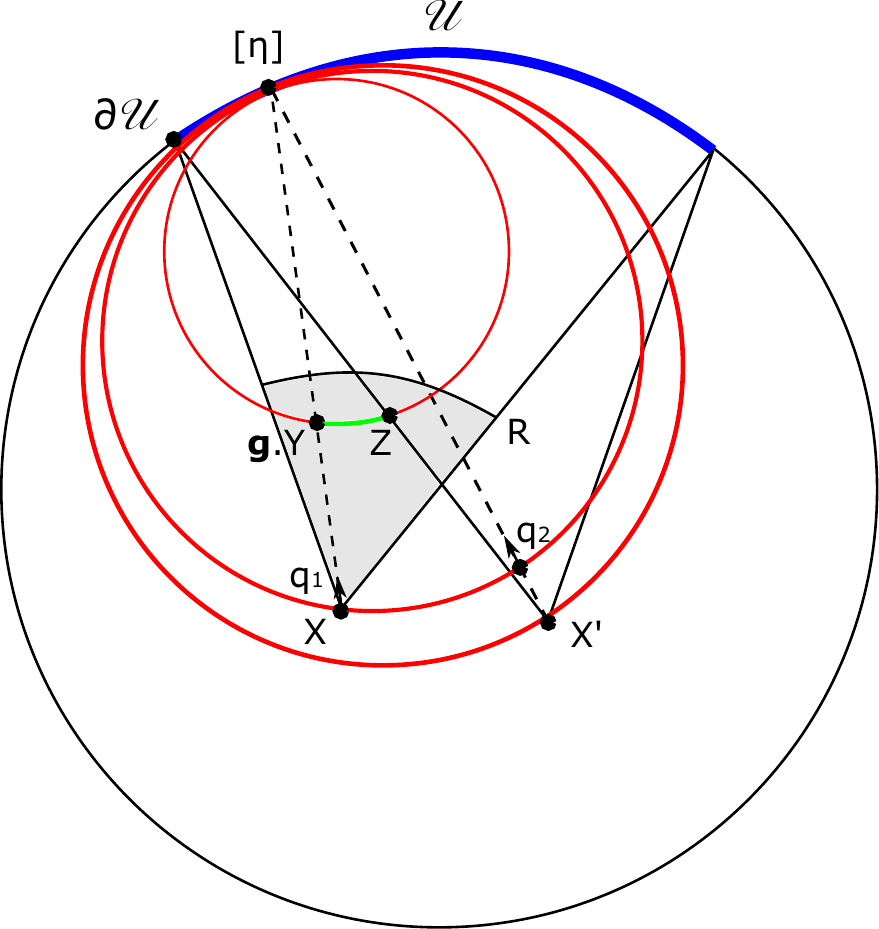}
		\caption{Case 3: $X' < X < \mc.Y$.}
		\label{fig:case_3}
	\end{subfigure}
	\caption{Picture of (\ref{eq:boundary_point_2}). Straight lines represent Teichmüller geodesics and circles tangent to the boundary represent leaves of $\mathcal{F}^{ss}$. Three cases need to be considered depending on the relative order $<$ of the points $X$, $X'$, and $\mc.Y$ with respect to $[\eta] \in \pmf$. The green segments represent piecewise smooth paths of Teichmüller length $< 2r$ joining $\mc.Y$ to a point $Z \in B_{R+ 2r}(X') \cap \text{Sect}_{V}(X') \cap \text{Mod}_g \cdot \text{Nbhd}_{2r}(\mathcal{K})$.} \label{fig:boundary_2} 
\end{figure}

We are now ready to prove Proposition \ref{prop:sector_comparison_2}.

\begin{proof}[Proof of Proposition \ref{prop:sector_comparison_2}]
	Fix $\mathcal{K} \subseteq \mathcal{T}_g$ compact. Let $\usecond{d:cube_thick} = \usecond{d:cube_thick}(\mathcal{K}) > 0$ be as in Proposition \ref{prop:cube_thickening} and $\usecond{d:sect_dif_4} = \usecond{d:sect_dif_4}(\mathcal{K}) > 0$, $\useconn{n:sect_dif_4} = \useconn{n:sect_dif_4}(\mathcal{K}) > 0$, and  $\useconk{k:sect_dif_4} = \useconk{k:sect_dif_4}(\mathcal{K}) > 0$ be as in Proposition \ref{prop:sector_diff_prelim_4}. Denote $R_0 = R_0(\mathcal{K}) := \max \{(-1/\useconk{k:sect_dif_4})\log(\usecond{d:cube_thick}/\useconn{n:sect_dif_4}),0\} \geq 0$ so that $\useconn{n:sect_dif_4} e^{-\useconk{k:sect_dif_4}R} < \usecond{d:cube_thick}$ for every $R > R_0$.  Fix $\mathcal{B} \subseteq \pmf$ a cube, $0 < \delta < \usecond{d:sect_dif_4}$, $X,Y \in \mathcal{K}$, and $X' \in \mathcal{T}_g$ with $d_\mathcal{T}(X,X') < \delta$.  By Proposition \ref{prop:sector_diff_prelim_4},
	\begin{equation}
	F_R(X,Y,\mathcal{B}) - F_{R+\delta}(X',Y,\mathcal{B}) 
	\preceq_{\mathcal{K}} \overline{\nu}(\Re(V(\useconn{n:sect_dif_4}e^{-\useconk{k:sect_dif_4}   R}))) \cdot e^{hR} + e^{(h-\useconk{k:sect_dif_4})R}. \label{eq:ll1}
	\end{equation}
	For every $R > R_0$, Proposition \ref{prop:cube_thickening} ensures
	\begin{equation}
	\overline{\nu}(\Re(V(\useconn{n:sect_dif_4}e^{-\useconk{k:sect_dif_4} R}))) \preceq_{\mathcal{K},\mathrm{DT}} e^{-\useconk{k:sect_dif_4}R}. \label{eq:ll2} 
	\end{equation}
	Putting together (\ref{eq:ll1}) and (\ref{eq:ll2}) we deduce that, for every $R > R_0$,
	\[
	F_R(X,Y,\mathcal{B}) - F_{R+\delta}(X',Y,\mathcal{B})\preceq_{\mathcal{K},\mathrm{DT}} e^{(h-\useconk{k:sect_dif_4})R}.
	\]	
	The same bound holds for every $R > 0$ by increasing the implicit constant.
\end{proof}

Proposition \ref{prop:sector_comparison_0} follows directly from Propositions \ref{prop:sector_comparison_1} and \ref{prop:sector_comparison_2}.

\subsection*{Comparison of leading terms.} For every $X \in \mathcal{T}_g$ and every $\mathcal{U} \subseteq \pmf$ measurable consider the integral
\begin{align*}
I(X,\mathcal{U}) := \int_{S(X)} \mathbbm{1}_{\mathcal{U}}([\Re(q)]) \thinspace \lambda(q) \thinspace ds_X(q).
\end{align*}
The following bounds will be used to compare leading terms in the proof of Theorem \ref{theo:sector_count_box}.

\begin{proposition}
	\label{prop:lead}
	Let $\delta > 0$ and $X,X' \in \mathcal{T}_g$ with $d_\mathcal{T}(X,X') < \delta$. Then, for every $\mathcal{U} \subseteq \pmf$ measurable,
	\[
	e^{-h\delta} \cdot \thinspace I(X,\mathcal{U}) \leq I(X',\mathcal{U}) \leq e^{h\delta} \cdot \thinspace I(X,\mathcal{U}).
	\]
\end{proposition}

\begin{proof}
	By definition (\ref{eq:hm_3}) of the Hubbard-Masur function $\lambda$, for every $X \in \mathcal{T}_g$,
	\begin{equation}
	I(X,\mathcal{U}) = \nu(\{\eta \in \mf \ |  \ \mathrm{Ext}_\eta(X) \leq 1, \thinspace [\eta] \in \mathcal{U} \}). \label{eq:e1}
	\end{equation}
	Fix $\delta > 0$  and $X,X' \in \mathcal{T}_g$ such that $d_{\mathcal{T}}(X,X') < \delta$. By (\ref{eq:ker}), for every $\eta \in \mf$,
	\begin{equation}
	e^{-2\delta} \cdot \mathrm{Ext}_\eta(X) \leq \mathrm{Ext}_\eta(X')  \leq e^{2\delta} \cdot \mathrm{Ext}_\eta(X). \label{eq:e2}
	\end{equation}
	Using (\ref{eq:e1}), (\ref{eq:e2}), and the scaling properties of extremal lengths and the Thurston measure, we bound
	\begin{align*}
	I(X',\mathcal{U}) &=\nu(\{\eta \in \mf \ | \ \mathrm{Ext}_\eta(X') \leq 1, \thinspace [\eta] \in \mathcal{U} \})\\
	&\leq \nu(\{\eta \in \mf \ | \ \mathrm{Ext}_\eta(X) \leq e^{2\delta}, \thinspace [\eta] \in \mathcal{U} \})\\
	&= e^{h\delta} \cdot I(X,\mathcal{U}).
	\end{align*}
	An analogous argument proves the other inequality.
\end{proof}

\subsection*{Proofs of main theorems} We are now ready to prove Theorem \ref{theo:sector_count_box}.

\begin{proof}[Proof of Theorem \ref{theo:sector_count_box}]
	Fix $\mathcal{K} \subseteq \tt$ compact. Let $\usecond{d:vol} = \usecond{d:vol}(\mathcal{K}) > 0$ be as in Lemma \ref{lem:teich_ball_bound},  $\usecond{d:sect_comp_0} = \usecond{d:sect_comp_0}(\mathcal{K}) > 0$ be as in Proposition \ref{prop:sector_comparison_0}, and $\delta_0 = \delta_0(\mathcal{K}) := \min\{\usecond{d:vol}, \usecond{d:sect_comp_0}, 1\} >0$. Fix $X, Y \in \mathcal{K}$ and $\mathcal{B} \subseteq \pmf$ a cube. Let $R > 2$ and $0 < \delta = \delta(\mathcal{K},R) < \delta_0$, to be fixed later. As $0 < \delta < \usecond{d:sect_comp_0}$, Proposition \ref{prop:sector_comparison_0} ensures that, for every $X',Y' \in \tt$ such that $d_\mathcal{T}(X,X') < \delta$ and $d_\mathcal{T}(Y,Y') < \delta$,  
	\begin{gather}
		F_{R-2\delta}(X',Y',\mathcal{B}) - F_R(X,Y,\mathcal{B}) \preceq_{\mathcal{K},\mathrm{DT}} e^{(h-\useconk{k:sect_comp_0} )R}, \label{eq:sector_comp_1}\\
		F_R(X,Y,\mathcal{B}) - F_{R+2\delta}(X',Y',\mathcal{B}) \preceq_{\mathcal{K},\mathrm{DT}} e^{(h-\useconk{k:sect_comp_0} )R}.\label{eq:sector_comp_2}
	\end{gather}
	Multiplying (\ref{eq:sector_comp_1}) and (\ref{eq:sector_comp_2}) by $\mathbbm{1}_{B_\delta(X)}(X') \cdot \mathbbm{1}_{B_\delta(Y)}(Y')$ we obtain the following inequalities, valid for every $X',Y' \in \mathcal{T}_g$,
	\begin{gather}
	\mathbbm{1}_{B_\delta(X)}(X') \cdot \mathbbm{1}_{B_\delta(Y)}(Y') \cdot F_{R-2\delta}(X',Y'\mathcal{B}) - \mathbbm{1}_{B_\delta(X)}(X') \cdot \mathbbm{1}_{B_\delta(Y)}(Y') \cdot F_R(X,Y,\mathcal{B}) \label{eq:sector_comp_3}\\ \preceq_{\mathcal{K},\mathrm{DT}} \mathbbm{1}_{B_\delta(X)}(X') \cdot \mathbbm{1}_{B_\delta(Y)}(Y') \cdot e^{(h-\useconk{k:sect_comp_0} )R},  \nonumber\\
	\mathbbm{1}_{B_\delta(X)}(X') \cdot \mathbbm{1}_{B_\delta(Y)}(Y') \cdot F_R(X,Y,\mathcal{B}) - \mathbbm{1}_{B_\delta(X)}(X') \cdot \mathbbm{1}_{B_\delta(Y)}(Y') \cdot F_{R+2\delta}(X',Y',\mathcal{B}) \label{eq:sector_comp_4}\\ \preceq_{\mathcal{K},\mathrm{DT}} \mathbbm{1}_{B_\delta(X)}(X') \cdot \mathbbm{1}_{B_\delta(Y)}(Y') \cdot e^{(h-\useconk{k:sect_comp_0} )R}. \nonumber
	\end{gather}
	We use (\ref{eq:sector_comp_4}) together with Theorem 	\ref{theo:sector_equidistribution_box_extra} and Lemma \ref{lem:teich_ball_bound} to show that
	\begin{gather}
	F_R(X,Y,\mathcal{B})
	-  \frac{\Lambda_g}{h \cdot \widehat{\mathbf{m}}(\mathcal{M}_g)} \cdot \left( \int_{S(X)} \mathbbm{1}_{\mathcal{B}}([\Re(q)]) \thinspace \lambda(q) \thinspace d\mu(q) \right)  \cdot e^{hR} \label{eq:goal} \\ \preceq_{\mathcal{K},\mathrm{DT}} \delta \cdot e^{hR} + \delta^{-2h} \cdot e^{(h-\useconk{k:sebe})R} + e^{(h-\useconk{k:sect_comp_0})R}. \nonumber
	\end{gather}
	Integrating (\ref{eq:sector_comp_4})  with respect to $d\mathbf{m}(Y') \thinspace d\mathbf{m}(X')$ we deduce
	\begin{gather}
	\mathbf{m}(B_\delta(X))\cdot \mathbf{m}(B_\delta(Y)) \cdot F_R(X,Y,\mathcal{B})	\label{eq:sector_comp_5}\\
	-  \int_{\tt} \mathbbm{1}_{B_\delta(X)}(X') \left(\int_{\tt} \mathbbm{1}_{B_\delta(Y)}(Y') \thinspace F_{R+2\delta}(X',Y',\mathcal{B}) \thinspace d\mathbf{m}(Y') \right) d\mathbf{m}(X') \nonumber \\ \preceq_{\mathcal{K},\mathrm{DT}} \mathbf{m}(B_\delta(X))\cdot \mathbf{m}(B_\delta(Y)) \cdot e^{(h-\useconk{k:sect_comp_0})R} \nonumber.
	\end{gather}
	Fix a measurable fundamental domain $\mathcal{D}_g \subseteq \mathcal{T}_g$ for the action of $\text{Mod}_g$ on $\mathcal{T}_g$. Denote by $s\colon \mathcal{D}_g \to \mathcal{M}_g$ the restriction to $\mathcal{D}_g$ of the quotient map $\up \colon \mathcal{T}_g \to \mathcal{M}_g$. Using this map we can write
	\begin{gather}
	 \int_{\tt} \mathbbm{1}_{B_\delta(X)}(X') \left(\int_{\tt} \mathbbm{1}_{B_\delta(Y)}(Y') \thinspace F_{R+2\delta}(X',Y',\mathcal{B}) \thinspace d\mathbf{m}(Y') \right) d\mathbf{m}(X') \label{eq:progress_1} \\
	 = \sum_{\mc \in \mcg} \int_{\tt} \mathbbm{1}_{B_\delta(X)}(X') \left(\int_{\mm} \mathbbm{1}_{B_\delta(Y)}(\mc.s^{-1}(Y')) \thinspace F_{R+2\delta}(X',\mc.s^{-1}(Y'),\mathcal{B}) \thinspace d\widehat{\mathbf{m}}(Y') \right) d\mathbf{m}(X') \nonumber
	\end{gather}
	Fix $\mc \in \mcg$. An unfolding argument shows that, for every $X' \in \mathcal{M}_g$,
	\begin{gather}
	\int_{\mm} \mathbbm{1}_{B_\delta(Y)}(\mc.s^{-1}(Y')) \thinspace F_{R+2\delta}(X',\mc.s^{-1}(Y'),\mathcal{B}) \thinspace d\widehat{\mathbf{m}}(Y') \label{eq:progress_2}\\
	= \int_{\tt} \sum_{\mathbf{h} \in \mcg} \mathbbm{1}_{\mathcal{D}_g}(\mathbf{h}.Y') \thinspace \mathbbm{1}_{B_\delta(Y)}(\mc.\mathbf{h}.Y') \thinspace \thinspace \mathbbm{1}_{B_R(X') \cap\mathrm{Sect}_\mathcal{B}(X')}(Y')\thinspace d\mathbf{m}(Y') \nonumber\\
	= \int_{\tt} \sum_{\mathbf{h} \in \mcg} \mathbbm{1}_{\mathcal{D}_g}(\mathbf{h}.Y') \thinspace \mathbbm{1}_{B_\delta(Y)}(\mc.\mathbf{h}.Y') \thinspace d\mathbf{m}_{X,\mathcal{B}}^{R+2\delta}(Y') \nonumber.
	\end{gather}
	As $\widehat{\mathbf{m}}_{X,\mathcal{B}}^{R+2\delta}$ is the pushforward to $\mathcal{M}_g$ of the measure $\mathbf{m}_{X,\mathcal{B}}^{R+2\delta}$ on $\tt$,
	\begin{equation}
	\int_{\tt} \sum_{\mathbf{h} \in \mcg} \mathbbm{1}_{\mathcal{D}_g}(\mathbf{h}.Y') \thinspace \mathbbm{1}_{B_\delta(Y)}(\mc.\mathbf{h}.Y') \thinspace d\mathbf{m}_{X,\mathcal{B}}^{R+2\delta}(Y') = \int_{\mathcal{M}_g} \mathbbm{1}_{B_\delta(Y)}(\mathbf{g}. s^{-1}(Y')) \thinspace d \widehat{\mathbf{m}}_{X,\mathcal{B}}^{R+2\delta}. \label{eq:progress_3}
	\end{equation}
	Putting together (\ref{eq:progress_1}), (\ref{eq:progress_2}), and (\ref{eq:progress_3}) we deduce
	\begin{gather}
	\int_{\tt} \mathbbm{1}_{B_\delta(X)}(X') \left(\int_{\tt} \mathbbm{1}_{B_\delta(Y)}(Y') \thinspace F_{R+2\delta}(X',Y',\mathcal{B}) \thinspace d\mathbf{m}(Y') \right) d\mathbf{m}(X') \label{eq:progress_4}  \\
	= \sum_{\mc \in \mcg} \int_{\tt} \mathbbm{1}_{B_\delta(X)}(X') \left(\int_{\mathcal{M}_g} \mathbbm{1}_{B_\delta(Y)}(\mathbf{g}. s^{-1}(Y')) \thinspace d \widehat{\mathbf{m}}_{X,\mathcal{B}}^{R+2\delta}\right) d\mathbf{m}(X') \nonumber.
	\end{gather}
	By Theorem \ref{theo:sector_equidistribution_box_extra},
	\begin{gather}
	\sum_{\mc \in \mcg} \int_{\tt} \mathbbm{1}_{B_\delta(X)}(X') \left(\int_{\mathcal{M}_g} \mathbbm{1}_{B_\delta(Y)}(\mathbf{g}. s^{-1}(Y')) \thinspace d \widehat{\mathbf{m}}_{X,\mathcal{B}}^{R+2\delta}\right) d\mathbf{m}(X') \label{eq:progress_44}\\
	= \frac{\Lambda_g}{h \cdot \widehat{\mathbf{m}}(\mathcal{M}_g)} \cdot \left(\int_{\qut} \mathbbm{1}_{\mathcal{B}}([\Re(q)]) \thinspace \mathbbm{1}_{B_\delta(X)}(\pi(q)) \thinspace \lambda(q) \thinspace d\mu(q)\right) \cdot \left(\int_{\tt} \mathbbm{1}_{B_\delta(Y)}(Y')  \thinspace d\mathbf{m}(Y') \right) \cdot e^{h(R+2\delta)} \nonumber \\
	+ O_{\mathcal{K},\mathrm{DT}}\left(e^{(h-\useconk{k:sebe})R}\right). \nonumber
	\end{gather}
	Using (\ref{eq:fiberwise_measures}) and Proposition \ref{prop:lead} we deduce
	\begin{gather}
	\mathbf{m}(B_\delta(X)) \cdot \thinspace e^{-h\delta} \cdot \left( \int_{S(X)} \mathbbm{1}_{\mathcal{B}}([\Re(q)]) \thinspace \lambda(q) \thinspace d\mu(q) \right)
	\leq \int_{\qut} \mathbbm{1}_{\mathcal{B}}([\Re(q)]) \thinspace \mathbbm{1}_{B_\delta(X)}(\pi(q)) \thinspace \lambda(q) \thinspace d\mu(q), \label{eq:progress_5}\\
	\int_{\qut} \mathbbm{1}_{\mathcal{B}}([\Re(q)]) \thinspace \mathbbm{1}_{B_\delta(X)}(\pi(q)) \thinspace \lambda(q) \thinspace d\mu(q)  \leq \mathbf{m}(B_\delta(X)) \cdot \thinspace e^{h\delta} \cdot \left( \int_{S(X)} \mathbbm{1}_{\mathcal{B}}([\Re(q)]) \thinspace \lambda(q) \thinspace d\mu(q) \right). \label{eq:progress_6}
	\end{gather}
	As $0 < \delta < 1$, the mean value theorem ensures
	\begin{equation}
	\label{eq:exponential_approx}
	e^{h\delta} = 1 + O_g(\delta).
	\end{equation}
	Putting together (\ref{eq:progress_5}), (\ref{eq:progress_6}), and (\ref{eq:exponential_approx}) we deduce
	\begin{gather}
	\int_{\qut} \mathbbm{1}_{\mathcal{B}}([\Re(q)]) \thinspace \mathbbm{1}_{B_\delta(X)}(\pi(q)) \thinspace \lambda(q) \thinspace d\mu(q) \label{eq:progress_7}\\
	= \mathbf{m}(B_\delta(X)) \cdot  \left( \int_{S(X)} \mathbbm{1}_{\mathcal{B}}([\Re(q)]) \thinspace \lambda(q) \thinspace d\mu(q) \right) + O_{g}(\mathbf{m}(B_\delta(X)) \cdot \delta). \nonumber
	\end{gather}
	Combining (\ref{eq:progress_44}), (\ref{eq:exponential_approx}), and (\ref{eq:progress_7})  we deduce
	\begin{gather}
	\sum_{\mc \in \mcg} \int_{\tt} \mathbbm{1}_{B_\delta(X)}(X') \left(\int_{\mathcal{M}_g} \mathbbm{1}_{B_\delta(Y)}(\mathbf{g}. s^{-1}(Y')) \thinspace d \widehat{\mathbf{m}}_{X,\mathcal{B}}^{R+2\delta}\right) d\mathbf{m}(X') \label{eq:progress_8}\\
	= \frac{\Lambda_g}{h \cdot \widehat{\mathbf{m}}(\mathcal{M}_g)} \cdot \mathbf{m}(B_\delta(X)) \cdot \mathbf{m}(B_\delta(Y)) \cdot \left( \int_{S(X)} \mathbbm{1}_{\mathcal{B}}([\Re(q)]) \thinspace \lambda(q) \thinspace d\mu(q) \right)  \cdot e^{hR} \nonumber \\
	+ O_{\mathcal{K},\mathrm{DT}}\left( \mathbf{m}(B_\delta(X)) \cdot \mathbf{m}(B_\delta(Y)) \cdot \delta \cdot e^{hR} + e^{(h-\useconk{k:sebe})R}\right). \nonumber
	\end{gather}
	Putting together (\ref{eq:sector_comp_5}), (\ref{eq:progress_4}), and (\ref{eq:progress_8}), and dividing by $\mathbf{m}(B_\delta(X)) \cdot \mathbf{m}(B_\delta(Y))$ we deduce
	\begin{gather}
	F_R(X,Y,\mathcal{B})
	-  \frac{\Lambda_g}{h \cdot \widehat{\mathbf{m}}(\mathcal{M}_g)} \cdot \left( \int_{S(X)} \mathbbm{1}_{\mathcal{B}}([\Re(q)]) \thinspace\lambda(q) \thinspace d\mu(q) \right)  \cdot e^{hR} \label{eq:almost_0}\\ \preceq_{\mathcal{K},\mathrm{DT}} \delta \cdot e^{hR} + \mathbf{m}(B_\delta(X))^{-1} \cdot \mathbf{m}(B_\delta(Y))^{-1} \cdot e^{(h-\useconk{k:sebe})R} + e^{(h-\useconk{k:sect_comp_0})R}. \nonumber
	\end{gather}
	As $0 < \delta < \usecond{d:vol}$, (\ref{eq:almost_0}) and Lemma \ref{lem:teich_ball_bound} imply (\ref{eq:goal}) holds, that is,
	\begin{gather}
	F_R(X,Y,\mathcal{B})
	-  \frac{\Lambda_g}{h \cdot \widehat{\mathbf{m}}(\mathcal{M}_g)} \cdot \left( \int_{S(X)} \mathbbm{1}_{\mathcal{B}}([\Re(q)]) \thinspace\lambda(q) \thinspace d\mu(q) \right)  \cdot e^{hR} \label{eq:almost_1} \\ \preceq_{\mathcal{K},\mathrm{DT}} \delta \cdot e^{hR} + \delta^{-2h} \cdot e^{(h-\useconk{k:sebe})R} + e^{(h-\useconk{k:sect_comp_0})R}. \nonumber
	\end{gather}
	Analogous arguments using (\ref{eq:sector_comp_3}) instead of (\ref{eq:sector_comp_4}) give the lower bound
	\begin{gather}
	F_R(X,Y,\mathcal{B})
	-  \frac{\Lambda_g}{h \cdot \widehat{\mathbf{m}}(\mathcal{M}_g)} \cdot \left( \int_{S(X)} \mathbbm{1}_{\mathcal{B}}([\Re(q)]) \thinspace \lambda(q) \thinspace d\mu(q) \right)  \cdot e^{hR} \label{eq:almost_2} \\ \succeq_{\mathcal{K},\mathrm{DT}} \delta \cdot e^{hR} + \delta^{-2h} \cdot e^{(h-\useconk{k:sebe})R} + e^{(h-\useconk{k:sect_comp_0})R}. \nonumber
	\end{gather}
	Combining (\ref{eq:almost_1}) with (\ref{eq:almost_2}) we deduce
	\begin{align}
	F_R(X,Y,\mathcal{B})
	&=\frac{\Lambda_g}{h \cdot \widehat{\mathbf{m}}(\mathcal{M}_g)} \cdot \left( \int_{S(X)} \mathbbm{1}_{\mathcal{B}}([\Re(q)]) \thinspace \lambda(q) \thinspace d\mu(q) \right)  \cdot e^{hR} \label{eq:almost_3} \\
	&\phantom{=} +O_{\mathcal{K},\mathrm{DT}}\left(\delta \cdot e^{hR} + \delta^{-2h} \cdot e^{(h-\useconk{k:sebe})R} + e^{(h-\useconk{k:sect_comp_0})R}\right). \nonumber
	\end{align}
	Let $\delta = \delta(\mathcal{K},R) := \delta_0 e^{-\eta R}$ with $\eta > 0$ small enough so that $2h\delta < \useconk{k:sebe}$. Consider $\useconk{sect_count_box} = \useconk{sect_count_box}(g) := \min\{\eta, \useconk{k:sebe}- 2h\eta,\useconk{k:sect_comp_0}\} > 0$. It follows from (\ref{eq:almost_3}) that, for every $R > 2$,
	\begin{align}
	F_R(X,Y,\mathcal{B})
	=\frac{\Lambda_g}{h \cdot \widehat{\mathbf{m}}(\mathcal{M}_g)} \cdot \left( \int_{S(X)} \mathbbm{1}_{\mathcal{B}}([\Re(q)]) \thinspace \lambda(q) \thinspace d\mu(q) \right)  \cdot e^{hR}
	+ O_{\mathcal{K},\mathrm{DT}}\left(e^{(h-\useconk{sect_count_box})R}\right). \nonumber
	\end{align}
	The same equality holds for every $R > 0$ by increasing the implicit constant in the error term.
\end{proof}

We now use Theorem \ref{theo:sector_count_box} and an approximation argument to prove Theorem \ref{theo:sector_count_smooth}.

\begin{proof}[Proof of Theorem \ref{theo:sector_count_smooth}]
	Fix $\mathcal{K} \subseteq \tt$ compact, $X, Y \in \mathcal{K}$, $\psi \in \mathcal{PC}^1(\pmf)$, and $R > 0$. Let $0 < \delta = \delta(R) < 1$, to be fixed later. Consider a partition 
	$
	\pmf = \bigcup_{i=1}^{n(\delta)} \mathcal{B}_i
	$
	of $\pmf = \mathbf{S}^{6g-7}$ into $n(\delta) \preceq_g \delta^{-h}$ disjoint cubes $\mathcal{B}_i \subseteq \pmf$ of diameter $\mathrm{diam}(\mathcal{B}_i) \leq \delta$. For every $i \in \{1,\dots,n(\delta)\}$ denote
	\[
	m_i := \min_{[\eta] \in \mathcal{B}_i} \psi([\eta]), \quad M_i := \max_{[\eta] \in \mathcal{B}_i} \psi([\eta]).
	\]
	As $\psi \in \mathcal{PC}^1(\pmf)$, for every $i \in \{1,\dots,n(\delta)\}$,
	\begin{equation}
	\label{eq:difference}
	|M_i - m_i| \preceq_{\mathrm{DT}} \|\psi\|_{\mathcal{PC}^1} \cdot \delta.
	\end{equation}
	Consider the approximations $\psi_\delta^{\min} \leq \psi \leq \psi_{\delta}^{\max}$ given by
	\[
	\psi^{\min}_\delta([\eta]) := \sum_{i=1}^{n(\delta)} m_i \cdot \mathbbm{1}_{\mathcal{B}_i}([\eta]), \quad \psi^{\max}_\delta([\eta]) := \sum_{i=1}^{n(\delta)}  M_i \cdot \mathbbm{1}_{\mathcal{B}_i}([\eta]).
	\]
	Notice that
	\begin{equation}
	F_R(X,Y,\psi_\delta^{\min}) \leq F_R(X,Y,\psi)  \leq F_R(X,Y,\psi_\delta^{\max}). \label{eq:p_1}
	\end{equation}
	By Theorem \ref{theo:sector_count_box},
	\begin{align}
	F_R(X,Y,\psi_\delta^{\min}) &= \frac{\Lambda_g}{h \cdot \widehat{\mathbf{m}}(\mathcal{M}_g)} \cdot \left(\int_{S(X)} \psi_{\delta}^{\min}([\Re(q)]) \thinspace \lambda(q) \thinspace ds_X(q)\right) \cdot e^{hR} \label{eq:p_2}\\
	&\phantom{=} + O_\mathcal{K,\mathrm{DT}}\left(\|\psi\|_\infty \cdot n(\delta) \cdot e^{(h-\useconk{sect_count_box})R}\right), \nonumber\\
	F_R(X,Y,\psi_\delta^{\max}) &= \frac{\Lambda_g}{h \cdot \widehat{\mathbf{m}}(\mathcal{M}_g)} \cdot \left(\int_{S(X)} \psi_{\delta}^{\max}([\Re(q)]) \thinspace \lambda(q) \thinspace ds_X(q)\right) \cdot e^{hR} \label{eq:p_3} \\
	&\phantom{=} + O_\mathcal{K,\mathrm{DT}}\left(\|\psi\|_\infty \cdot n(\delta) \cdot e^{(h-\useconk{sect_count_box} )R}\right). \nonumber
	\end{align}
	By (\ref{eq:difference}),
	\begin{equation}
	\int_{S(X)} (\psi_{\delta}^{\max}([\Re(q)]) - \psi_{\delta}^{\min}([\Re(q)])) \thinspace \lambda(q) \thinspace ds_X(q) \preceq_{\mathrm{DT}} \|\psi\|_{\mathcal{PC}^1} \cdot \delta. \label{eq:p_4}
	\end{equation}
	Using (\ref{eq:p_1}), (\ref{eq:p_2}), (\ref{eq:p_3}), (\ref{eq:p_4}), and the inequality $n(\delta) \preceq_g \delta^{-h}$ we deduce
	\begin{align}
	F_R(X,Y,\psi) &= \frac{\Lambda_g}{h \cdot \widehat{\mathbf{m}}(\mathcal{M}_g)} \cdot \left(\int_{S(X)} \psi([\Re(q)]) \lambda(q) \thinspace ds_X(q)\right) \cdot e^{hR}  \label{eq:u1}\\
	&\phantom{=} + O_\mathcal{K,\mathrm{DT}}\left(\|\psi\|_{\mathcal{PC}^1} \cdot \left(\delta \cdot e^{hR} + \delta^{-h} \cdot e^{(h-\useconk{sect_count_box})R}\right)\right). \nonumber
	\end{align}
	Let $\delta = \delta(R) := e^{-\eta R}$ with $\eta > 0$ small enough so that $h\eta < \useconk{sect_count_box}$. Consider $\useconk{k:sect_count_smooth} = \useconk{k:sect_count_smooth}(g) := \min\{\eta, \useconk{sect_count_box}- h\eta\} > 0$. From (\ref{eq:u1}) we conclude that, for every $R > 0$,
	\[
	F_R(X,Y,\psi) = \frac{\Lambda_g}{h \cdot \widehat{\mathbf{m}}(\mathcal{M}_g)} \cdot \left(\int_{S(X)} \psi([\Re(q)]) \thinspace \lambda(q) \thinspace ds_X(q)\right) \cdot e^{hR} 
	+ O_\mathcal{K}\left(\|\psi\|_{\mathcal{PC}^1} \cdot e^{(h-\useconk{k:sect_count_smooth})R}\right). \qedhere
	\]
\end{proof}

\section{Effective lattice point count in bisectors of Teichmüller space}

\subsection*{Outline of this section.} The techniques used to prove the main results of \S8 and \S9 can also be used to prove effective mean equistribution theorems for sectors of $\qut$ and effective lattice point count theorems for bisectors of $\mathcal{T}_g$. In this section we state theorems of this kind and briefly outline their proofs. 

\subsection*{Equidistribution of sectors in $\boldsymbol{\qut}$.} Recall that $\mathbf{m} := \pi_* \mu$ denotes the pushforward to $\tt$ of the Masur-Veech measure $\mu$ on $\qut$ under the projection $\pi \colon \qut \to \tt$. Recall that $B_R(X) \subseteq \tt$ denotes the ball of radius $R > 0$ centered at $X \in \tt$ with respect to the Teichmüller metric. Recall that $q_s \colon \mathcal{T}_g \times \mathcal{T}_g \to \mathcal{Q}^1\mathcal{T}_g$ denotes the map which to every pair $X,Y \in \mathcal{T}_g$ assigns the quadratic differential $q_s(X,Y) \in S(X)$ corresponding to the tangent direction at $X$ of the unique Teichmüller geodesic segment from $X$ to $Y$. Fix $X \in \tt$ and $V \subseteq S(X)$ measurable. Recall that
\[
\mathrm{Sect}_V(X) := \{Y \in \tt \ | \ [\Re(q_s(X,Y))] \in V \}.
\]
Recall that $A_X \colon S(X) \times \mathbf{R}_{>0} \to \qut$ denotes the map $A_X(q,t) := a_t q$. Denote by $\mathcal{R}_X \subseteq \qut$ the image of this map. The projection $\pi \colon \qut \to \tt$ restricts to a homeomorphism $\pi|_{\mathcal{R}_X} \colon \mathcal{R}_X \to \tt \backslash \{X\}$. Recall that, for every $R > 0$,  $\mathbf{m}_{X,V}^R$ denotes the restriction of the measure $\mathbf{m}$ to the set $B_R(X) \cap \mathrm{Sect}_V(X) \subseteq \tt$. On $\qut$ consider the measure $\mu_{X,V}^R$ given by
\[
\mu_{X,V}^R := (\pi|_{\mathcal{R}_X})^*(\mathbf{m}_{X,V}^R).
\]
Recall that $\Phi_X \colon S(X) \times \mathbf{R}_{>0} \to \mathcal{T}_g$ denotes the map $\Phi_X(q,t) := \pi(a_tq)$ and that the volume form $\mathbf{m}$ on $\mathcal{T}_g$ can be described in polar coordinates as
\begin{equation*}
%\label{eq:polar_new}
|\Phi_X^*(\mathbf{m})(q,t)| := \Delta(q,t) \cdot |s_X(q) \wedge dt|,
\end{equation*}
where $\Delta \colon S(X) \cap \mathcal{Q}^1\mathcal{T}_g(\mathbf{1}) \times \mathbf{R}_{>0} \to \mathbf{R}_{>0}$ is a positive, smooth function. In terms of this description,
\[
\mu_{X,V}^R = (A_X)_*(\Delta(q,t) \cdot \mathbbm{1}_V(q) \cdot \mathbbm{1}_{(0,R)}(t) \cdot |s_X(q) \wedge dt|).
\]
More generally, given an arbitrary non-negative, measurable function $\varphi \colon \mathcal{Q}^1\mathcal{T}_g \to \mathbf{R}_{\geq 0}$ and $R> 0$, consider the measure $\mu_{X,\varphi}^R$ on $\qut$ given by
\[
\mu_{X,\varphi}^R := (A_X)_*(\Delta(q,t) \cdot \varphi(q) \cdot \mathbbm{1}_{(0,R)}(t) \cdot |s_X(q) \wedge dt|).
\]
Denote by $\widehat{\mu}_{X,\varphi}^R$ the pushforward of $\mu_{X,\varphi}^R$ to $\qum$. The measures $\widehat{\mu}_{X,\varphi}^R$ keep track of how the sector centered at $X$ and cut out by $\varphi$ wraps around $\qum$. 

As in \S 8, we will be particularly interested in the case $\varphi = \psi \circ [\Re]$ for $\psi \colon \pmf \to \mathbf{R}_{\geq 0}$ a measurable function in a suitable class and $[\Re] \colon \qut \to \mf$ the map $[\Re](q) := [\Re(q)]$. Given $X \in \mathcal{T}_g$, $\psi \colon \pmf \to \mathbf{R}_{\geq 0}$ measurable, and $R > 0$, denote $\mu_{X,\psi}^R := \mu_{X,\psi \circ [\Re]}^R$ and $\widehat{\mu}_{X,\psi}^R := \widehat{\mu}_{X,\psi \circ [\Re]}^R$.  

Recall that $\widehat{\mu}$ denotes the Masur-Veech measure on $\qum$ and that $\widehat{\mathbf{m}}$ denotes the local pushforward to $\mm$ of the measure $\mathbf{m}$. Recall that $\mathcal{PC}^1(\pmf)$ denotes the class of piecewise $\mathcal{C}^1$ functions $\psi \colon \pmf \to \mathbf{R}$. Recall the definition of the Ratner class of observables $\mathcal{R}(\qum,\widehat{\mu}) \subseteq L^2(\qum,\widehat{\mu})$ introduced in \S 3. Recall that $\up \colon \tt \to \mm$ denotes the quotient map and that $\underline{\pi} \colon \qum \to \mm$ denote the standard projection. Recall that $h := 6g-6$. The techniques used in the proof of Theorem \ref{theo:sector_equidistribution} can also be used to prove the following result, which generalizes Theorem \ref{theo:ball_equidistribution} to sectors of $\qut$ cut out by functions in $\mathcal{PC}^1(\pmf)$. 

\newconk{k:bisect_count_1}
\begin{theorem}
	\label{theo:bisector_equidistribution}
	Let $\mathcal{K} \subseteq \mathcal{T}_g$ be a compact subset, $\phi_1 \in L^\infty(\tt,\mathbf{m})$ be an essentially bounded function with $\esupp(\phi_1) \subseteq \mathcal{K}$, and $\psi_1  \in \mathcal{PC}^1(\pmf)$ non-negative. Then, for every function $\varphi_2 \in L^\infty(\qum,\widehat{\mu}) \cap \mathcal{R}(\qum,\widehat{\mu})$ with $\esupp(\varphi_2) \subseteq \underline{\pi}^{-1}(\up(\mathcal{K}))$ and every $R > 0$,
	\begin{gather*}
	\int_{\tt} \phi_1(X) \left( \int_{\mm} \varphi_2(q) \thinspace d\widehat{\mu}_{X,\psi_1}^R(q) \right) d\mathbf{m}(X) \\
	=  \frac{1}{h \cdot \widehat{\mathbf{m}}(\mathcal{M}_g)} \cdot \left(\int_{\qut} \psi_1([\Re(q)]) \thinspace \phi_1(\pi(q)) \thinspace \lambda(q) \thinspace d\mu(q)\right) \cdot \left(\int_{\qum} \varphi_2(q) \thinspace \lambda(q)  \thinspace d\widehat{\mu}(q) \right) \cdot e^{hR} \\
	+ O_\mathcal{K}\left( \|\psi_1 \|_{\mathcal{PC}^1} \cdot \|\phi_1\|_\infty \cdot \left(\|\varphi_2\|_\infty + \|\varphi_2\|_{\mathcal{R}(\widehat{\mu})}\right) \cdot e^{(h-\useconk{k:bisect_count_1})R}\right),
	\end{gather*}
	where $\useconk{k:bisect_count_1} = \useconk{k:bisect_count_1}(g) > 0$ is a constant depending only on $g$.
\end{theorem}

As in \S 8, we fix a set of Dehn-Thurston coordinates of $\mf$, consider the corresponding identifications $\mf = \mathbf{R}^{6g-6}$ and $\pmf = \mathbf{S}^{6g-7}$, and endow $\mathbf{S}^{6g-7}$ with the restriction of the Riemannian Euclidean metric. Recall that when an implicit constant depends on the choice of Dehn-Thurston coodinates we add the subscript $\mathrm{DT}$. As in \S 8, we consider cubes $\mathcal{B} \subseteq \pmf = \mathbf{S}^{6g-7}$ with closed and/or open facets. 

Given $X \in \mathcal{T}_g$, $\mathcal{B} \subseteq \pmf$ a cube, and $R > 0$, denote $\mu_{X,\mathcal{B}}^R := \mu_{X,\mathbbm{1}_\mathcal{B}}^R$ and $\widehat{\mu}_{X,\mathcal{B}}^R := \widehat{\mu}_{X,\mathbbm{1}_\mathcal{B}}^R$. The approximation arguments used to prove Theorem \ref{theo:sector_equidistribution_box}, in particular, Proposition \ref{prop:approx}, can be used to deduce the following result from Theorem \ref{theo:bisector_equidistribution}. This result generalizes Theorem \ref{theo:ball_equidistribution} to sectors of $\qut$ cut out by cubes of $\pmf$.

\newconk{k:bisec_equid_2}
\begin{theorem}
	\label{theo:bisector_equidistribution_box}
	Let $\mathcal{K} \subseteq \mathcal{T}_g$ be a compact subset, $\phi_1 \in L^\infty(\tt,\mathbf{m})$ be an essentially bounded function with $\esupp(\phi_1) \subseteq \mathcal{K}$, and $\mathcal{B}_1 \subseteq \pmf$ be a cube. Then, for every  function $\varphi_2 \in L^\infty(\qum,\widehat{\mu}) \cap \mathcal{R}(\qum,\widehat{\mu})$ with $\esupp(\varphi_2) \subseteq \underline{\pi}^{-1}(\up(\mathcal{K}))$ and every $R > 0$,
	\begin{gather*}
	\int_{\tt} \phi_1(X) \left( \int_{\mm} \varphi_2(q) \thinspace d\widehat{\mu}_{X,\mathcal{B}_1}^R(q) \right) d\mathbf{m}(X) \\
	=  \frac{1}{h \cdot \widehat{\mathbf{m}}(\mathcal{M}_g)} \cdot \left(\int_{\qut} \mathbbm{1}_{\mathcal{B}_1}([\Re(q)]) \thinspace \phi_1(\pi(q)) \thinspace \lambda(q) \thinspace d\mu(q)\right) \cdot \left(\int_{\qum} \varphi_2(q) \thinspace \lambda(q)  \thinspace d\widehat{\mu}(q) \right) \cdot e^{hR} \\
	+ O_{\mathcal{K},\mathrm{DT}}\left( \|\phi_1\|_\infty \cdot \left(\|\varphi_2\|_\infty + \|\varphi_2\|_{\mathcal{R}(\widehat{\mu})}\right) \cdot e^{(h-\useconk{k:bisec_equid_2})R}\right),
	\end{gather*}
	where $\useconk{k:bisec_equid_2} = \useconk{k:bisec_equid_2}(g) > 0$ is a constant depending only on $g$.
\end{theorem}

Let $\mathcal{D}_g \subseteq \mathcal{T}_g$ be a measurable fundamental domain for the action of $\mcg$ on $\tt$. Denote by $qs \colon \pi^{-1}(\mathcal{D}_g) \to \qum$ the restriction to $\pi^{-1}(\mathcal{D}_g)$ of the quotient map $p \colon \qut \to \qum$. For $g =2$ this map is invertible and for $g > 2$ this map is invertible away from the fibers of Riemann surfaces with automorphisms. Denote by $qs^{-1} \colon \qum \to \pi^{-1}(\mathcal{D}_g)$ its measurable inverse. Recall the definition of the Ratner class of observables $\mathcal{R}(\qut,\mu) \subseteq L^2(\qut,\mu)$ introduced in \S8. The following result is a direct consequence of Theorem \ref{theo:bisector_equidistribution_box}; compare to Theorem \ref{theo:sector_equidistribution_box_extra}.

\newconk{k:bisect_equid_extra} 
\begin{theorem}
	\label{theo:bisector_equidistribution_box_extra_1}
	Let $\mathcal{K} \subseteq \mathcal{T}_g$ be a compact subset, $\phi_1 \in L^\infty(\tt,\mathbf{m})$ be an essentially bounded function with $\esupp(\phi_1) \subseteq \mathcal{K}$, and $\mathcal{B}_1 \subseteq \pmf$ be a cube. Then, for every  function $\varphi_2 \in L^\infty(\qut,\mu) \cap \mathcal{R}(\qut,\mu)$ with $\esupp(\varphi_2) \subseteq \pi^{-1}(\mathcal{K})$ and every $R > 0$,
	\begin{gather*}
	\sum_{\mc \in \mcg} \int_{\tt} \phi_1(X) \left( \int_{\mm} \varphi_2(\mc.\mathrm{qs}^{-1}(q)) \thinspace d\widehat{\mu}_{X,\mathcal{B}_1}^R(q) \right) d\mathbf{m}(X) \\
	=  \frac{1}{h \cdot \widehat{\mathbf{m}}(\mathcal{M}_g)} \cdot \left(\int_{\qut} \mathbbm{1}_{\mathcal{B}_1}([\Re(q)]) \thinspace \phi_1(\pi(q)) \thinspace \lambda(q) \thinspace d\mu(q)\right) \cdot \left(\int_{\qum} \varphi_2(q) \thinspace\lambda(q)  \thinspace d\widehat{\mu}(q) \right) \cdot e^{hR} \\
	+ O_\mathcal{\mathcal{K},\mathrm{DT}}\left( \|\phi_1\|_\infty \cdot \left(\| \varphi_2\|_\infty + \|\varphi_2\|_{\mathcal{R}(\mu)} \right) \cdot e^{(h-\useconk{k:bisect_equid_extra})R}\right),
	\end{gather*}
	where $\useconk{k:bisect_equid_extra} =\useconk{k:bisect_equid_extra}(g) > 0$ is a constant depending only on $g$.
\end{theorem}

The following result is a direct consequence of Theorem \ref{theo:bisector_equidistribution_box_extra_1} and Proposition \ref{prop:ratner_observable_new}.

\newconk{k:bisect_equid_extra_x} 
\begin{theorem}
	\label{theo:bisector_equidistribution_box_extra_x}
	Let $\mathcal{K} \subseteq \mathcal{T}_g$ be a compact subset, $\phi_1,\phi_2 \in L^\infty(\tt,\mathbf{m})$ be essentially bounded functions with $\esupp(\phi_1), \esupp(\phi_2) \subseteq \mathcal{K}$, and $\mathcal{B}_1 \subseteq \pmf$ be a cube. Then, for every function $\psi_2 \in \mathcal{PC}^1(\pmf)$ and every $R > 0$,
	\begin{gather*}
	\sum_{\mc \in \mcg} \int_{\tt} \phi_1(X) \left( \int_{\mm} \psi_2(\Re[\mc.\mathrm{qs}^{-1}(q)]) \thinspace \phi_2(\pi(\mc.\mathrm{qs}^{-1}(q))) \thinspace d\widehat{\mu}_{X,\mathcal{B}_1}^R(q) \right) d\mathbf{m}(X) \\
	=  \frac{1}{h \cdot \widehat{\mathbf{m}}(\mathcal{M}_g)} \cdot \left(\int_{\qut} \mathbbm{1}_{\mathcal{B}_1}([\Re(q)]) \thinspace \phi_1(\pi(q)) \thinspace \lambda(q) \thinspace d\mu(q)\right) \cdot \left(\int_{\qum} \psi_2([\Re(q)]) \thinspace\lambda(q)  \thinspace \phi_2(\pi(q)) \thinspace d\widehat{\mu}(q) \right) \cdot e^{hR} \\
	+ O_\mathcal{\mathcal{K},\mathrm{DT}}\left( \|\phi_1\|_\infty \cdot \|\phi_2\|_\infty \cdot \| \psi_2 \|_{\mathcal{PC}^1} \cdot e^{(h-\useconk{k:bisect_equid_extra_x})R}\right),
	\end{gather*}
	where $\useconk{k:bisect_equid_extra_x} =\useconk{k:bisect_equid_extra_x}(g) > 0$ is a constant depending only on $g$.
\end{theorem}

The following result is a direct consequence of Theorem \ref{theo:bisector_equidistribution_box_extra_x} and Proposition \ref{prop:approx}.

\newconk{k:bisec_box_extra_2}
\begin{theorem}
	\label{theo:bisector_equidistribution_box_extra_2}
	Let $\mathcal{K} \subseteq \mathcal{T}_g$ be a compact subset, $\phi_1,\phi_2 \in L^\infty(\tt,\mathbf{m})$ be essentially bounded functions with $\esupp(\phi_1), \esupp(\phi_2) \subseteq \mathcal{K}$, and $\mathcal{B}_1,\mathcal{B}_2 \subseteq \pmf$ be cubes. Then, for every $R > 0$,
	\begin{gather*}
	\sum_{\mc \in \mcg} \int_{\tt} \phi_1(X) \left( \int_{\mm} \mathbbm{1}_{\mathcal{B}_2}(\Re[\mc.\mathrm{qs}^{-1}(q)]) \thinspace \phi_2(\pi(\mc.\mathrm{qs}^{-1}(q)))\thinspace d\widehat{\mu}_{X,\mathcal{B}_1}^R(q) \right) d\mathbf{m}(X) \\
	=  \frac{1}{h \cdot \widehat{\mathbf{m}}(\mathcal{M}_g)} \cdot \left(\int_{\qut} \mathbbm{1}_{\mathcal{B}_1}([\Re(q)]) \thinspace \phi_1(\pi(q)) \thinspace \lambda(q) \thinspace d\mu(q)\right) \cdot \left(\int_{\qum} \mathbbm{1}_{\mathcal{B}_2}([\Re(q)]) \thinspace \phi_2(\pi(q)) \thinspace \lambda(q) \thinspace d\widehat{\mu}(q) \right) \cdot e^{hR} \\
	+ O_\mathcal{K}\left( \|\phi_1\|_\infty \cdot \|\phi_2\|_{\infty} \cdot e^{(h-\useconk{k:bisec_box_extra_2})R}\right),
	\end{gather*}
	where $\useconk{k:bisec_box_extra_2} = \useconk{k:bisec_box_extra_2}(g) > 0$ is a constant depending only on $g$.
\end{theorem}

\subsection*{Effective lattice point count in bisectors of $\boldsymbol{\mathcal{T}_g}$.} Recall that, for every $X \in \mathcal{T}_g$ and every $\mathcal{U} \subseteq \mathcal{PMF}$,
\[
\mathrm{Sect}_\mathcal{U}(X) := \{Y \in \mathcal{T}_g \ | \ [\Re(q_s(X,Y))] \in \mathcal{U}\}.
\]
Let $X,Y \in \tt$ and $\mathcal{B}_1,\mathcal{B}_2 \subseteq \pmf $ cubes. For every $R > 0$ consider the counting function
\begin{align*}
	F_R(X,Y,\mathcal{B}_1,\mathcal{B}_2) &:= \# \{\mc \in \mcg \ | \ \mc.Y \in B_R(X) \cap \mathrm{Sect}_{\mathcal{B}_1}(X), \thinspace \mc^{-1}.X \in \mathrm{Sect}_{\mathcal{B}_2}(Y)\}\\
	&\phantom{:}= \sum_{\mc \in \mcg} \mathbbm{1}_{B_R(X)}(\mathbf{g}.Y) \cdot \mathbbm{1}_{\mathcal{B}_1}([\Re(q_s(X,\mc.Y))]) \cdot \mathbbm{1}_{\mathcal{B}_2}([\Re(Y,\mc^{-1}.X)]).
\end{align*}
Following the outline of the proof of Theorem \ref{theo:sector_count_box}, one can use Theorem  \ref{theo:bisector_equidistribution_box_extra_2} and Proposition \ref{prop:sector_comparison_0} to prove the following result, which generalizes Theorem \ref{theo:count} to bisectors of $\mathcal{T}_g$ cut out by cubes of $\pmf$. 

\newconk{k:count_bisect_box}
\begin{theorem}
	\label{theo:bisector_count_box}
	Let $\mathcal{K} \subseteq \tt$ compact, $X, Y \in \mathcal{K}$, and $\mathcal{B}_1,\mathcal{B}_2 \subseteq \pmf$ cubes. Then, for every $R > 0$,
	\begin{gather*}
		F_R(X,Y,\mathcal{B}_1,\mathcal{B}_2) \\
		= \frac{1}{h \cdot \widehat{\mathbf{m}}(\mathcal{M}_g)} \cdot \left(\int_{S(X)} \mathbbm{1}_{\mathcal{B}_1}([\Re(q)]) \thinspace \lambda(q) \thinspace ds_X(q)\right) \cdot \left(\int_{S(Y)} \mathbbm{1}_{\mathcal{B}_2}([\Re(q)]) \thinspace \lambda(q) \thinspace ds_Y(q)\right) \cdot e^{hR}  \\
		+ O_\mathcal{K,\mathrm{DT}}\left(e^{(h-\useconk{k:count_bisect_box})R}\right),
	\end{gather*}
	where $\useconk{k:count_bisect_box} = \useconk{k:count_bisect_box}(g) > 0$ is a constant depending only on $g$.
\end{theorem}

Let $X,Y \in \tt$ and $\psi_1,\psi_2 \in \mathcal{PC}^1(\pmf)$. For every $R > 0$ consider the counting function
\[
F_R(X,Y,\psi_1,\psi_2) := \sum_{\mc \in \mcg} \mathbbm{1}_{B_R(X)}(\mathbf{g}.Y) \cdot \psi_1([\Re(q_s(X,\mc.Y))]) \cdot \psi_2([\Re(q_s(Y,\mc^{-1}.X))]).
\]
The approximation arguments used to prove Theorem \ref{theo:bisector_count_smooth} can be used to deduce the following result from Theorem \ref{theo:bisector_count_box}. This result generalizes Theorem \ref{theo:count} to bisectors of $\tt$ cut out by functions in $\mathcal{PC}^1(\pmf)$.

\newconk{k:bisect_count_smooth} 
\begin{theorem}
	\label{theo:bisector_count_smooth}
	Let $\mathcal{K} \subseteq \tt$ compact, $X, Y \in \mathcal{K}$, and $\psi_1,\psi_2 \in \mathcal{PC}^1(\pmf)$. Then, for every $R > 0$,
	\begin{gather*}
		F_R(X,Y,\psi_1,\psi_2) \\
		= \frac{1}{h \cdot \widehat{\mathbf{m}}(\mathcal{M}_g)} \cdot \left(\int_{S(X)} \psi_1([\Re(q)]) \thinspace \lambda(q) \thinspace ds_X(q)\right) \cdot \left(\int_{S(Y)} \psi_2([\Re(q)]) \thinspace \lambda(q) \thinspace ds_Y(q)\right) \cdot e^{hR} \\
		+ O_\mathcal{K}\left(\|\psi_1\|_{\mathcal{PC}^1} \cdot \|\psi_2\|_{\mathcal{PC}^1} \cdot e^{(h-\useconk{k:bisect_count_smooth})R}\right),
	\end{gather*}
	where $\useconk{k:bisect_count_smooth} = \useconk{k:bisect_count_smooth}(g) > 0$ is a constant depending only on $g$.
\end{theorem}

%    Bibliographies can be prepared with BibTeX using amsplain,
%    amsalpha, or (for "historical" overviews) natbib style.

\bibliographystyle{amsalpha}

%    Insert the bibliography data here.

\bibliography{bibliography}

\end{document}